\documentclass[11pt,oneside]{amsart}

\usepackage{amsmath,amsthm,amssymb,amsfonts,amscd}
\usepackage[dvips]{graphicx}
\usepackage{latexsym}
\usepackage{epsfig,graphicx,psfrag,xcolor}
\usepackage[all]{xy}
\usepackage{url}
\usepackage{fancyhdr}
\usepackage{enumerate}
\usepackage{mathrsfs}
\usepackage{caption}
\usepackage{subcaption}

\newcounter{notes}%

\renewcommand{\phi}{\varphi}

\newcommand{\bv}{\left(\begin{array}{c}}
\newcommand{\ev}{\end{array}\right)}
\newcommand{\bmat}{\begin{pmatrix}} 
\newcommand{\emat}{\end{pmatrix}}
\newcommand{\lip}{\mathrm{lip}}
\newcommand{\Lip}{\mathrm{Lip}}

\newcommand{\AdS}{\mathrm{AdS}^3}
\newcommand{\AdSS}{\mathrm{AdS}}
\newcommand{\Mink}{\mathrm{Min}}

\newcommand{\SO}{\mathrm{SO}}
\newcommand{\OO}{\mathrm{O}}
\newcommand{\PO}{\mathrm{PO}}

\newcommand{\PSL}{\mathrm{PSL}}
\newcommand{\PGL}{\mathrm{PGL}}

\newcommand{\g}{\mathfrak{g}}

\newcommand{\ssl}{\mathfrak{sl}}
\newcommand{\RR}{\mathbb{R}}

\newcommand{\A}{\mathcal{A}}
\newcommand{\RP}{\mathbb{RP}}
\newcommand{\ZZ}{\mathbb{Z}}
\newcommand{\NN}{\mathbb{N}}
\newcommand{\HH}{\mathbb{H}}
\newcommand{\Ad}{\operatorname{Ad}}
\newcommand{\ad}{\operatorname{ad}}

\newcommand{\D}{\mathrm{d}}
\newcommand{\Hom}{\mathrm{Hom}}
\newcommand{\ie}{\emph{i.e.}\ }
\newcommand{\eg}{\emph{e.g.}\ }
\newcommand{\eps}{\varepsilon}

\newcommand{\CH}{\operatorname{Conv}}

\newcommand{\Dev}{\operatorname{Dev}}
\newcommand{\dev}{\operatorname{dev}}
\newcommand{\hatDev}{\widehat{\operatorname{Dev}}}
\newcommand{\hatdev}{\widehat{\operatorname{dev}}}
\newcommand{\specialC}{\mathscr{C}}
\newcommand{\gmetric}{\varrho}
\newcommand{\hmetric}{\varsigma}
\newcommand{\newf}{\nu}
\newcommand{\smooth}{s}
\newcommand*{\longhookrightarrow}{\ensuremath{\lhook\joinrel\relbar\joinrel\rightarrow}}

\theoremstyle{plain}
\newtheorem{Theorem}{Theorem}
\newtheorem{Proposition}[Theorem]{Proposition}

\newtheorem{Corollary}[Theorem]{Corollary}
\newtheorem{Observation}[Theorem]{Observation}
\newtheorem{Claim}[Theorem]{Claim}

\newtheorem{Lemma}[Theorem]{Lemma}

\newtheoremstyle{Theoremwithref}{}{}{\itshape}{}{\bfseries}{.}{.5em}{#1 #2 #3}
\theoremstyle{Theoremwithref}
\newtheorem{Theoremwithref}[Theorem]{Theorem}
\newtheorem{Propositionwithref}[Theorem]{Proposition}

\theoremstyle{definition}
\newtheorem{Definition}[Theorem]{Definition}

\newtheorem{Remark}[Theorem]{Remark}

\numberwithin{Theorem}{section}
\numberwithin{equation}{section}

\usepackage{amsthm}

\makeatletter
\newtheorem*{rep@theorem}{\rep@title}
\newcommand{\newreptheorem}[2]{%
\newenvironment{rep#1}[1]{%
 \def\rep@title{#2~\ref{##1}}%
 \begin{rep@theorem}}%
 {\end{rep@theorem}}}
\makeatother

\newreptheorem{Theorem}{Theorem}

\newreptheorem{lemma}{Lemma}

\title[Complete Lorentz spacetimes of constant curvature]{Geometry and topology of complete Lorentz spacetimes of constant curvature}

\author{Jeffrey Danciger}
\address{Department of Mathematics, The University of Texas at Austin, 1 University Station C1200, Austin, TX 78712, USA}
\email{jdanciger@math.utexas.edu}

\author{Fran\c{c}ois Gu\'eritaud}
\address{CNRS and Universit\'e Lille 1, Laboratoire Paul Painlev\'e, 59655 Villeneuve d'Ascq Cedex, France}
\email{francois.gueritaud@math.univ-lille1.fr}

\author{Fanny Kassel}
\address{CNRS and Universit\'e Lille 1, Laboratoire Paul Painlev\'e, 59655 Villeneuve d'Ascq Cedex, France}
\email{fanny.kassel@math.univ-lille1.fr}

\thanks{J.D. is partially supported by the National Science Foundation under the grant DMS 1103939.
F.G. and F.K. are partially supported by the Agence Nationale de la Recherche under the grants DiscGroup (ANR-11-BS01-013) and ETTT (ANR-09-BLAN-0116-01), and through the Labex CEMPI (ANR-11-LABX-0007-01)}

\begin{document}

\begin{abstract}
We study proper, isometric actions of nonsolvable discrete groups~$\Gamma$ on the $3$-dimensional Minkowski space $\RR^{2,1}$ as limits of actions on the $3$-dimensional anti-de Sitter space $\AdS$. 
To each such action is associated a deformation of a hyperbolic surface group $\Gamma_0$ inside $\OO(2,1)$.
When $\Gamma_0$ is convex cocompact, we prove that $\Gamma$ acts properly on~$\RR^{2,1}$ if and only if this group-level deformation is realized by a deformation of the quotient surface that everywhere contracts distances at a uniform rate.
We give two applications in this case.
(1) Tameness: A complete flat spacetime is homeomorphic to the interior of a compact manifold.
(2) Geometric transition: A complete flat spacetime is the rescaled limit of collapsing $\AdSS$ spacetimes.
\end{abstract}

\maketitle

\section{Introduction}

A Lorentzian 3-manifold of constant negative curvature is locally modeled on the \emph{anti-de Sitter} space $\AdS = \PO(2,2)/\OO(2,1)$, which can be realized in $\RP^3$ as the set of negative points with respect to a quadratic form of signature $(2,2)$.
A flat Lorentzian 3-manifold is locally modeled on the \emph{Minkowski} space~$\RR^{2,1}$, which is the affine space $\RR^3$ endowed with the Lorentzian structure induced by a quadratic form of signature $(2,1)$.
Observe that the tangent space at a point of $\AdS$ identifies with $\RR^{2,1}$; this basic fact motivates the point of view of this paper that a large class of manifolds modeled on $\RR^{2,1}$ (convex cocompact Margulis spacetimes) are infinitesimal versions of manifolds modeled on $\AdS$.
We consider only \emph{complete} Lorentzian manifolds which are quotients of $\AdS$ or $\RR^{2,1}$ by discrete groups $\Gamma$ of isometries acting properly discontinuously.

The following facts, specific to dimension~$3$, will be used throughout the paper.
The anti-de Sitter space $\AdS$ identifies with the manifold $G=\nolinebreak\PSL_2(\RR)$ endowed with the Lorentzian metric induced by (a multiple of) the Killing form.
The group of orientation and time-orientation preserving isometries is $G\times G$ acting by right and left multiplication: $(g_1,g_2) \cdot g = g_2 g g_1^{-1}$.
The Minkowski space $\RR^{2,1}$ can be realized as the Lie algebra $\g = \ssl_2(\RR)$.
The group of orientation and time-orientation preserving isometries is $G \ltimes \g$ acting affinely: $(g,v) \cdot w = \Ad(g) w +v$.

Examples of groups of isometries acting properly discontinuously on $\AdS$ are easy to construct: one can take $\Gamma =\Gamma_0\times\{ 1\}$ where $\Gamma_0$ is any discrete subgroup of~$G$; in this case the quotient $\Gamma\backslash\AdS$ identifies with the unit tangent bundle to the hyperbolic orbifold $\Gamma_0\backslash\HH^2$.
Such quotients are called standard.
Goldman \cite{gol85} produced the first nonstandard examples by deforming standard ones, a technique that was later generalized by Kobayashi \cite{kob98}.
Salein \cite{sal00} then constructed the first examples that were not deformations of standard ones.

On the other hand, although cyclic examples are readily constructed, it is not obvious that there exist \emph{nonsolvable} groups acting properly discontinuously on~$\RR^{2,1}$.
The Auslander conjecture in dimension~$3$, proved by Fried--Goldman \cite{fg83}, states that any discrete group acting properly discontinuously \emph{and cocompactly} on~$\RR^{2,1}$ is solvable up to finite index, generalizing Bieberbach's theory of crystallographic groups.
Milnor \cite{mil77} asked if the cocompactness assumption could be removed.
This was answered negatively by Margulis \cite{mar83,mar84}, who constructed the first examples of nonabelian free groups acting properly discontinuously on~$\RR^{2,1}$; the quotient manifolds coming from such actions are often called \emph{Margulis spacetimes}.
Drumm \cite{dru92,dru93} constructed more examples of Margulis spacetimes by introducing \emph{crooked planes} to produce fundamental domains.

\subsection{Proper actions and contraction}

A discrete group~$\Gamma$ acting on $\AdS$ by isometries that preserve both orientation and time orientation is determined by two representations $j,\rho : \Gamma\rightarrow G=\PSL_2(\RR)$, called the \emph{first projection} and \emph{second projection} respectively.
We refer to the group of isometries determined by $(j,\rho)$ using the notation $\Gamma^{j,\rho}$. 
By work of Kulkarni--Raymond \cite{kr85}, if such a group $\Gamma^{j,\rho}$ acts properly on $\AdS$ and is torsion-free, then one of the representations $j,\rho$ must be injective and discrete; if $\Gamma$ is finitely generated (which we shall always assume), then we may pass to a finite-index subgroup that is torsion-free by the Selberg lemma \cite[Lem.\,8]{sel60}.
We assume then that $j$ is injective and discrete.
When $j$ is convex cocompact, Kassel \cite{kasPhD} gave a full characterization of properness of the action of $\Gamma^{j,\rho}$ in terms of a double contraction condition.
Specifically $\Gamma^{j,\rho}$ acts properly on $\AdS$ if and only if either of the following two equivalent conditions holds (up to switching $j$ and $\rho$ if both are convex cocompact):
\begin{itemize}
  \item (\emph{Lipschitz contraction}) There exists a $(j,\rho)$-equivariant Lipschitz~map $f: \HH^2 \rightarrow \HH^2$ with Lipschitz constant $<1$.
  \item (\emph{Length contraction}) 
  \begin{equation}\label{eqn:length_contraction}
  \sup_{\gamma\in\Gamma \text{ with } \lambda(j(\gamma)) > 0}\  \frac{\lambda(\rho(\gamma))}{\lambda(j(\gamma))} \,<\, 1,
  \end{equation} 
\end{itemize}
where $\lambda(g)$ is the hyperbolic translation length of $g\in G$ (defined to be~$0$ if $g$ is not hyperbolic, see \eqref{eqn:def-lambda}).
This was extended by Gu\'eritaud--Kassel \cite{gk12} to the case that the finitely generated group~$j(\Gamma)$ is allowed to have parabolic elements. 
The two (equivalent) types of contraction appearing above are easy to illustrate in the case when $\rho$ is also discrete and injective: the Lipschitz contraction criterion says that there exists a map $j(\Gamma)\backslash\HH^2 \rightarrow \rho(\Gamma)\backslash\HH^2$ (in the correct homotopy class) that uniformly contracts all distances on the surface, while the length contraction criterion says that any closed geodesic on $\rho(\Gamma)\backslash\HH^2$ is uniformly shorter than the corresponding geodesic on $j(\Gamma)\backslash\HH^2$. 
Lipschitz contraction easily implies length contraction, but the converse is not obvious.
One important consequence that can be deduced from either criterion is that for a fixed convex cocompact~$j$, the representations~$\rho$ that yield a proper action form an open set.
In Section~\ref{sec:applications} (which can be read independently), we derive topological and geometric information about the quotient manifold directly from the Lipschitz contraction property. 

We remark that $\Gamma^{j,\rho}$ does not act properly on $\AdS$ in the case that $\Gamma$ is a closed surface group and $j,\rho$ are both Fuchsian (\ie injective and discrete).
For Thurston showed, as part of his theory of the asymmetric metric on Teichm\"uller space \cite{thu86}, that the best Lipschitz constant of maps $j(\Gamma) \backslash \HH^2 \rightarrow \rho(\Gamma) \backslash \HH^2$ (in the correct homotopy class) is $\geq 1$, with equality only if $\rho$ is conjugate to~$j$.
However, $\Gamma^{j,\rho}$ does act properly on a convex subdomain of $\AdS$; the resulting $\AdSS$ manifolds are the globally hyperbolic spacetimes studied by Mess \cite{mes90}.

We now turn to the flat case.
A discrete group~$\Gamma$ acting on~$\RR^{2,1}$ by isometries that preserve both orientation and time orientation is determined by a representation $j : \Gamma\rightarrow\PSL_2(\RR)$ and a $j$-cocycle $u : \Gamma\rightarrow\ssl_2(\RR)$, \ie a map satisfying
$$u(\gamma_1\gamma_2) = u(\gamma_1) + \Ad(j(\gamma_1))\,u(\gamma_2)$$
for all $\gamma_1,\gamma_2\in\Gamma$.
We refer to the group of isometries determined by $(j,u)$ using the notation $\Gamma^{j,u}$, where $j(\Gamma)$ is the \emph{linear part} and $u$ the \emph{translational part} of~$\Gamma^{j,u}$.
The cocycle $u$ may be thought of as an infinitesimal deformation of~$j$ (see Section~\ref{subsec:small-deformation}).
Fried--Goldman \cite{fg83} showed that if $\Gamma$ acts properly on~$\RR^{2,1}$ and is not virtually solvable, then $j$ must be injective and discrete on a finite-index subgroup of $\Gamma$; in particular $j(\Gamma)$ is a surface group (up to finite index).
Unlike in the $\AdSS$ case, here $j(\Gamma)$ cannot be cocompact (see Mess \cite{mes90}).
In the case that it is convex cocompact, Goldman--Labourie--Margulis \cite{glm09} gave a properness criterion in terms of the so-called \emph{Margulis invariant}.
Given the interpretation of this invariant as a derivative of translation lengths \cite{gm00}, the group~$\Gamma^{j,u}$ (with $j$ convex cocompact) acts properly on~$\RR^{2,1}$ if and only if, up to replacing $u$ by~$-u$, the deformation $u$ contracts the lengths of group elements at a uniform rate:
\begin{align}\label{eqn:inf_length_contraction}
\sup_{\gamma\in\Gamma \text{ with } \lambda(j(\gamma)) >0 }\ \frac{\D}{\D t}\Big|_{t=0}\ \frac{\lambda(e^{tu(\gamma)}j(\gamma))}{\lambda(j(\gamma))} &< 0\,.
\end{align}
As a consequence, for a fixed~$j$, the set of $j$-cocycles~$u$ giving a proper~action is open.
The proof involves an extension of the Margulis invariant to the space of geodesic currents on $j(\Gamma)\backslash\HH^2$ and the dynamics of the geodesic flow on the affine bundle $\Gamma^{j,u}\backslash (G\ltimes \g)\rightarrow j(\Gamma)\backslash G$, where $j(\Gamma)\backslash G$ identifies with the unit tangent bundle of $j(\Gamma) \backslash\HH^2$. 

It is natural to view the properness criterion \eqref{eqn:inf_length_contraction} for $\RR^{2,1}$ as an infinitesimal version of the length contraction criterion \eqref{eqn:length_contraction} for $\AdS$, with $\rho$ approaching $j$ along the cocycle $u$. 
In the first part of this paper, we pursue this analogy further by developing an $\RR^{2,1}$ version of the Lipschitz theory of \cite{kasPhD, gk12}, replacing equivariant Lipschitz maps with their infinitesimal analogues, namely deformation vector fields that change distances in a uniformly controlled way.
This yields an infinitesimal version of the Lipschitz contraction criterion as well as a new proof of the infinitesimal length contraction criterion \eqref{eqn:inf_length_contraction} of \cite{glm09}.

\subsection{A new properness criterion for $\RR^{2,1}$}

As before, let $j : \Gamma \rightarrow G = \PSL_2(\RR)$ be a convex cocompact representation. 
An \emph{infinitesimal deformation} of the hyperbolic surface $S:=j(\Gamma)\backslash\HH^2$ is given by a vector field~$X$ on the universal cover $\widetilde{S}=\HH^2$ and a $j$-cocycle $u: \Gamma \rightarrow \g$ such that $X$ is \emph{$(j,u)$-equivariant}: for any $p \in \HH^2$ and $\gamma \in \Gamma$,
$$X(\gamma \cdot p) = \gamma_{\ast}X(p) + u(\gamma)(\gamma \cdot p),$$
where $\Gamma$ acts on $\HH^2$ via~$j$, and elements of $\g$ such as $u(\gamma)$ are interpreted as Killing vector fields on~$\HH^2$ in the usual way.
A $(j,u)$-equivariant vector field~$X$ should be thought of as the derivative of a family of \emph{developing maps} $f_t : \widetilde{S} = \HH^2 \rightarrow \HH^2$ describing a varying family of hyperbolic surfaces $j_t(\Gamma) \backslash \HH^2$, with $t=0$ corresponding to the original hyperbolic structure~$S$ (hence the map $f_0$ is the identity of~$\HH^2$).
The failure of the vector field $X=\frac{\D}{\D t}\big|_{t=0}\,f_t$ to descend to the surface is measured by the derivative of the holonomy representation, which is precisely the $\g$-valued cocycle~$u$:
$$u(\gamma) = \frac{\D}{\D t}\Big|_{t=0}\, j_t(\gamma)\,j(\gamma)^{-1} \in \g.$$ 
We call $u$ the \emph{holonomy derivative} of the deformation~$X$.

We say an infinitesimal deformation $X$ is \emph{$k$-lipschitz} (with a lowercase `l') if for any $p\neq q$ in~$\HH^2$, 
$$\frac{\D}{\D t}\Big|_{t=0}\, d\big(\exp_p(tX(p)), \exp_q(tX(q))\big) \leq k\,d(p,q).$$
The \emph{lipschitz constant} $\lip(X)$ will refer to the infimum of $k\in\RR$ such that $X$ is $k$-lipschitz.
We shall see (Proposition~\ref{prop:Lipschitz-family}) that under appropriate conditions, $\lip(X)$ is the derivative of the Lipschitz constants of a family of developing maps $f_t$ tangent to~$X$ as above.
The lowercase `l' is not intended to diminish the work of Rudolf Otto Sigismund Lipschitz (K\"onigsberg 1832 -- Bonn 1903), but rather to distinguish this notion from the traditional one while emphasizing its infinitesimal nature (as in the notational convention for Lie groups and their Lie algebras).
While a Lipschitz section of the tangent bundle is always lipschitz, the converse is false: for example, if $\chi$ is the characteristic function of the negative reals, then the $0$-lipschitz vector field $x \mapsto \chi(x) \frac{\partial}{\partial x}$ on $\RR$ is not even continuous.
The number $\lip(X)$ can be negative: this means that the vector field $X$ is in an intuitive sense ``contracting''.

With this terminology, here is what we prove.

\begin{Theorem}\label{thm:newGLM}
Let $\Gamma$ be a discrete group, $j\in\Hom(\Gamma,G)$ a convex cocompact representation, and $u : \Gamma\rightarrow\g$ a $j$-cocycle.
The action of $\Gamma^{j,u}$ on~$\RR^{2,1}$ is properly discontinuous if and only if, up to replacing $u$ by~$-u$, one of the equivalent conditions holds:
\begin{enumerate}
\item \emph{(Infinitesimal lipschitz contraction)} For some $k < 0$, there exists a $k$-lipschitz infinitesimal deformation of $j(\Gamma)\backslash\HH^2$ with holonomy derivative $u$.
 \item \emph{(Infinitesimal length contraction)} As in \cite{glm09}:
 $$\sup_{\gamma\in\Gamma\ \mathrm{ with}\ \lambda(j(\gamma))>0}\ \frac{\D}{\D t}\Big|_{t=0}\ \frac{\lambda(e^{tu(\gamma)}j(\gamma))}{\lambda(j(\gamma))} < 0.$$
 \end{enumerate}
\end{Theorem}

We note that our proof of Theorem~\ref{thm:newGLM} and the resulting applications is independent of \cite{glm09}.
As in the $\AdSS$ case, the geometric and topological descriptions of flat Lorentzian manifolds that we give in this paper (Theorem~\ref{thm:fibrations} and Corollary~\ref{cor:tameness} below) are derived directly from the lipschitz contraction criterion; it is not clear to us that they could be derived from length contraction only. 
The proof of Theorem~\ref{thm:newGLM} requires $j(\Gamma)$ to be convex cocompact.
However, we believe that in the future similar techniques could be applied, with some adjustment, to the case when $j(\Gamma)$ contains parabolic elements (as has already been done by \cite{gk12} in the $\AdSS$ case).

\subsection{The topology of quotients of $\AdS$ and $\RR^{2,1}$}

Theorem~\ref{thm:newGLM} and its $\AdSS$ predecessor from \cite{kasPhD,gk12} allow for a complete characterization of the topology of the quotient manifold when $j$ is convex cocompact.
We prove:
 
\begin{Theorem}\label{thm:fibrations}
Let $\Gamma$ be a torsion-free discrete group and $j\in\Hom(\Gamma,G)$ a convex cocompact with quotient surface $S=j(\Gamma)\backslash\HH^2$.
\begin{enumerate}
  \item Let $\rho\in\Hom(\Gamma,G)$ be any representation such that $\Gamma^{j,\rho}$ acts properly on $\AdS$. Then the quotient manifold $\Gamma^{j,\rho} \backslash \AdS$ is a principal $\mathbb{S}^1$-bundle over~$S$.
  \item Let $u : \Gamma\rightarrow\g$ be any $j$-cocycle such that $\Gamma^{j,u}$ acts properly on~$\RR^{2,1}$. Then the quotient manifold $\Gamma^{j,u} \backslash \RR^{2,1}$ is a principal $\RR$-bundle over~$S$.
\end{enumerate}
In both cases the fibers are timelike geodesics.
\end{Theorem}
 
Based on a question of Margulis, Drumm--Goldman \cite{dg95} conjectured in the early 1990s that all Margulis spacetimes should be \emph{tame}, meaning homeomorphic to the interior of a compact manifold.
Since then, Charette--Drumm--Goldman have obtained partial results toward this conjecture, including a proof in the special case that the linear holonomy is a three-holed sphere group \cite{cdg11}.
In the context of Theorem~\ref{thm:fibrations}, we obtain tameness in both the flat and negatively-curved case as a corollary:

\begin{Corollary}\label{cor:tameness}
\begin{enumerate}
  \item Any manifold which is the quotient of $\AdS$ by a group of isometries with convex cocompact first projection is Seifert fibered over a hyperbolic orbifold.
  \item A complete flat Lorentzian manifold with convex cocompact linear holonomy is homeomorphic to the interior of a handlebody. 
\end{enumerate}
\end{Corollary}

In the compact case, Corollary~\ref{cor:tameness}.(1) follows from Kulkarni--Raymond's description of the fundamental groups of quotients of $\AdS$ and from classical results of Waldhausen \cite{wal67} and Scott~\cite{sco83} (see \cite[\S\,3.4.2]{salPhD}).
However, Corollary~\ref{cor:tameness}.(1) is to our knowledge the first result on the topology of noncompact quotients of $\AdS$.
The noncompact quotients appearing here are finitely covered by the tame manifolds of Theorem~\ref{thm:fibrations}.(1) and are therefore tame (e.g. using Tucker's criterion \cite{tuc74}); the Seifert-fibered statement then follows from classical results of Waldhausen.
We note also that Theorem~\ref{thm:fibrations}.(1) and Corollary~\ref{cor:tameness}.(1) actually hold in the more general case that~$j$ is any finitely generated surface group representation; indeed, the properness criterion of \cite{kasPhD} on which we rely holds in this more general setting, by~\cite{gk12}.

Choi--Goldman have recently announced a different proof of the tameness of complete flat Lorentzian manifolds with convex cocompact linear holonomy \cite{cho12,cg13}.
Their proof, which builds a bordification of the $\RR^{2,1}$ spacetime by adding a real projective surface at infinity, is very different from the proof given here. 
In particular, we do not use any compactification and our proof is independent of \cite{glm09}.

\subsection{Margulis spacetimes are limits of AdS manifolds}\label{sec:intro-geomtrans}

We also develop a \emph{geometric transition} from $\AdSS$ geometry to flat Lorentzian geometry.
The goal is to find collapsing $\AdSS$ manifolds which, upon zooming in on the collapse, limit to a given Margulis spacetime.
We obtain two statements that make this idea precise, the first in terms of convergence of real projective structures, the second in terms of convergence of Lorentzian metrics.

The projective geometry approach follows work of Danciger \cite{dan13} in describing the transition from hyperbolic to $\AdSS$ geometry. 
Both $\AdS$ and~$\RR^{2,1}$ are real projective geometries: each space can be represented as a domain in~$\RP^3$, with isometries acting as projective linear transformations. 
As such, all manifolds modeled on either $\AdS$ or $\RR^{2,1}$ naturally inherit a real projective structure.
We show that every quotient of~$\RR^{2,1}$ by a group of isometries with convex cocompact linear holonomy is (contained in) the limit of a collapsing family of complete $\AdSS$ manifolds, in the sense that the underlying real projective structures converge.
Note that collapsing $\AdSS$ manifolds need not (and in this case do not) collapse as projective manifolds, because there is a larger group of coordinate changes that may be used to prevent collapse.

\begin{Theorem}\label{thm:geomtrans}
Let $M = \Gamma^{j,u} \backslash \RR^{2,1}$ be a Margulis spacetime such that $S = j(\Gamma) \backslash \HH^2$ is a convex cocompact hyperbolic surface.
Let $t\mapsto j_t$ and $t\mapsto\rho_t$ be smooth paths with $j_0=\rho_0=j$ and $\frac{\D}{\D t}\big|_{t=0}\,\rho_tj_t^{-1} = u$.
\begin{enumerate}
  \item \label{short-time-admissible} There exists $\delta > 0$ such that for all $t \in (0,\delta)$ the group $\Gamma^{j_t,\rho_t}$ acts properly discontinuously on~$\AdS$.
  \item \label{parameterization} There is a smooth family of $(j_t, \rho_t)$-equivariant diffeomorphisms (developing maps) $\HH^2 \times \mathbb{S}^1 \to \AdS$, defined for $t\in (0,\delta)$, determining complete $\AdSS$ structures $\mathscr{A}_t$ on the fixed manifold $S \times \mathbb{S}^1$.
  \item \label{limit} The real projective structures $\mathscr P_t$ underlying $\mathscr A_t$ converge to a projective structure $\mathscr P_0$ on $S \times \mathbb{S}^1$.
  The Margulis spacetime $M$ is the restriction of $\mathscr P_0$ to $S \times (-\pi, \pi)$, where $\mathbb{S}^1=\RR/2\pi\ZZ$.
\end{enumerate}
\end{Theorem}

In order to construct this geometric transition very explicitly, we arrange for the geodesic fibrations of the $\mathscr A_t$, given by Theorem~\ref{thm:fibrations}, to change continuously in a controlled manner.
In particular, the geodesic fibrations of the collapsing $\AdSS$ manifolds converge to a geodesic fibration of the limiting Margulis spacetime.
The surface $S \times \{\pi\}$ in $\mathscr{P}_0$ compactifies each timelike geodesic fiber, making each fiber into a circle.

As a corollary, we derive a second geometric transition statement in terms of convergence of Lorentzian metrics. 

\begin{Corollary}\label{cor:smooth-metrics}
Let $M$ be a complete flat Lorentzian $3$-manifold with convex cocompact linear holonomy $j(\Gamma)$.
Let $S = j(\Gamma)\backslash \HH^2$ be the associated surface.
Then there exist complete anti-de Sitter metrics $\gmetric_t$ on $S \times \mathbb{S}^1$, defined for all sufficiently small $t > 0$, such that when restricted to $S \times (-\pi,\pi)$, the metrics $t^{-2} \gmetric_t$ converge uniformly on compact sets to a complete flat Lorentzian metric $\gmetric$ that makes $S \times (-\pi,\pi)$ isometric to~$M$.
\end{Corollary}

This second statement is proved using the projective coordinates given in Theorem~\ref{thm:geomtrans}.
Note that the convergence of metrics is more delicate in the Lorentzian setting than in the Riemannian setting.
In particular, even using the topological characterization of Theorem~\ref{thm:fibrations}, it would be difficult to prove Corollary~\ref{cor:smooth-metrics} directly.

\subsection{Maximally stretched laminations}

The key step in the proof of Theorem~\ref{thm:newGLM} is to establish the existence of a \emph{maximally stretched lamination}.
We recall briefly the corresponding statement in the $\AdSS$ setting, as established in \cite{kasPhD,gk12}.
Let $j, \rho: \Gamma \rightarrow G$ be representations with $j$ convex cocompact (or more generally geometrically finite) and let $K$ be the infimum of all possible Lipschitz constants of $(j,\rho)$-equivariant maps $\HH^2 \rightarrow \HH^2$.
If $K \geq 1$, then there is a nonempty geodesic lamination $\mathscr{L}$ in the convex core of $j(\Gamma) \backslash \HH^2$ that is ``maximally stretched'' by \emph{any} $K$-Lipschitz $(j,\rho)$-equivariant map $f : \HH^2\rightarrow\HH^2$, in the sense that $f$ multiplies arc length by exactly~$K$ on the leaves of the lift to~$\HH^2$ of~$\mathscr{L}$.
(In fact, a similar result holds when replacing $\HH^2$ by~$\HH^n$ and $\PSL_2(\RR)\simeq\SO(2,1)_0$ by $\SO(n,1)_0$.)

Now let $u: \Gamma \rightarrow \g$ be a $j$-cocycle and consider $(j,u)$-equivariant vector fields $X$ on~$\HH^2$.
Assume that the infimum~$k$ of lipschitz constants of all such~$X$ satisfies $k\geq 0$.
By analogy, one hopes for the existence of a geodesic lamination that would be stretched at rate exactly~$k$ by any $k$-lipschitz~$X$.
This turns out to be true, but there is a crucial problem: it is not clear that a vector field~$X$ with the best possible lipschitz constant exists.
Indeed, bounded $k$-lipschitz vector fields on a compact set are not necessarily equicontinuous, and so the Arzel\`a--Ascoli theorem does not apply.
In fact, a limit of lipschitz vector fields is something more general that we call a \emph{convex field}.
A convex field is a closed subset of $T\HH^2$ such that the fiber above each point of~$\HH^2$ is a convex set (Definition~\ref{def:convfield}); in other words, a convex field is a closed convex set-valued section of the tangent bundle which is upper semicontinuous in the Hausdorff topology.

\begin{Theorem}\label{thm:laminations}
Let $\Gamma$ be a discrete group, $j \in \Hom(\Gamma, G)$ a convex cocompact representation, and $u : \Gamma\rightarrow\g$ a $j$-cocycle.
Let $k\in\RR$ be the infimum of lipschitz constants of $(j,u)$-equivariant vector fields on~$\HH^2$.
If $k\geq 0$, then there exists a geodesic lamination $\mathscr{L}$ in the convex core of $S:=j(\Gamma)\backslash\HH^2$ that is \emph{maximally stretched} by any $(j,u)$-equivariant, $k$-lipschitz convex field~$X$, meaning that 
$$\frac{\D}{\D t}\Big|_{t=0}\, d\big(\exp_p(t x_p), \exp_q(t x_q)\big) = k\,d(p,q)$$
for any distinct points $p,q\in \HH^2$ on a common leaf of the lift to~$\HH^2$ of~$\mathscr{L}$ and any vectors $x_p \in X(p)$ and $x_q \in X(q)$; such convex fields~$X$ exist.
\end{Theorem}

Let us describe briefly how Theorem~\ref{thm:laminations} implies Theorem~\ref{thm:newGLM}.
When the infimum $k$ of lipschitz constants is $<0$, the action of $\Gamma^{j,u}$ on~$\RR^{2,1}$ is proper: this relatively easy fact is the content of Proposition~\ref{prop:fibrations}, whose proof can be read independently.
Theorem~\ref{thm:laminations} implies that the converse also holds: if $k\geq 0$ then the action is not proper (Proposition~\ref{prop:proper-lip}).
Roughly speaking, following leaves of the maximally stretched lamination~$\mathscr{L}$ gives a sequence $(\gamma_n)_{n\in\NN}$ of pairwise distinct elements of~$\Gamma$ that fail to carry a compact subset of~$\RR^{2,1}$ off itself under the $(j,u)$-action.

We mention also that Theorem~\ref{thm:laminations} recovers the result, due to Goldman--Labourie--Margulis--Minsky \cite{glmm}, that the length contraction criterion \eqref{eqn:inf_length_contraction} still holds when the supremum is taken over simple closed curves rather than the entire fundamental group $\Gamma$.

\subsection{Organization of the paper}

In Section~\ref{sec:prelim} we introduce some notation and recall elementary facts about affine actions, the Margulis invariant, and geodesic laminations.
In Section~\ref{sec:convfields}, we define and give some basic properties about convex fields.
In Section~\ref{sec:extlipschitz} we develop the main tool of the paper, namely the extension theory of lipschitz convex fields on~$\HH^2$, in the spirit of the more classical theory of the extension of Lipschitz maps on~$\HH^2$.
In Section~\ref{sec:prooflamin} we give a proof of Theorem~\ref{thm:laminations}, which is needed for the more difficult direction of Theorem~\ref{thm:newGLM}.
In Section~\ref{sec:applications}, we give the connection between geodesic fibrations and contracting lipschitz fields and prove both directions of Theorem~\ref{thm:newGLM}, as well as Theorem~\ref{thm:fibrations}.
Finally, Section~\ref{sec:geotrans} is dedicated to Theorem~\ref{thm:geomtrans}, showing how to build $\AdSS$ manifolds that limit to a given Margulis spacetime.

\subsection*{Acknowledgements}

We are very grateful to Thierry Barbot for the key idea of Proposition~\ref{prop:fibrations}, which gives the connection between Lipschitz maps of $\HH^2$ and fibrations of $\AdS$.
This paper was also nurtured by discussions with Bill Goldman and Virginie Charette during the 2012 special program on Geometry and analysis of surface group representations at the Institut Henri Poincar\'e in Paris.
We would like to thank the Institut Henri Poincar\'e, the Institut CNRS-Pauli (UMI 2842) in Vienna, and the University of Illinois in Urbana-Champaign for giving us the opportunity to work together in very stimulating environments.
We are also grateful for support from the GEAR Network, funded by the National Science Foundation under grant numbers DMS 1107452, 1107263, and 1107367 (``RNMS: GEometric structures And Representation varieties").
Finally, we thank Olivier Biquard for encouraging us to derive the metric transition statement, Corollary~\ref{cor:smooth-metrics}, from our projective geometry formulation of the transition from $\AdS$ to $\RR^{2,1}$.

\section{Notation and preliminaries}\label{sec:prelim}

\subsection{Anti-de Sitter and Minkowski spaces}\label{subsec:defspaces}

Throughout the paper, we denote by~$G$ the group $\SO(2,1)_0\simeq\PSL_2(\RR)$ and by~$\g$ its Lie algebra.
Let $\langle\cdot |\cdot\rangle$ be half the Killing form of~$\g$: for all $v,w\in\g$,
$$\langle v|w\rangle = \frac{1}{2}\,\mathrm{tr}\big(\mathrm{ad}(v)\mathrm{ad}(w)\big).$$

As mentioned in the introduction, we identify $\AdS$ with the $3$-dimensional real manifold~$G$, endowed with the bi-invariant Lorentzian structure induced by $\langle\cdot |\cdot\rangle$.
Then the identity component of the group of isometries of $\AdS$ is $G\times G$, acting by right and left multiplication:
$$(g_1,g_2)\cdot g := g_2gg_1^{-1}.$$
(Letting $g_1$ act on the right and $g_2$ on the left ensures later compatibility with the usual definition of a cocycle.)

We identify $\RR^{2,1}$ with the Lie algebra~$\g$, endowed with the Lorentzian structure induced by $\langle\cdot |\cdot\rangle$.
Then the identity component of the group of isometries of~$\RR^{2,1}$ is $G\ltimes\g$, acting by affine transformations:
$$(g,v)\cdot w := \Ad(g)w+v.$$
\emph{In the rest of the paper, we will write $g\cdot w$ for $\Ad(g)w$.}

We shall use the usual terminology for rank-one groups: a nontrivial element of~$G$ is \emph{hyperbolic} if it has exactly two fixed points in the boundary at infinity $\partial_{\infty}\HH^2$ of~$\HH^2$, \emph{parabolic} if it has exactly one fixed point in $\partial_{\infty}\HH^2$, and \emph{elliptic} if it has a fixed point in~$\HH^2$.
If $g\in G$ is hyperbolic, we will denote its (oriented) translation axis in $\HH^2$ by~$\A_g$.
For any $g\in G$, we set
\begin{equation}\label{eqn:def-lambda}
\lambda(g) := \inf_{p\in\HH^2} d(p,g\cdot p)\ \geq 0.
\end{equation}
This is the translation length of~$g$ if $g$ is hyperbolic, and $0$ if $g$ is parabolic, elliptic, or trivial.

\subsection{Affine actions}\label{subsec:affineactions}

Recall that a \emph{Killing field} on~$\HH^2$ is a vector field whose flow preserves the hyperbolic metric.
Each $X\in\g$ defines a Killing field
$$p \longmapsto \frac{\D}{\D t}\Big|_{t=0}\, (e^{tX}\cdot p)\ \in T_p\HH^2$$
on~$\HH^2$,
and any Killing field on~$\HH^2$ is of this form for a unique $X\in\g$.
We henceforth identify~$\g$ with the space of Killing fields on~$\HH^2$, writing $X(p)\in T_p\HH^2$ for the vector at $p\in \HH^2$ of the Killing field $X\in\g$.
Under this identification, the adjoint action of $G$ on~$\g$ coincides with the pushforward action of~$G$ on vector fields of~$\HH^2$: 
\begin{equation}\label{eqn:adjoint-pushforward}
(\Ad(g)X)(g\cdot p) = (g_{\ast} X) (g\cdot p) = g_{\ast}(X(p)) = \mathrm{d}_p(L_g)(X(p)),
\end{equation}
where $L_g : \HH^2\rightarrow\HH^2$ is the left translation by~$g$.
We can express Killing fields directly in Minkowski space: if $\g\simeq\RR^{2,1}$ is seen as $\RR^3$ with the quadratic form $x^2+y^2-z^2$ and $\HH^2$ as the upper hyperboloid $\{ (x,y,z)\in\RR^{2,1} \,|\linebreak x^2+y^2-z^2=-1,\ z>0\}$, then
\begin{equation}\label{eqn:crossproduct-killing}
X(p) = X\wedge p \ \in T_p \HH^2 \subset \g
\end{equation}
for all $p\in\HH^2$, where $\wedge$ is the natural Minkowski crossproduct on~$\RR^{2,1}$: 
$$(x_1,x_2,x_3) \wedge (y_1,y_2,y_3) := (x_2y_3-x_3y_2 \, , \,  x_3y_1-x_1y_3 \, , \,-x_1y_2+x_2y_1).$$
(In Lie-theoretic terms, if we see $\HH^2$ as a subset of~$\g$ as above, then $X(p)=\ad(X)\,p$ for all $X\in\g$ and $p\in\HH^2\subset\g$.) 
The cross-product is $\Ad(G)$-equivariant: $g\cdot (v\wedge w)=(g\cdot v)\wedge (g\cdot w)$ for all $g\in G$ and $v,w\in \g$.
Note that in general $\langle p | q\wedge r \rangle = \det(p,q,r)$ is invariant under cyclic permutations~of~$p,q,r$.

Let $\Gamma$ be a discrete group and $j\in\Hom(\Gamma,G)$ a convex cocompact representation.
By \emph{convex cocompact} we mean that $j$ is injective and that $j(\Gamma)$ is a discrete subgroup of~$G$ acting cocompactly on the convex hull $\mathcal{C}_{j(\Gamma)}\subset\HH^2$ of the limit set $\Lambda_{j(\Gamma)}\subset\partial_{\infty}\HH^2$ (the image of $\mathcal{C}_{j(\Gamma)}$ in $j(\Gamma)\backslash\HH^2$ is called the \emph{convex core} of $j(\Gamma)\backslash\HH^2$); equivalently, the hyperbolic orbifold $j(\Gamma)\backslash\HH^2$ has finitely many funnels and no cusp.
By definition, a \emph{$j$-cocycle} is a map $u : \Gamma\rightarrow\g$ such that
\begin{equation} \label{eqn:cocycle-condition}
u(\gamma_1\gamma_2) = u(\gamma_1) + j(\gamma_1)\cdot u(\gamma_2)
\end{equation}
for all $\gamma_1,\gamma_2\in\Gamma$.
A \emph{$j$-coboundary} is a $j$-cocycle of the form
$$u_X(\gamma) = X - j(\gamma)\cdot X$$
where $X\in\g$.
The condition \eqref{eqn:cocycle-condition} means exactly that $\Gamma$ acts on~$\g$ by affine isometries:
\begin{equation}\label{eqn:affineaction}
\gamma \bullet X = j(\gamma) \cdot X + u(\gamma).
\end{equation}
This action fixes a point $X\in\g$ if and only if $u$ is the coboundary~$u_X$.

\begin{Definition}\label{def:proper-def}
We say that $u$ is a \emph{proper deformation} of~$j(\Gamma)$ if the $\Gamma$-action \eqref{eqn:affineaction} on~$\g$ is properly discontinuous. 
\end{Definition}

\subsection{Small deformations}\label{subsec:small-deformation}

The above terminology of proper \emph{deformation} comes from the fact that $j$-cocycles $u : \Gamma\rightarrow\g$ are the same as infinitesimal deformations of the homomorphism~$j$, in the following sense.
Suppose $t\mapsto j_t\in\Hom(\Gamma,G)$ is a smooth path of representations with $j_0=j$.
For $\gamma\in\Gamma$, the derivative $\frac{\D}{\D t}\big|_{t=0}\,j_t(\gamma)$ takes any $p\in\HH^2$ to a vector of $T_{j(\gamma)\cdot p}\,\HH^2$, these vectors forming a Killing field $u(\gamma)$ as $p$ ranges over~$\HH^2$.
As above, we see $u(\gamma)$ as an element of~$\g$; if we also see $\HH^2$ and its tangent spaces as subsets of~$\g$, then \eqref{eqn:crossproduct-killing} yields the formula
\begin{equation}\label{eqn:cocycle-deform}
\frac{\D}{\D t}\Big|_{t=0}\ j_t(\gamma)\cdot p \,=\, u(\gamma) \wedge (j(\gamma)\cdot p)
\end{equation}
for all $p\in\HH^2\subset\g$.
Equivalently, $\Ad_*(\frac{\D}{\D t} j_t(\gamma))=\ad(u(\gamma))\circ \Ad(j(\gamma))$.
The multiplicativity relation in~$\Gamma$ is preserved to first order in~$t$ if and only if for all $\gamma_1,\gamma_2\in \Gamma$,
\begin{eqnarray*}
\frac{\D}{\D t}\Big|_{t=0}\,j_t(\gamma_1\gamma_2)\cdot p & = & \frac{\D}{\D t}\Big|_{t=0}\, \big(j_t(\gamma_1) \cdot (j_t(\gamma_2)\cdot p)\big) \\ 
& = & \!\!\!\Big( \frac{\D}{\D t}\Big|_{t=0}\, j_t(\gamma_1) \Big ) \!\cdot\!  (j(\gamma_2)\cdot p) + j(\gamma_1) \!\cdot\! \Big( \frac{\D}{\D t}\Big|_{t=0}\, j_t(\gamma_2)\cdot p \Big ) \\
&=& u(\gamma_1)\wedge \big(j(\gamma_1\gamma_2) \cdot p\big) + j(\gamma_1)\cdot \big(u(\gamma_2) \wedge (j(\gamma_2)\cdot p)\big) \\
&=& \big(u(\gamma_1) + j(\gamma_1)\cdot u(\gamma_2)\big) \wedge \big(j(\gamma_1\gamma_2) \cdot p\big).
\end{eqnarray*}
Since the left-hand side is also equal to $u(\gamma_1\gamma_2)\wedge (j(\gamma_1\gamma_2)\cdot p)$, this is equivalent to the fact that $u$ is a $j$-cocycle.
Given $X\in \g$, it is easy to check that if $j_t$ is the conjugate of $j$ by $g_t$ where $\frac{\D}{\D t}\big|_{t=0}\,g_t=X\in\g$, then the cocycle $\frac{\D}{\D t}\big|_{t=0}\,j_t$ is the coboundary~$u_X$.

\subsection{The Margulis invariant}\label{subsec:margulis-invariant}

Let $u : \Gamma\rightarrow\g$ be a $j$-cocycle.
We now recall the definition of the Margulis invariant ${\boldsymbol\alpha}_u(\gamma)$ for $\gamma\in\Gamma$ with $j(\gamma)$ hyperbolic (see \cite{mar83,mar84}).
The adjoint action of $j(\gamma)\in G$ on $\g\cong\RR^{2,1}$ has three distinct eigenvalues $\mu>1>\mu^{-1}$.
Let $c^+, c^-$ be eigenvectors in the \emph{positive} light cone of~$\g$, for the respective eigenvalues $\mu,\mu^{-1}$, and let $c^0\in\g$ be the unique \emph{positive} real multiple of $c^-\wedge c^+$ with $\langle c^0|c^0\rangle =1$.
For instance, if $j(\gamma)=\begin{pmatrix} a & 0\\ 0 & a^{-1}\end{pmatrix}\in\PSL_2(\RR)=G$ with $a>1$, then $\mu=a^2$ and we can take
$$c^+ = \begin{pmatrix} 0 & 1\\ 0 & 0\end{pmatrix}, \quad c^- = \begin{pmatrix} 0 & 0\\ -1 & 0\end{pmatrix}, \quad c^0 = \frac{1}{2} \begin{pmatrix} 1 & 0\\ 0 & -1\end{pmatrix}.$$
By definition, the Margulis invariant of~$\gamma$ is
\begin{equation}\label{eqn:margulis}
{\boldsymbol\alpha}_u(\gamma) := \langle u(\gamma) \,|\, c^0\rangle~. 
\end{equation}
It is an easy exercise to check that ${\boldsymbol\alpha}_u$ is invariant under conjugation and that ${\boldsymbol\alpha}_u(\gamma^n)=|n|\,{\boldsymbol\alpha}_u(\gamma)$ for all $n\in\ZZ$.
The affine action of $\gamma$ on~$\g$ by $\gamma\bullet X=j(\gamma)\cdot X+u(\gamma)$ preserves a unique affine line directed by~$c^0$, and ${\boldsymbol\alpha}_u(\gamma)$ is the (signed) translation length along this line.
Since the image of $\mathrm{Id}_{\g}-\Ad(j(\gamma))$ is orthogonal to~$c^0$, we have ${\boldsymbol\alpha}_u(\gamma)=\langle\gamma\bullet X\,|\,c^0\rangle$ for all $X\in\g$.
In particular, if $u$ is a coboundary, then ${\boldsymbol\alpha}_u(\gamma)=0$ for all~$\gamma$.
If $u_1$ and~$u_2$ are two $j$-cocycles and $t_1,t_2\in\RR$, then
$${\boldsymbol\alpha}_{t_1u_1+t_2u_2} = t_1\,{\boldsymbol\alpha}_{u_1} + t_2\,{\boldsymbol\alpha}_{u_2}.$$
Thus ${\boldsymbol\alpha}_u$ depends only on the cohomology class of~$u$.

Note that the projection of the Killing field $u(\gamma)\in\g$ to the oriented translation axis $\A_{j(\gamma)}\subset\HH^2$ is the same at all points $p$ of~$\A_{j(\gamma)}$, equal to ${\boldsymbol\alpha}_u(\gamma)$.
Indeed, the unit tangent vector to $\A_{j(\gamma)}$ at $p$ is $c^0\wedge p$, and so the $\A_{j(\gamma)}$-component $\pi_{\A_{j(\gamma)}}$ of $u(\gamma)(p)=u(\gamma)\wedge p\in T_p\HH^2$ is 
\begin{eqnarray}\label{eqn:constant-component}
\pi_{\A_{j(\gamma)}}\big(u(\gamma)(p)\big) & = & \langle c^0\wedge p \,|\, u(\gamma)\wedge p \rangle = \langle u(\gamma) \,|\, p\wedge (c^0\wedge p) \rangle\nonumber\\ 
& = & \langle u(\gamma) \,|\, c^0 \rangle = {\boldsymbol\alpha}_u(\gamma)
\end{eqnarray}
since $c^0$ and $p$ are mutually orthogonal (respectively spacelike and timelike) with unit norms.
More generally, any Killing field always has a constant component along any given line.

\subsection{Length derivative}\label{subsec:length_deriv}

Finally, we recall that if the $j$-cocycle~$u$ is the derivative $\frac{\D}{\D t}\big|_{t=0}\,j_t$ of a smooth path $t\mapsto j_t\in\Hom(\Gamma,G)$ with $j_0=j$, then the Margulis invariant ${\boldsymbol\alpha}_u(\gamma)$ associated with~$u$ is also the $t$-derivative of the length $\lambda(j_t(\gamma))$ of the geodesic curve in the class of~$\gamma$:
$${\boldsymbol\alpha}_u(\gamma) = \frac{\D}{\D t}\Big|_{t=0}\, \lambda(j_t(\gamma)) = \frac{\D}{\D t}\Big|_{t=0}\, \lambda\big(e^{tu(\gamma)}j(\gamma)\big) .$$
This was first observed by Goldman--Margulis \cite{gm00}.
Here is a short explanation.
By conjugating $j_t(\gamma)$ by a smooth path based at $1\in G$, we may assume that the translation axis of $j_t(\gamma)$ is constant; indeed, $\lambda$ is invariant under conjugation, and conjugation changes $u$ by a coboundary.
The key point is that $j(\gamma)=e^{\lambda(j(\gamma))\,c^0}$, where $c^0$ is the unit spacelike vector of Section~\ref{subsec:margulis-invariant}.
Since $j_t(\gamma)$ has the same translation axis as $j(\gamma)$, we can write
$$j_t(\gamma) = e^{\lambda(j_t(\gamma))\,c^0} = e^{[\lambda(j_t(\gamma))-\lambda(j(\gamma))]\,c^0}\,j(\gamma)\,.$$
Thus $u(\gamma)=\frac{\D}{\D t}\big|_{t=0}\,\lambda(j_t(\gamma))\,c^0$.
The formula follows: since $\langle c^0|c^0\rangle = 1$,
$${\boldsymbol\alpha}_u(\gamma) = \langle u(\gamma) \,|\, c^0\rangle = \frac{\D}{\D t}\Big|_{t=0}\, \lambda(j_t(\gamma))\,.$$

By rigidity of the marked length spectrum for surfaces (see \cite{dg01})
we thus have ${\boldsymbol\alpha}_u=0$ if and only if $u$ is a coboundary.

\subsection{Geodesic laminations}\label{subsec:geo-lamin}

Let $\Omega$ be an open subset of $\HH^2$. 
In this paper, we call \emph{geodesic lamination} in~$\Omega$ any closed subset $\tilde{\mathscr{L}}$ of~$\Omega$ endowed with a partition into straight lines, called \emph{leaves}.
We allow leaves to end at the boundary of~$\Omega$ in~$\HH^2$.
Note that the disjointness of the leaves implies that the collection of leaves is closed in the $C^1$ sense: any limiting segment $\sigma$ of a sequence of leaf segments $\sigma_i$ is a leaf segment (otherwise the leaf of~$\tilde{\mathscr{L}}$ through any point of $\sigma$ would intersect the~$\sigma_i$).
If we worked in $\HH^n$ with $n>2$, then $C^1$-closedness would have to become part of the definition.

When $\tilde{\mathscr{L}}$ and $\Omega$ are globally invariant under some discrete group $j(\Gamma)$, we also call lamination the projection of $\tilde{\mathscr{L}}$ to the quotient $j(\Gamma)\backslash \Omega$ (it is a closed disjoint union of injectively immersed geodesic copies of the circle and/or the line).
In such a quotient lamination, if some half-leaf does not escape to infinity, then it accumulates on a sublamination which can be approached by a sequence of simple closed geodesics.
Geodesic laminations are thus in some intuitive sense a generalization of simple closed (multi)curves.

\section{Vector fields and convex fields}\label{sec:convfields}

Some fundamental objects in the paper are equivariant, lipschitz vector fields, as well as what we call \emph{convex fields}.
\begin{Definition}\label{def:convfield}
A \emph{convex field} on~$\HH^2$ is a closed subset $X$ of the tangent bundle $T\HH^2$ whose intersection~$X(p)$ with $T_p\HH^2$ is convex for any $p\in\HH^2$.
\end{Definition}

Equivalently, a convex field is a subset $X$ of~$\HH^2$ whose intersection~$X(p)$ with $T_p\HH^2$ is convex and closed for any~$p$ and such that $X(p)$ depends in an upper semicontinuous way on~$p$ for the Hausdorff topology.
For instance, any continuous vector field $X$ on~$\HH^2$ is a convex field; \emph{we shall assume all vector fields in the paper to be continuous}.
In general, we do allow certain fibers $X(p)$ to be empty, but say that the convex field $X$ is \emph{defined} over a set $A\subset\HH^2$ if all fibers above~$A$ are nonempty.
We shall use the notation
$$X(A) := \bigcup_{p\in A} X(p)\subset T\HH^2$$
and
\begin{equation}\label{eqn:normconvexfield}
\| X(A)\| := \sup_{x\in X(A)} \| x\| .
\end{equation}

\subsection{Definitions and basic properties of convex fields}

For any convex fields $X_1$ and~$X_2$ and any real-valued functions $\psi_1$ and~$\psi_2$, we define the sum $\psi_1X_1+\psi_2X_2$ fiberwise:
$$(\psi_1X_1 + \psi_2X_2)(p) = \big\{ \psi_1(p)\,v_1 + \psi_2(p)\,v_2 \ |\ v_i\in X_i(p)\big\} .$$
It is still a convex field.

\begin{Definition}\label{def:equivariant}
Let $\Gamma$ be a discrete group, $j\in\Hom(\Gamma,G)$ a representation, and $u : \Gamma\rightarrow\g$ a $j$-cocycle.
We say that a convex field $X$ on~$\HH^2$ is \emph{$(j,u)$-equivariant} if for all $\gamma \in \Gamma$ and $p\in \HH^2$,
$$X(j(\gamma)\cdot p) = j(\gamma)_{\ast}(X(p)) + u(\gamma)(j(\gamma)\cdot p).$$
A $(j,0)$-equivariant field is called \emph{$j$-invariant}.
\end{Definition}

If $t\mapsto j_t$ is a deformation of $j$ tangent to~$u$, then the $(j,u)$-equivariance of a vector field~$X$ expresses the fact that whenever $j(\gamma)\cdot p = q$, the relation persists to first order under the flow of~$X$, that is, $d(j_t(\gamma)\cdot p_t, q_t)=o(t)$ where $p_t = \exp_p(t X(p))$ and $q_t = \exp_q(t X(q))$.

We can rephrase Definition~\ref{def:equivariant} in terms of group actions.
The group~$\Gamma$ acts on convex fields via the \emph{pushforward} action $(\gamma_* X)(j(\gamma)\cdot p)=j(\gamma)_*(X(p))$, and also via the affine $u$-action 
\begin{equation}\label{eqn:u-action}
\gamma\bullet X = \gamma_*X + u(\gamma)
\end{equation}
(which is a group action, due to \eqref{eqn:adjoint-pushforward} and the cocycle condition \eqref{eqn:cocycle-condition}).
A convex field $X$ is $(j,u)$-equivariant (resp.\ $j$-invariant) if and only if $\gamma\bullet X = X$ (resp.\ $\gamma_*X=X$) for all $\gamma\in\Gamma$.

\begin{Definition}\label{def:lipX}
A convex field $X$ is \emph{$k$-lipschitz} (lowercase `l') if for any distinct points $p,q\in \HH^2$ and any vectors $x_p\in X(p)$ and $x_q \in X(q)$, the rate of change of the distance between $p$ and $q$ satisfies
\begin{equation}\label{eqn:d'}
d'(x_p,x_q) \: := \: \frac{\D}{\D t}\Big|_{t=0}\, d\big(\exp_p(tx_p), \exp_q(tx_q)\big) \: \leq \:  k\,d(p,q).
\end{equation}
The \emph{lipschitz constant} of~$X$, denoted by $\lip(X)$, is the infimum of $k\in\RR$ such that $X$ is $k$-lipschitz.
For $A\subset\HH^2$, we set
$$\lip_A(X) := \lip(X(A)).$$ 
Finally, for $p\in\HH^2$, we define the \emph{local lipschitz constant} $\lip_p(X)$ to be the infimum of $\lip_{\mathcal{U}}(X)$ over all neighborhoods $\mathcal{U}$ of $p$ in~$\HH^2$.
We shall often use the notation
$$d'_X(p,q) := \sup\big\{d'(x_p, x_q) \ |\ x_p\in X(p),\ x_q\in X(q)\big\}\, .$$
The inverse of the map $\exp_p : T_p\HH^2\rightarrow\HH^2$ will be written $\log_p$. 
\end{Definition}

A diagonal argument shows that the ``local lipschitz constant'' function $p\mapsto\lip_p(X)$ is upper semicontinuous: for any converging sequence $p_n\rightarrow p$,
$$\lip_p(X) \,\geq\, \limsup_{n\rightarrow +\infty}\, \lip_{p_n}(X).$$

In order to compute (or estimate) lipschitz constants, we will often make use of the following observation (see Figure~\ref{fig:dprime}).

\begin{Remark}\label{rem:d'proj}
The quantity $d'(x_p,x_q)$ is the difference of the (signed) projections of $x_p$ and~$x_q$ to the geodesic line $(p,q)\subset\HH^2$, oriented from $p$ to~$q$.
\end{Remark}

\begin{figure}[h!]
{\centering
\def\svgwidth{7cm}
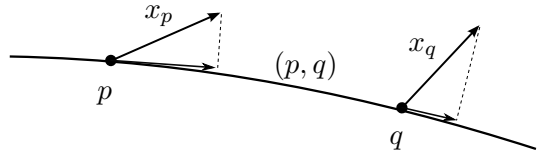}
\caption{The quantity $d'(x_p, x_q)$ may be calculated as the difference of signed projections of $x_p$ and $x_q$ to the line $(p,q)$. Here the contribution from $x_p$ is negative ($x_p$ pushes $p$ towards~$q$), while the contribution from $x_q$ is positive ($x_q$ pushes $q$ away from~$p$). }
\label{fig:dprime}
\end{figure}

In the case that $X$ is a smooth vector field, the local lipschitz constant is given by the formula
\begin{equation}\label{eqn:lipsmooth}
\lip_p(X) = \sup_{y \in T^1_p \HH^2} \langle \nabla_y X, y \rangle,
\end{equation}
where $\nabla$ is the Levi-Civita connection and $T^1_p \HH^2$ denotes the unit tangent vectors based at~$p$.
This is the vector field analogue of the formula $\Lip_p(f) = \|\D_pf\|$ for smooth maps~$f$.

The following remarks are straightforward.

\begin{Observation}\label{obs:donkey}
Let $X$ be a convex field.
\begin{enumerate}[(i)]
  \item If $X$ is $(j,u)$-equivariant and if $X_0$ is a $j$-invariant convex field, then the convex field $X+X_0$ is $(j, u)$-equivariant.
  \item If $(X_i)_{i\in I}$ is a family of $(j,u)$-equivariant convex fields with $\bigcup_{i\in I}X_i(p)$ bounded for all $p\in\HH^2$, and $(\mu_i)_{i\in I}$ a family of nonnegative reals summing up to~$1$, then the convex field $\sum_{i\in I} \mu_i X_i$ is well defined and $(j,u)$-equivariant.
  \item If in addition $X_i$ is $k_i$-lipschitz, with $(k_i)_{i\in I}$ bounded, then the convex field $\sum_{i\in I} \mu_i X_i$ is $\left ( \sum_{i\in I}\mu_i k_i\right )$-lipschitz.
  \item Subdivision: if a segment $[p,q]$ is covered by open sets $\mathcal{U}_i$ such that $\lip_{\mathcal{U}_i}(X)\leq k$ for all~$i$, then $d'_X(p,q)\leq k\,d(p,q)$.
  \item In particular, if $A\subset\HH^2$ is convex, then $\lip_A(X)=\sup_{p\in A} \lip_p(X)$.
  \item If $d'_X(p,q)=k:=\lip(X)$, then $d'_X(p,r)=d'_X(r,q)=k$ for any point $r$ in the interior of the geodesic segment $[p,q]$; in this case we say that the segment $[p,q]$ is \emph{$k$-stretched} by~$X$.
  \item Invariance: if $X$ is $(j,u)$-equivariant, then $\lip_{j(\gamma)\cdot A}(X)=\lip_A(X)$ and $\lip_{j(\gamma)\cdot p}(X)=\lip_p(X)$ for all $\gamma\in\Gamma$, all $A\subset\HH^2$, and all $p\in\HH^2$.
  \item The map $d'$ is subscript-additive: $d'_X(p,q)+d'_Y(p,q)=d'_{X+Y}(p,q)$.
  \item The map $d'_X(\cdot, \cdot)$ is uniformly $0$ if and only if $X$ is a Killing field.
\end{enumerate}
\end{Observation}

If $k<0$, then any $k$-lipschitz \emph{vector field} $X$ on~$\HH^2$ tends to bring points closer together; in particular, $X$ has a positive inward component on the boundary of any large enough round ball of fixed center.
By Brouwer's theorem, $X$ therefore has a zero in~$\HH^2$, necessarily unique since $k<0$.
In fact, this extends to convex fields:

\begin{Proposition}\label{prop:fixpoint}
Any $k$-lipschitz convex field $X$ with $k<0$, defined on all of $\HH^2$, has a unique zero (that is, there is a unique fiber $X(p)$ containing $0\in T_p\HH^2$).
\end{Proposition}

\begin{proof}
We prove this by contradiction: suppose $X$ is a counterexample; let us construct a \emph{vector field} $Y$ on a large ball $B$ of~$\HH^2$ with no zero, but with positive inward component everywhere on $\partial B$.
Fix $p\in \HH^2$ and $x\in X(p)$.
If $B$ is a large enough ball centered at~$p$, of radius~$R$, then every vector $y\in X(q)$ for $q\in \partial B$ is inward-pointing because $d'(x,y)\leq k \, d(p,q)=kR\ll 0$.
Since $X$ has no zero and is closed in $T\HH^2$, with convex fibers, we can find for any $q\in B$ a neighborhood $V_q$ of~$q$ in~$\HH^2$ and a \emph{vector field}~$Y_q$ defined on~$V_q$, such that $Y(q')$ has positive scalar product with any vector of $X(q')$ when $q'\in V_q$.
Moreover, we can assume that the fields $Y_q$ are all inward-pointing in a neighborhood of $\partial B$.
Extract a finite covering $V_{q_1}\cup \dots \cup V_{q_n}$ of~$B$ and pick a partition of unity $(\psi_i)_{1\leq i\leq n}$ adapted to the~$V_{q_i}$.
Then the vector field $Y:=\sum_{i=1}^n \psi_iY_{q_i}$ is continuous, defined on all of~$B$, inward-pointing on $\partial B$, and with no zero (it everywhere pairs to positive values with~$X$).
This vector field~$Y$ cannot exist by Brouwer's theorem, hence $X$ must have a zero --- necessarily unique since $X$ is $k$-lipschitz with $k<0$.
\end{proof}

Proposition~\ref{prop:fixpoint} and its proof may be compared to Kakutani's fixed point theorem for set-valued maps \cite{kak41}.
Here are two related results, which will be important throughout the paper:

\begin{Proposition}\label{prop:limitofconvex}
Any Hausdorff limit of a sequence of convex fields that are uniformly bounded and $k$-lipschitz over a ball~$B$, is a $k$-lipschitz convex field defined over~$B$.
\end{Proposition}

\begin{proof}
Let $(X_n)_{n\in\NN}$ be such a sequence of convex fields, and $X_{\infty}$ their Hausdorff limit (a closed subset of $T\HH^2$).
For any $p\in B$, the closed set $X_{\infty}(p)$ is nonempty because $(X_n(p))_{n\in\NN}$ is uniformly bounded.
To see that $X_{\infty}$ is $k$-lipschitz, we fix distinct points $p,q\in B$ and consider sequences $x_n\in X_n(p_n)$ converging to $x\in X_{\infty}(p)$ and $y_n \in X_n(q_n)$ converging to $y\in X_{\infty}(q)$.
Then $d'(x_n,y_n) \leq k\,d(p_n, q_n)$ for all~$n$, and taking the limit as $n\to +\infty$ gives $d'(x,y)\leq k\,d(p,q)$.

We now check that $X_{\infty}(p)$ is convex for all $p\in B$.
By adding a Killing field, it is enough to show that if the zero vector lies in the convex hull $\CH(X_{\infty}(p))$ of $X_{\infty}(p)$ in $T_p\HH^2$, then $0\in X_{\infty}(p)$.
Consider the vector field $W: q \mapsto \log_q(p)$ that points toward $p$ with strength equal to the distance from~$p$.
By convexity of the distance function in~$\HH^2$, the vector field $W$ is $-1$~lipschitz.
Let $Y_n = X_n + c W$, where $c\gg 1$ is large enough so that for any~$n$, the convex field $Y_n$ is $-1$~lipschitz and for all $p\in\partial B$, all vectors of $Y_n(p)$ point strictly into~$B$.
By the proof of Proposition~\ref{prop:fixpoint}, the convex field $Y_n$ has a zero at a point $q_n\in B$, and after taking a subsequence we may assume that $(q_n)_{n\in\NN}$ converges to some $q\in B$.
Then $Y_{\infty} = X_{\infty} + c W$ is $-1$~lipschitz and $0\in Y_{\infty}(q)$.
The fiberwise convex hull of $Y_{\infty}$ is still $-1$~lipschitz and contains $0\in T_q\HH^n$.
Note that $X_{\infty}(p)=Y_{\infty}(p)$.
Therefore, if $0\in\CH(X_{\infty}(p))$, then $p = q$ and we have $0 \in X_{\infty}(p) = Y_{\infty}(p)$.
\end{proof}

\begin{Proposition}\label{prop:zeroconvex}
Let $X$ be a $0$-lipschitz convex field defined over~$\HH^2$.
For any Killing field $Y$ on~$\HH^2$, the set $\specialC:=\{p\in \HH^2 \,|\, Y(p)\in X\}$ is convex.
\end{Proposition}

\begin{proof}
We may assume that $Y$ is the zero vector field $\underline{0}$, up to replacing $X$ with $X-Y$ (which is still $0$-lipschitz).
Consider two distinct points $p,q\in\specialC$ and a point $r$ on the segment $[p,q]$.
Thinking of $[p,q]$ as the horizontal direction, consider a point $r'\in\nolinebreak\HH^2$ very close to~$r$ above $[p,q]$.
Let $x_{r'} \in X(r')$.
Since $d'(\underline{0}(p),x_{r'}) \leq 0$ and $d'(\underline{0}(q),x_{r'}) \leq 0$, the vector $x_{r'}$ must belong to a narrow angular sector around the vertical, downward direction (see Figure~\ref{fig:flat}).
In the limit as $r'$ approaches $r$ from above $[p,q]$, we find that $X(r)$ must contain a vector $x_0$ orthogonal to $[p,q]$ pointing (weakly) down.
By letting $r'$ approach $r$ from below, we find that $X(r)$ also contains a vector $x_1$ orthogonal to $[p,q]$ pointing (weakly) up.
Since $\underline{0}(r)$ is in the convex hull of $\{x_0, x_1\}$, we have $\underline{0}(r)\in X(r)$, hence $r\in\specialC$.
\end{proof}

\begin{figure}[h!]
{\centering
\def\svgwidth{9cm}
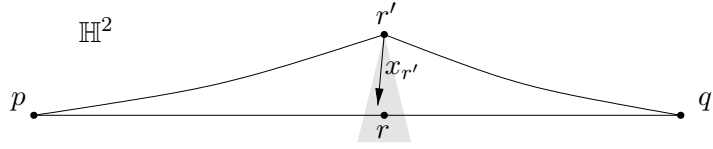}
\caption{If $d'_X(p,q)=0=\lip(X)$, then $X$ contains the restriction of a Killing field to $[p,q]$.}
\label{fig:flat}
\end{figure}

\subsection{Computing the Margulis invariant from an equivariant vector field}\label{subsec:computingMarginv}

The Margulis invariant ${\boldsymbol\alpha}_u$ and the map $d':T\HH^2\times T\HH^2\rightarrow\RR$ of Definition~\ref{def:lipX} both record rates of variation of hyperbolic lengths.
We now explain how one can be expressed in terms of the other via equivariant convex fields.
Fix a $(j,u)$-equivariant convex field $X$ on~$\HH^2$.
For $\gamma\in\Gamma$~with $j(\gamma)$ hyperbolic, choose a point $p$ on the oriented axis $\A_{j(\gamma)}\subset\HH^2$ of $j(\gamma)$ and a vector $x_p$ in the convex set $X(p)$.
Define also $x_{j(\gamma)\cdot p}:=j(\gamma)_*(X_p)+u(\gamma)(j(\gamma)\cdot p)$, which belongs to $X(j(\gamma)\cdot p)$ by equivariance of~$X$.
If $\pi_{\A_{j(\gamma)}}(x)$ denotes the $\A_{j(\gamma)}$-component of any vector $x\in T\HH^2$ based at a point of~$\A_{j(\gamma)}$, then \eqref{eqn:constant-component} implies
\begin{eqnarray}\label{eqn:lip-marg1}
d'(x_p, x_{j(\gamma)\cdot p}) & = & \pi_{\A_{j(\gamma)}}(x_{j(\gamma)\cdot p}) - \pi_{\A_{j(\gamma)}}(x_p)\\ 
& = & \pi_{\A_{j(\gamma)}}\big(u(\gamma)(j(\gamma)\cdot p)\big) = {\boldsymbol\alpha}_u(\gamma). \notag 
\end{eqnarray}
In particular,
\begin{equation}\label{eqn:lip-marg2}
\frac{{\boldsymbol\alpha}_u(\gamma)}{\lambda(j(\gamma))} \leq \lip(X).
\end{equation}
We now assume that $X$ is a smooth equivariant \emph{vector field}.
The function $\newf_X :\nolinebreak\A_{j(\gamma)}\rightarrow\nolinebreak\RR$ defined by $\newf_X(p)=\pi_{\A_{j(\gamma)}}(X(p))$ satisfies $\newf_X(j(\gamma)\cdot p)=\newf_X(p)+{\boldsymbol\alpha}_u(\gamma)$, hence the derivative $\newf_X'$ is periodic and descends to a \emph{scalar} function on the geodesic loop $c_{\gamma}$ representing the isotopy class of $\gamma$ on the hyperbolic orbifold $S=j(\Gamma)\backslash\HH^2$.
By construction, $d'_X(p,j(\gamma)\cdot p)$ is just the integral of $\newf_X'$ along $c_{\gamma}$ for the Lebesgue measure~$\D\mu_{\gamma}$.
Therefore,
\begin{equation}\label{eqn:muintegral}
{\boldsymbol\alpha}_u(\gamma) = \int_{c_{\gamma}} \newf_X' \,\D\mu_{\gamma}.
\end{equation}
This formula holds independently of the choice of the smooth equivariant vector field~$X$.

Moreover, we can generalize this process and extend ${\boldsymbol\alpha}_u$ to the space of geodesic currents.
This extension was described, in different terms, in \cite{lab01} and \cite{glm09}. 
First, the functions $\newf_X : \A_{j(\gamma)}\rightarrow\RR$ above, for $\gamma\in\Gamma$ with $j(\gamma)$ hyperbolic, piece together and extend to a smooth function on the unit tangent bundle $T^1\HH^2$ of~$\HH^2$, which we again denote by~$\newf_X$: it takes $y\in T^1_p\HH^2$ to $\langle X(p), y \rangle\in\RR$.
By construction, the derivative $\newf_X' : T^1\HH^2\rightarrow \RR$ of~$\newf_X$ along the geodesic flow satisfies
\begin{equation}\label{eqn:dintegral}
d'_X(p,q)=\int_{[p,q]} \newf'_X
\end{equation}
for any distinct $p,q\in \HH^2$, where the geodesic flow line $[p,q]\subset T^1\HH^2$ from $p$ to~$q$ is endowed with its natural Lebesgue measure.
In terms of the Levi-Civita connection~$\nabla$,
$$\newf_X'(y) = \langle\nabla_yX,y\rangle_{T_p \HH^2}$$
for any unit vector $y$ in the Euclidean plane $T_p\HH^2$.
Remarkably, the function $\newf'_X$ is $j(\Gamma)$-invariant, because $j(\gamma)_* X$ and $X$ differ only by a Killing field $u(\gamma)$, and Killing fields have constant component along any geodesic flow line.
Therefore $\newf_X'$ descends to the unit tangent bundle $T^1S$ of $S=j(\Gamma)\backslash\HH^2$, and \eqref{eqn:muintegral} can be rewritten in the form 
\begin{equation}\label{eqn:muintegral-currents}
{\boldsymbol\alpha}_u(\gamma)=\int_{T^1S} \newf_X' \,\D\mu_{\gamma},
\end{equation}
which extends to all geodesic currents $\D\mu$ on $T^1S$ since $\newf_X'$ is continuous.
Here is a useful consequence of this construction.

\begin{Proposition}\label{prop:longcurves}
Suppose there exists a $(j,u)$-equivariant convex field $Y$ that $k$-stretches a geodesic lamination $\mathscr{L}$ in the convex core of~$S$, in the sense that $d'_Y(p,q)=k\,d(p,q)$ for any distinct $p,q\in\HH^2$ on a common leaf of preimage $\tilde{\mathscr{L}}\subset\HH^2$ of~$\mathscr{L}$.
Then for any sequence $(\gamma_n)_{n\in\NN}$ of elements of~$\Gamma$ whose translation axes $\A_{j(\gamma_n)}$ converge to (a sublamination of) $\tilde{\mathscr{L}}$ in the Hausdorff topology,
$$\lim_{n\to +\infty} \frac{{\boldsymbol\alpha}_u(\gamma_n)}{\lambda(j(\gamma_n))} = k.$$
\end{Proposition}

\begin{proof}
Let $X$ be any smooth, $(j,u)$-equivariant vector field.
Then $Y-X$ is $j$-invariant, hence bounded over the convex core by convex cocompactness.
Thus for any distinct $p,q\in\HH^2$ on a common leaf of~$\tilde{\mathscr{L}}$, the difference between $d'_Y(p,q)=k\,d(p,q)$ and $d'_{X}(p,q)$ is bounded.
Hence the average value of~$\newf'_{X}$ over a segment of length~$L$ of~$\mathscr{L}$ is $k+O(L^{-1})$.
On the other hand, $\newf'_X$ is a uniformly continuous function on $T^1S$.
Since the loops representing $\gamma_n$ lift to long segments $c_n\subset\HH^2$, of length $\lambda(j(\gamma_n))$, uniformly close to leaves of~$\tilde{\mathscr{L}}$, uniform continuity of~$\newf'_X$ implies ${\boldsymbol\alpha}_u(\gamma_n)=\int_{c_n}\newf'_X \sim k\,\lambda(j(\gamma_n))$.
\end{proof}

\subsection{A priori bounds inside the convex core}\label{subsec:corebounds}

Let $j\in\Hom(\Gamma,G)$ be convex cocompact and let $U\subset\HH^2$ be the preimage of the interior of the convex core of $j(\Gamma)\backslash\HH^2$.
Let $u : \Gamma\rightarrow\g$ be a $j$-cocycle.
In this section we prove the following.

\begin{Proposition}\label{prop:bounds}
For any compact subset $\specialC$ of~$U$ and any $k\in\RR$, there exists $R>0$ such that for any $k$-lipschitz, $(j,u)$-equivariant convex field~$X$, any vector $x \in X(\specialC)$ satisfies $\| x \| < R$. 
\end{Proposition}

\begin{proof}
Consider $\gamma\in\Gamma$ such that $j(\gamma)$ is hyperbolic, with translation axis $\A=\A_{j(\gamma)}$.
Let $p\in\HH^2\smallsetminus\A$ and $q := j(\gamma)\cdot p$ be at distance $r>0$ from~$\A$.
Let $\hat n$ be the unit vector field pointing away from~$\A$ in the direction orthogonal to~$\A$, and $\theta\in (0,\frac{\pi}{2})$ the angle at $p$ (or $q$) between $-\hat{n}$ and the segment $[p,q]$ (see Figure~\ref{fig:propbounds}).
A classical formula gives $\tan \theta= \frac{\coth \lambda(j(\gamma))/2}{\sinh r}.$
We claim that for any $x\in T_p \HH^2$,
\begin{equation}\label{eqn:boundtheta}
d'(x,j(\gamma)_* x) = 2 \langle x, \hat n \rangle \cos \theta.
\end{equation}
Indeed, let $\hat a$ be the unit vector field along the segment $[p,q]$, oriented towards~$q$.
Let $\hat e$ be the unit vector field orthogonal to $\hat n$ such that at $p$, the vectors $\hat a$ and~$\hat e$ form an angle $\theta'=\frac{\pi}{2}-\theta$.
By symmetry, at~$q$, the vectors $\hat a$ and~$\hat e$ form an angle $-\theta'$.
Note that the fields $\hat e$ and~$\hat n$ are invariant under~$j(\gamma)$.
\begin{figure}[h]
{\centering
\def\svgwidth{6cm}
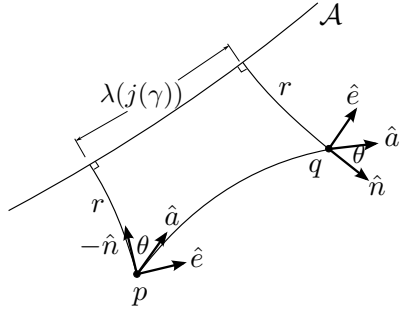}
\caption{Computing $d'(x,j(\gamma)_* x)$ in the proof of Proposition~\ref{prop:bounds}.}
\label{fig:propbounds}
\end{figure}
We have
\begin{align*}
\langle x, \hat a \rangle &= \langle x , \hat e \rangle \cos \theta' - \langle x, \hat n \rangle \sin \theta',\\
\langle j(\gamma)_*x, \hat a \rangle &= \langle x , \hat e \rangle \cos \theta' + \langle x, \hat n \rangle \sin \theta',
\end{align*}
hence $d'(x, j(\gamma)_* x) = 2 \langle x, \hat n \rangle \sin \theta'=2\langle x, \hat n \rangle \cos \theta$ (by Remark~\ref{rem:d'proj}), proving~\eqref{eqn:boundtheta}.

Now, given a point $p\in U$, we choose three elements $\gamma_1, \gamma_2, \gamma_3 \in \Gamma$ such that $j(\gamma_i)$ is hyperbolic and its translation axis $\A_i=\A_{j(\gamma_i)}$ does not contain~$p$, for each $1\leq i\leq 3$. 
Set $q_i=j(\gamma_i)\cdot p$ and $\lambda_i=\lambda(j(\gamma_i))$, and define $r_i, \theta_i, \hat{n}_i$ similarly to above. 
Since $p\in U$, we may choose the $\gamma_i$ so that the positive span of the $\hat n_i$ is all of $T_p \HH^2$, \ie the $\hat n_i$ are not all contained in a closed half-plane. Let $\vartheta_{ii'}\in (0,\pi)$ be the angle between~$\hat n_i$ and~$\hat n_{i'}$ (see Figure~\ref{fig:propbounds2}). 
Set $Q:=\max_{1\leq i \leq 3} \| u(\gamma_i)(q_i)\|$.
\begin{figure}[h!]
{\centering
\def\svgwidth{6cm}
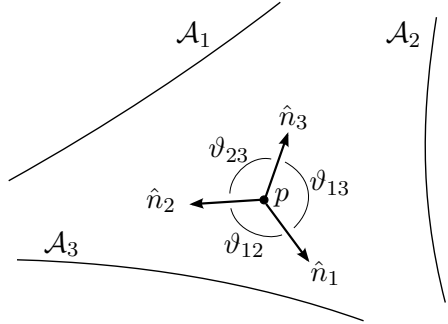}
\caption{Proof of Proposition~\ref{prop:bounds}: the angles $\vartheta_{ii'}$ are bounded away from $\pi$ for $p$ lying in a compact subset of the interior of the convex core.}
\label{fig:propbounds2}
\end{figure}

Now consider a $(j,u)$-equivariant convex field $X$ and a vector $x\in X(p)$.
For $1\leq i\leq 3$, the vector $\gamma_i \bullet x  = j(\gamma_i)_* x + u(\gamma_i)(q_i) \in T_{q_i} \HH^2$ also belongs to~$X$.
By \eqref{eqn:boundtheta},
$$d'(x, \gamma_i \bullet x) \geq 2 \langle x, \hat n_i \rangle \cos \theta_i  - Q.$$ 
However, $d(p, q_i) \leq \lambda_i + 2 r_i$.
If $X$ is $k$-lipschitz, it follows that 
$$2 \langle x, \hat n_i \rangle \leq \frac{k (\lambda_i + 2 r_i) + Q}{\cos \theta_i}.$$
Now, for some $i,i'$ the vector $x$ makes an angle $\leq \vartheta_{ii'}/2$ with $\hat n_i$.
Thus 
$$\| x \| \leq \max_{i,i'} \frac{k (\lambda_i + 2 r_i) + Q}{\cos \theta_i \cos\!\big(\frac{\vartheta_{ii'}}{2}\big)}\,.$$
Since $\theta_i=\arctan \frac{\coth\lambda_i/2}{\sinh r_i}$ is bounded away from $\frac{\pi}{2}$ when $r_i$ is bounded away from~$0$, and $Q$ is bounded by a continuous function of $p\in\specialC$, this gives a uniform bound in an open neighborhood of~$p$ where $r_i$ is bounded away from~$0$ and $\vartheta_{ii'}$ is bounded away from~$\pi$.
\end{proof}

\subsection{Standard fields in the funnels}\label{subsec:standard-funnels}

We now focus on the exterior of the convex core, namely on the so-called \emph{funnels} of the hyperbolic surface (or orbifold) $j(\Gamma)\backslash\HH^2$.
We define explicit vector fields in the funnels which will be used in Section~\ref{sec:prooflamin} to extend a $k$-lipschitz convex field on the interior of the convex core to the entire surface. 

We work explicitly with Fermi coordinates.
Let $\A$ be an oriented geodesic line in~$\HH^2$ and $p$ a point on~$\A$.
For $q \in \HH^2$, let $p'$ be the point of $\A$ closest to~$q$; we define $\xi(q)\in\RR$ and $\eta(q)\in\RR$ to be the signed distance from $p$ to~$p'$ and from $p'$ to~$q$, respectively. 
The numbers $\xi(q)$ and~$\eta(q)$ are called the \emph{Fermi coordinates} of~$q$ with respect to $(\A,p)$.
Note that the Fermi coordinate map $F : \RR^2\overset{\scriptscriptstyle\sim}{\rightarrow}\HH^2$ sends $\RR \times \{0\}$ isometrically to~$\A$, and $\{\xi\} \times \RR$ isometrically to the geodesic line orthogonal to $\A$ at $F(\xi,0)$.
The curves $F(\RR \times\{ \eta\})$, which lie at constant (signed) distance $\eta$ from~$\A$, are called \emph{hypercycles}.

For $k,r\in\RR$, we define the \emph{$(k,r)$-standard} vector field with respect to $(p,\A)$ by 
\begin{equation}\label{eqn:standardfield}
X^k_r :\ F(\xi,\eta) \,\longmapsto\, k\xi\frac{\partial F}{\partial \xi}+ r\eta \frac{\partial F}{\partial \eta}.
\end{equation}
It is a smooth vector field on~$\HH^2$.

\begin{Proposition}\label{prop:standard-lipschitz}
Suppose that $r < \min(k, 0)$.
Then the vector field $X^k_r$ is $k$-lipschitz.
Further, for any $\eta_0 > 0$ there exists $\varepsilon > 0$ such that at any $p \in \HH^2$ with $d(p,\A) \geq \eta_0$, the local lipschitz constant satisfies $\lip_p(X^k_r) \leq k - \varepsilon$.
In particular, $d'_{X^k_r}(p,q)<k\,d(p,q)$ for all distinct $p,q\in F(\RR\times\RR_-^{\ast})$.
\end{Proposition}

\begin{proof}
For any tangent vector $y = a \frac{\partial F}{\partial \eta} + b \frac{\partial F}{\partial \xi} \in T_{F(\xi,\eta)}$, where $a,b,\xi,\eta\in\RR$, direct computation yields
\begin{equation}\label{eqn:nablaXkr}
\langle \nabla_y X^k_r, y \rangle = r a^2  + k b^2 \, \cosh^2\eta + r b^2 \eta \, \sinh\eta \,\cosh\eta.
\end{equation}
In particular, $\langle\nabla_y X^k_r,y\rangle\leq k (a^2 + b^2 \cosh^2 \eta)=k\| y\|^2$, hence by \eqref{eqn:lipsmooth} we have $\lip_{F(\xi,\eta)}(X^k_r)\leq\nolinebreak k$, which implies that $\lip(X^k_r)\leq k$ (Observation~\ref{obs:donkey}).
If~we assume that $|\eta|\geq\eta_0$, then~\eqref{eqn:nablaXkr} gives the more precise estimate $\langle \nabla_y X^k_r, y \rangle < \max\{ r, k+r\eta_0\tanh\eta_0\}\,\| y\|^2$, hence $\lip_{F(\xi,\eta)}(X^k_r)$ is uniformly bounded away from~$k$ by \eqref{eqn:lipsmooth}.
To deduce that $d'_{X^k_r}(p,q)<k\,d(p,q)$ for all $p,q\in F(\RR\times\RR_-^{\ast})$, we use Observation~\ref{obs:donkey}.(v).
\end{proof}

Now let $j: \Gamma \to G$ be convex cocompact and let $u: \Gamma \to \g$ be a $j$-cocycle.
For $\gamma \in \Gamma$ with $j(\gamma)$ hyperbolic, let $F : \RR^2 \to \HH^2$ be a Fermi coordinate map with respect to the translation axis $\A_{j(\gamma)}$ and let $X^k_r$ be the standard vector field given by \eqref{eqn:standardfield}.
If $u(\gamma)$ is an infinitesimal translation along~$\A_{j(\gamma)}$ (which we can always assume after adjusting $u$ by a coboundary), then $X^k_r$ is $(j|_{\langle\gamma\rangle}, u|_{\langle\gamma\rangle})$-equivariant if and only if $k = k_{\gamma} := {{\boldsymbol\alpha}_u(\gamma)}/{\lambda(j(\gamma))}$.
In the case that $\gamma$ is a peripheral element, we orient $\A_{j(\gamma)}$ so that $F(\RR \times \RR^*_-)$ is a component of the complement of the convex core; this region covers a funnel of the quotient $j(\Gamma)\backslash\HH^2$.

\begin{Definition}\label{def:standard-funnels}
\begin{itemize}
  \item For peripheral~$\gamma$, we say that a $(j|_{\langle\gamma\rangle},u|_{\langle\gamma\rangle})$-equi\-variant convex field $X$ on~$\HH^2$ is \emph{standard} in the funnel $F(\RR \times \RR^*_-)$ if there exists $\eta<0$ such that $X$ coincides on $F(\RR \times (-\infty,\eta))$, up to addition of a Killing field, with a vector field of the form~$X^k_r$.
  \item We say that a $(j,u)$-equivariant convex field $X$ on~$\HH^2$ is \emph{standard in the funnels} if it is standard in every funnel.
\end{itemize}
\end{Definition}

The following proposition, in combination with the lipschitz extension theory of Section~\ref{sec:extlipschitz}, will be used in Section~\ref{sec:prooflamin} to extend $k$-lipschitz convex fields defined on the interior of the convex core to $k$-lipschitz convex fields defined over all of~$\HH^2$.

\begin{Proposition}\label{prop:fixup-boundary}
For $\gamma \in \Gamma$ with $j(\gamma)$ hyperbolic, let $X$ be a locally bounded, $(j|_{\langle\gamma\rangle}, u|_{\langle\gamma\rangle})$-equivariant convex field defined over a $j(\langle\gamma\rangle)$-invariant subset $\Omega\neq\emptyset$ of~$\HH^2$. 
\begin{enumerate}
  \item We have $\lip(X)\geq k_{\gamma} :={{\boldsymbol\alpha}_u(\gamma)}/{\lambda(j(\gamma))}$.
  \item Suppose $\Omega$ is a hypercycle $F(\RR \times \{\eta_0\})$ with $\eta_0>0$, and let $N:= F(\RR \times (-\infty, \eta_1])$ for some $\eta_1 < 0$.
  There is an extension $Y$ of $X$ to $\Omega \cup N$ such that
  \begin{itemize}
    \item $Y$ is standard on~$N$ (in particular, $\lip_N(Y)\leq k_{\gamma}$),
    \item $\lip(Y)=\lip(X)$,
    \item $d'_Y(p,q) <k_{\gamma} \,d(p,q)$ for all $p\in N$ and $q\in\Omega$.
  \end{itemize}
\end{enumerate}
\end{Proposition}

Note that unlike in Section~\ref{subsec:computingMarginv}, here we do not assume $X$ to be defined over $\A_{j(\gamma)}$, hence Proposition~\ref{prop:fixup-boundary}.(1) does not follow directly from \eqref{eqn:lip-marg1}.

\begin{proof}
We first prove~(1).
Without loss of generality, we may assume that $\Omega$ is compact modulo $j(\langle\gamma\rangle)$.
Up to adjusting $u$ by a coboundary, we may assume that $u(\gamma)$ is an infinitesimal translation along the axis~$\A_{j(\gamma)}$.
Set $k := k_{\gamma}$ and fix $r<\min(0,k)$.
The convex field $X-X^k_r$ is $j|_{\langle\gamma\rangle}$-invariant and locally bounded, hence globally bounded on~$\Omega$: there exists $b>0$ such that $\| (X - X^k_r)(\Omega)\| < b$ (notation \eqref{eqn:normconvexfield}).
For $p,q\in\Omega$,
$$\frac{d'_X(p,q)}{d(p,q)} \geq \frac{d'_{X^k_r}(p,q)}{d(p,q)} - \frac{2b}{d(p,q)}.$$
However, for points $p,q \in \Omega$ further and further apart on fixed hypercycles, the ratio $d'_{X^k_r}(p,q)/d(p,q)$ limits to~$k$.
Indeed, if $p$ and~$q$ both belong to $\Omega\cap\A_{j(\gamma)}$, then $d'_{X^k_r}(p,q) = k\,d(p,q)$.
Otherwise, note that if $p$ and~$q$ are very far apart, the segment $[p,q]$ spends most of its length close to the axis~$\A_{j(\gamma)}$; we can then conclude using \eqref{eqn:dintegral} and the uniform continuity of $\newf'_{X^k_r}$ near~$\A_{j(\gamma)}$.
As a consequence, $\lip(X)\geq k$, proving~(1).

For~(2), choose $R< r - \frac{b}{|\eta_1 \tanh \eta_1|}$ and define $Y$ to be $X$ on~$\Omega$ and $X^{k}_{R}$ on~$N$.
Then $Y$ is $k$-lipschitz on~$N$ by Proposition~\ref{prop:standard-lipschitz}.
Thus we only need to check that $d'_Y(p,q) <k\,d(p,q)$ for all $p\in N$ and $q\in\Omega$.
Let $\theta \leq \pi/2$ be the angle at~$p$ between $[p,q]$ and $\frac{\partial F}{\partial \eta}(p)$.
Then $\cos \theta \geq \tanh |\eta_1|$ by a standard trigonometric formula.
In particular, $d'\left(\frac{\partial F}{\partial \eta}(p),\underline{0}(q)\right)\leq - \tanh |\eta_1|$.
Then for $x \in X(q)$,
\begin{eqnarray*}
 d'(Y(p),x) &=& d'(X^{k}_{R}(p),x) \\
& = & d'\left( X^{k}_{r}(p)+(R-r)\,\eta(p)\,\dfrac{\partial F}{\partial \eta}(p)\:,\: X^{k}_{r}(q)+x-X^{k}_{r}(q) \right) \\
& = & d'\big (X^{k}_{r}(p),X^{k}_{r}(q)\big )~+~d'\big (\underline{0}(p),x-X^{k}_{r}(q)\big ) \\
& & +\ d'\left( (R-r)\,\eta(p)\,\frac{\partial F}{\partial \eta}(p),\underline{0}(q)\right) \\
& \leq & k \ d(p,q) + b + (R-r) \eta_1 (-\tanh |\eta_1|),
\end{eqnarray*}
which is $<k\,d(p,q)$ by choice of~$R$. 
\end{proof}

\section{Extension of lipschitz convex fields}\label{sec:extlipschitz}

To produce the lipschitz convex fields promised in Theorem~\ref{thm:laminations}, we will need to extend lipschitz convex fields that are only defined on part of~$\HH^2$.
This falls into the subject of Lipschitz extension, a topic initiated by Kirsz\-braun's theorem \cite{kir34} to the effect that a partially-defined, $K$-Lipschitz map from a Euclidean space to another always extends (with the same Lipschitz constant~$K$) to the whole space.
An analogue of Kirszbraun's theorem in~$\HH^n$ (when $K\geq 1$) was proved by Valentine \cite{val44}.
We need a generalization of this in several directions:
\begin{itemize}
  \item at the infinitesimal level ($k$-lipschitz fields, not $K$-Lipschitz maps);
  \item with local control, \ie information on which pairs of points achieve the lipschitz constant (eventually, pairs belonging to a leaf of some lamination for $k\geq 0$);
  \item in an equivariant context.
\end{itemize}
Negative curvature is responsible for the sharp divide taking place at $K=1$ (resp.\ $k=0$).
The ``macroscopic case'' of maps from $\HH^n$ to~$\HH^n$, in an equivariant context and with a local control of the Lipschitz constant, was treated in \cite{gk12}, refining \cite{kasPhD}.
For context, we quote:

\begin{Theoremwithref}[{\cite[Th.\,1.6 \& 5.1]{gk12}}]\label{thm:macro-kirszbraun}
Let $\Gamma$ be a discrete group and $j,\rho : \Gamma\rightarrow\mathrm{Isom}(\HH^n)$ two representations with $j$ convex cocompact.
Suppose $\rho(\Gamma)$ does \emph{not} have a unique fixed point in $\partial_{\infty}\HH^n$. 
Let $\specialC\neq\emptyset$ be a $j(\Gamma)$-invariant cocompact subset of~$\HH^n$ and $\varphi : \specialC\rightarrow \HH^n$ a $(j,\rho)$-equivariant Lipschitz map with Lipschitz constant~$K$.
Then there exists an equivariant extension $f :\nolinebreak\HH^n\rightarrow\HH^n$ of~$\varphi$ with
$$\left \{ \begin{array}{lll}  \Lip(f) < 1 & \text{if } K<1,\\
\Lip(f) = K & \text{if } K\geq 1.
  \end{array} \right .$$
Moreover, if $K>1$ (resp.\ $K=1$), then the \emph{relative stretch locus} $E_{\specialC,\varphi}(j,\rho)$~is nonempty, contained in the convex hull of $\specialC$, and is (resp.\ contains) the union of the stretch locus of~$\varphi$ and of the closure of a geodesic lamination $\tilde{\mathscr{L}}$ of $\HH^n\smallsetminus\specialC$ that is maximally stretched by any $K$-Lipschitz $(j,\rho)$-equivariant extension $f : \HH^n\rightarrow\HH^n$ of~$\varphi$.
\end{Theoremwithref}

By \emph{maximally stretched} we mean that distances are multiplied by~$K$ on every leaf of~$\tilde{\mathscr{L}}$.
The \emph{stretch locus} of~$\varphi$ is by definition the set of points~$p\in\specialC$ such that the Lipschitz constant of~$\varphi$ restricted to $\mathcal{U}\cap\specialC$ is $K$ (and no smaller) for all neighborhoods $\mathcal{U}$ of $p$ in~$\HH^n$.
The \emph{relative stretch locus} $E_{\specialC,\varphi}(j,\rho)$ is the set of points $p\in\HH^n$ such that the Lipschitz constant of any $K$-Lipschitz equivariant extension of~$\varphi$ is~$K$ on any neighborhood of $p$ in~$\HH^n$.

Note that when $K=1$, a $K$-Lipschitz equivariant extension of~$\varphi$ may be forced to be isometric on a larger set (for instance, if $j=\rho$ and $\varphi=\mathrm{Id}_{\specialC}$, then $f$ must be the identity map on the convex hull of~$\specialC$).
Theorem~5.1 of \cite{gk12} describes precisely which pairs of points $p,q\in \HH^n$ achieve $d(f(p),f(q))=Kd(p,q)$ for all $K$-Lipschitz equivariant maps~$f$ when $K=1$, and also allows for geometrically finite $j(\Gamma)$ (with parabolic elements) when $K\geq 1$.

To save space and focus on applications, we will treat the microscopic analogue less thoroughly here, restricting in particular to~$\HH^2$ with no parabolic elements in $j(\Gamma)$, and to special ``$\specialC$'' and ``$\varphi$''.
However all macroscopic ideas of \cite{gk12} should generalize.

To work out a microscopic analogue of Theorem~\ref{thm:macro-kirszbraun}, the starting point is to consider a sequence of equivariant convex fields with lipschitz constants converging to the infimum.
Note however that to an equivariant field, we can always add an invariant field pointing strongly towards the convex core, without increasing the lipschitz constant: this is why minimizing sequences will usually not converge outside the convex core.
We therefore resort to imposing a ``standard'' form (as in Section~\ref{subsec:standard-funnels}) to the convex field inside the funnels, and minimize under that constraint.

In Section~\ref{subsec:local-control}, we prove a local lipschitz extension theorem with a local control: this is where $k$-stretched lines appear for the first time.
In Section~\ref{subsec:equivariant-extension}, we turn this into a global equivariant extension result for vector fields defined away from the convex core (typically the standard vector fields of Section~\ref{subsec:standard-funnels}).
These tools will be used in Section~\ref{sec:prooflamin} to prove Theorem~\ref{thm:laminations}.

\subsection{Local lipschitz extensions of convex fields with a local control}\label{subsec:local-control}

We say that a convex field $Y$ is an \emph{extension} of a convex field~$X$ if $Y\supset X$ as subsets of $T\HH^2$ and if $X(U')=Y(U')$ for any \emph{open} set $U'\subset\HH^2$ on which $X$ is defined.
This means that $Y$ is defined at least on the largest domain where $X$ is and that $X$ and~$Y$ coincide on the interior of this domain.

The following key theorem lets us extend convex fields locally without loss in the lipschitz constant. 

\begin{Theorem}\label{thm:compact-extension}
Let $\specialC\subset\HH^2$ be a compact set and $X$ a compact (\ie bounded) convex field defined over~$\specialC$.
Suppose $X$ is $k$-lipschitz with $k\geq 0$.
Then $X$ admits a $k$-lipschitz compact extension to the convex hull $\CH(\specialC)$.
\end{Theorem}
The proof will be simplified by the following lemma, which is also useful for several other arguments in the paper.
We use the notation \eqref{eqn:normconvexfield}.

\begin{Lemma}\label{lem:convexbounds}
Consider the vertices $p_1, \ldots, p_m \in\HH^2$ of a convex polygon $\Pi$, vectors $x_i\in\nolinebreak T_{p_i}\HH^2$, and a compact subset $\specialC'$ contained in the interior of $\Pi$.
For any $k\in\RR$, there exists $R >0$ such that $\|Y(\specialC')\|\leq R$ for any $k$-lipschitz convex field extension $Y$ of $\{x_1,\ldots,x_m\}$ to $\specialC'\cup \{p_1,\ldots, p_m\}$.
\end{Lemma}

\begin{proof}
Consider the vectors $\log_p(p_i)\in T_p\HH^2$ pointing towards~$p_i$.
By compactness of~$\specialC'$, the maximum angle $2\vartheta$ between $\log_p(p_i)$ and $\log_p(p_{i'})$ for $1\leq i,i'\leq m$ and $p\in\specialC'$ is $<\pi$.
For any $p\in\specialC'$ and $x \in T_p\HH^2$, there exists $1\leq i\leq m$ such that the angle between $x$ and $\log_p(p_i)$ lies between zero and~$\vartheta$.
By Remark~\ref{rem:d'proj},
$$d'(x,x_i) \geq \| x \| \cos \vartheta - \| x_i \|.$$
Therefore, if $x=Y(p)$ for some $k$-lipschitz extension~$Y$ of $\{x_1,\ldots,x_m\}$, then
\begin{equation} \label{eqn:convbounds}
\| x \| \leq \frac{\| x_i \| + k\,d(p, p_i)}{\cos \vartheta} \leq \frac{\max_{i'}(\| x_i\|) + |k| \max_{i',p}(d(p,p_{i'}))}{\cos \vartheta},
\end{equation}
where $1\leq i'\leq m$ and $p$ ranges over~$\specialC'$.
This gives the desired uniform~bound.
\end{proof}

\begin{proof}[Proof of Theorem~\ref{thm:compact-extension}]
Let $p_1,\dots,p_m \in \HH^2$ be the vertices of a convex polygon $\Pi$ containing $\CH(\specialC)$ in its interior.
Since $X$ is bounded we may extend $X$ to each $p_i$ by choosing a large vector $x_i$ pointing into $\Pi$  along the bisector of the angle at $p_i$, so that the extension remains $k$-lipschitz on $\specialC\cup \{p_1,\ldots,p_m\}$.
As we extend~$X$, maintaining the $k$-lipschitz property, to points of $\CH(\specialC)$, these helper points, via Lemma~\ref{lem:convexbounds}, guarantee that our extension will be bounded: $\|X(\CH(\specialC))\| \leq R$.

First we extend $X$ to a single point $p$ of $\CH(\specialC) \smallsetminus \specialC$.
To choose $X(p)$ optimally with respect to the lipschitz property, consider the map $\varphi_p : T_p\HH^2\rightarrow\RR$ defined by
$$\varphi_p(x):=\sup_{(q,y)\in X(\specialC)}\, \frac{d'(x, y)}{d(p,q)}\,.$$
By Lemma~\ref{lem:convexbounds}, the function~$\varphi_p$ is proper.
In particular, it has a minimum~$k'$, achieved at a vector $x_0\in T_p\HH^2$ with $\|x_0 \| \leq R$. 
For any $(q,y)\in X(\specialC)$, the ratio $d'(x, y)/d(p,q)$ is an affine function of $x\in T_p\HH^2$, of gradient intensity $1/d(p,q)$.
Since affine functions are convex, so is~$\varphi_p$.
Let us show that $k'\leq k$. 
Since the convex field $X$ is compact in $T\HH^2$, the supremum defining $\varphi_p$ is achieved.
Let $\{ (q_i, y_i) \,|\, i\in I\}$ be the (compact) set of all vectors $(q,y)\in X(\specialC)$ such that $d'(x_0,y)=k'\,d(p,q)$.
Suppose, for contradiction, that the convex hull of the $q_i$ does not contain~$p$.
Then there is an open half-plane $\mathscr{H}\subset\HH^2$ that is bounded by a line through $p$ and contains all the~$q_i$.
By compactness,
$$\max_{(q,y)\in X(\specialC\smallsetminus\mathscr{H})} \frac{d'(x_0, y)}{d(p,q)} < k'.$$
Since the gradient intensities $1/d(p,q)$ for $q\in\specialC$ are bounded from above, it follows that $\varphi_p(x_0+\xi)<k'$ for any short enough vector $\xi$ pointing orthogonally into~$\mathscr{H}$: a contradiction with the minimality of~$k'$.
Therefore, $p$ lies in the convex hull of the~$q_i$.
There are two cases to consider.

\noindent
$\bullet$ \emph{Case (i):} Suppose $p$ lies on a segment $[q_i,q_{i'}]$ with $i,i'\in I$.
Then
\begin{align*}
k\,d(q_i,q_{i'}) \geq d'(y_i, y_{i'})
&= d'(y_i, x_0) + d'(x_0, y_{i'}) \\
&= k'\,d(q_i,p) + k'\,d(p,q_{i'})\\
&= k'\,d(q_i,q_{i'}),
\end{align*}
showing that $k' \leq k$. 

\noindent
$\bullet$ \emph{Case (ii):} If $p$ does not lie on such a segment, then it lies in the interior of a nondegenerate triangle $q_i q_{i'} q_{i''}$ with $i,i',i''\in I$.
We write $\{ i,i',i''\}=\{ 1,2,3\}$ to simplify notation.
By adding a Killing field, we may assume that $x_0 = \underline{0}(p)$ and that $y_1$ is parallel to the segment $[p,q_1]$, so that $y_1 = - k'\log_{q_1}(p)$.
The geodesic rays from~$p$ passing through each of $q_1,q_2,q_3$ divide $\HH^2$ into three connected components (see Figure~\ref{fig:tripod}).
There is a pair of distinct indices $a,b\in \{1,2,3\}$ such that $y_a$ and~$y_b$ point (weakly) away from the component bordered by the rays from $p$ through~$q_a$ and through~$q_b$.
Then $d'(y_a, y_b)\geq d'(\hat{y_a}, \hat{y_b})$ where $\hat{y_i}$ is the projection of~$y_i$ to the $[p,q_i]$ direction.
Note that $\hat{y_i} = - k' \log_{q_i}(p)$ by Remark~\ref{rem:d'proj}, because $d'(x_0,y_i)=k'\,d(p,q_i)$.
Now set $q_a^t := \exp_{q_a}(t \hat y_a)$ and $q_b^t := \exp_{q_b}(t \hat y_b)$.
Since the angle $\widehat{q_a p q_b}$ is different from $0$ and~$\pi$, the distance function $\psi : t\mapsto d(q_a^t, q_b^t)$ is strictly convex (a feature of negative curvature) and vanishes at $t = -1/k'$ as long as $k' > 0$ (if $k' \leq 0$, then we already have $k' \leq k$). Thus
\begin{align}\label{eqn:strictk'}
 k' & \,<\, \frac{\psi'(0)}{\psi(0)} \,=\, \frac{d'(\hat y_a, \hat y_b)}{d(q_a, q_b)} \,\leq\, \frac{d'(y_a, y_b)}{d(q_a, q_b)} \,\leq\, k.
\end{align}

\begin{figure}[h!]
{\centering
\def\svgwidth{10cm}
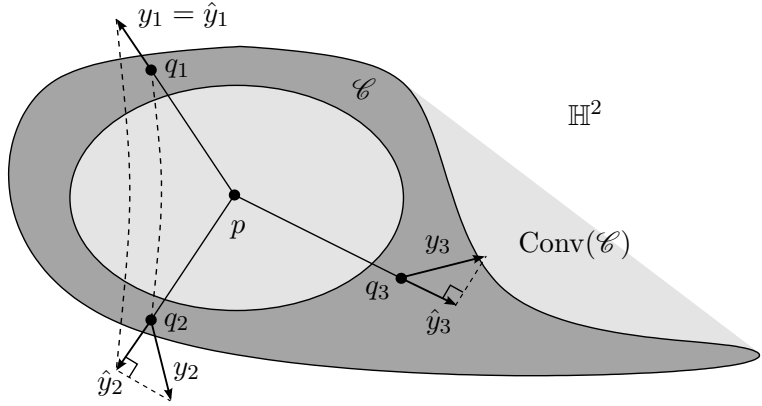}
\caption{In this illustration of case~(ii), the vectors $y_1$ and $y_2$ point weakly away from the sector $q_1 p q_2$, hence $d'(y_1,y_2) \geq d'(\hat y_1, \hat y_2)$. Next, $d'(\hat y_1, \hat y_2) > k' d(q_1,q_2)$ by convexity of the function $t\mapsto d(\exp_{q_1}(t\hat y_1),\exp_{q_2}(t \hat y_2))$. }
 \label{fig:tripod}
\end{figure}

We have shown that $X$ admits a $k$-lipschitz extension to $\specialC\cup \{p\}$.
Replacing $\specialC$ with $\specialC\cup \{p\}$, we can extend to a second point $p'$ of $\CH(\specialC)\smallsetminus\specialC$, then to a third, and eventually to a dense subset $\mathscr{S}$ of $\CH(\specialC)$.
We take our final extension $Y$ to be the fiberwise convex hull of the closure of $X(\mathscr{S})$ in $T\HH^2$.
That is, for any $p\in\CH(\specialC)$, we define $Y(p)$ to be the (closed) convex hull in $T_p\HH^2$ of all limits of sequences $(p_n,y_n)\in X(\mathscr{S})$ with $p_n\rightarrow p$.
Note that $Y(p)\neq \emptyset$ because $X(\mathscr{S})$ was bounded uniformly at the beginning of the proof.
By construction, $Y$ is closed in $T\HH^2$, and $k$-lipschitz because $d'$ is continuous.
It agrees with the original convex field~$X$ on the interior of~$\specialC$, but may have larger fibers above points of the boundary of $\specialC$ in~$\HH^2$.
\end{proof}

As mentioned at the beginning of Section~\ref{sec:extlipschitz}, we will also need a version of Theorem~\ref{thm:compact-extension} with locally improved lipschitz constant: it is given by the following proposition, which is a simple consequence of the proof of Theorem~\ref{thm:compact-extension}.
Recall the definition~\ref{def:lipX} of the local lipschitz constant $\lip_p(X)$.

\begin{Proposition}\label{prop:strict-extension}
Under the hypotheses of Theorem~\ref{thm:compact-extension}, for any point $p\in\CH(\specialC) \smallsetminus \specialC$, the convex field $X$ admits a $k$-lipschitz extension $Y$ to $\CH(\specialC)$ such that $\lip_p(Y)<k$, unless one of the following holds:
\begin{enumerate}
  \item \label{case1} $k > 0$ and $p$ lies in the interior of a geodesic segment $[q_1,q_2]$ with $q_1,q_2\in \specialC$ and $d'_X(q_1,q_2)=k\,d(q_1,q_2)$, the direction of this segment at~$p$ being unique;
  \item \label{case2} $k=0$ and $p$ lies in the convex hull of three (not necessarily distinct) points $q_1,q_2,q_3\in\specialC$ such that $X$ contains the restriction of a Killing field to $\{q_1,q_2,q_3\}$ (in particular, $d'_X(q_i, q_{i'}) = 0$ for all $1\leq i<\nolinebreak i'\leq\nolinebreak 3$).
\end{enumerate}
In case~(\ref{case2}), any $k$-lipschitz extension of~$X$ to $\CH(\specialC)$ restricts to a Killing field on the interior of the triangle $q_1q_2q_3$.
\end{Proposition}

The segments $[q_i,q_{i'}]$ are \emph{$k$-stretched} by~$X$ in the sense of Observation~\ref{obs:donkey}.

\begin{proof}
Fix $p\in\CH(\specialC) \smallsetminus \specialC$.
As in the proof of Theorem~\ref{thm:compact-extension}, we first define $Y(p)$ to be a vector at which $\varphi_p$ achieves its minimum $k'$.
If $k'<k$, then we can extend $Y$ as we wish in a continuous and locally $k'$-lipschitz way near $p$ without destroying the global $k$-lipschitz property, and then continue extending $Y$ to the rest of $\CH(\specialC)$ as in the proof of Theorem~\ref{thm:compact-extension}.
We need therefore only understand the case $k'=k$.

Suppose $k' = k > 0$.
Then the strict inequality \eqref{eqn:strictk'} shows that we cannot be in case~\textit{(ii)} in the proof of Theorem~\ref{thm:compact-extension}.
Therefore, we are in case~\textit{(i)}, and so $p$ lies in the interior of a $k$-stretched segment $[q_1,q_2]$ with $q_1,q_2\in \specialC$.
To see that (\ref{case1}) holds, suppose that $p$ lies in the interior of a $k$-stretched segment $[q_3,q_4]$ of a different direction, with $q_3,q_4\in\specialC$.
Then for $1\leq i\leq 4$, there are $y_i\in X(q_i)$ such that $d'(Y(p),y_i)=k\,d(p,q_i)$ and we may argue as in case~\textit{(ii)} of the proof of Theorem~\ref{thm:compact-extension} (with four directions instead of three) that $d'(y_a,y_b)>k\,d(q_a, q_b)$ for some $1\leq a<b\leq 4$, a contradiction.

Now suppose $k' = k = 0$.
If we are in case~\textit{(i)} in the proof of Theorem~\ref{thm:compact-extension}, then $p$ belongs to the interior of a $0$-stretched segment $[q_1,q_2]$ with $q_1,q_2\in\nolinebreak\specialC$; in particular, $X$ contains the restriction to $\{ q_1,q_2\}$ of a Killing field.
Suppose we are in case~\textit{(ii)}, \ie $p$ lies in the interior of a nondegenerate triangle $q_1q_2q_3$ such that for any $1\leq i\leq 3$ we have $d'(x_0,y_i)=0$ for some $x_0 \in T_p \HH^2$ and $y_i\in X(q_i)$.
Up to adding a Killing field, we may assume $x_0 = \underline{0}(p)$ and $y_1= \underline{0}(q_1)$.
Then, for $i \in \{2,3\}$, the component $\hat{y_i}$ of $y_i$ in the direction $[p,q_i]$ must be zero, since $d'(x_0, y_i) = d'(\underline{0}(p), \hat{y_i}) = 0$.
The component of $y_i$ orthogonal to $[p,q_i]$ must also be zero for each~$i$ or else $d'(y_a, y_b) > 0$ for some $1\leq a<b\leq 3$, contradicting that $X$ is $0$-lipschitz.
This proves that (\ref{case2}) holds.
In general, if $X$ contains the restriction of a Killing field $Z$~to $\{q_1,q_2,q_3\}$, then Proposition~\ref{prop:zeroconvex} shows that any $0$-lipschitz extension $Y$ of $X$ to the full triangle $q_1q_2q_3$ contains the restriction of $Z$ to $q_1q_2q_3$.~Further~,~$Y = Z$ on the interior of $q_1 q_2 q_3$ (apply Remark~\ref{rem:d'proj} with $q=q_i$ for $1\leq i\leq 3$).
\end{proof}

\subsection{Equivariant extensions of vector fields defined in the funnels}\label{subsec:equivariant-extension}

We next derive a technical consequence of Theorem~\ref{thm:compact-extension}, namely that equivariant lipschitz vector fields defined outside the convex core, with nice enough inward-pointing properties, can be extended equivariantly to all of~$\HH^2$.
This will be applied in the next section to standard vector fields (Definition~\ref{def:standard-funnels}). 

\begin{Proposition}\label{prop:equivariant-kirszbraun}
Let $\Gamma$ be a discrete group, $j : \Gamma \to G$ a convex cocompact representation, and $u: \Gamma \rightarrow \g$ a $j$-cocycle.
Let $\overline{U}\subset\HH^2$ be the preimage of the convex core of $j(\Gamma)\backslash\HH^2$, let $U_1$ be the open $1$-neighborhood of~$\overline{U}$, and let $N_1 := \HH^2 \smallsetminus U_1$ be its complement.
Let $X$ be a $(j,u)$-equivariant lipschitz vector field on~$N_1$ and set $k:=\lip(X)$.
Suppose that there exists $\varepsilon>0$ such that for all distinct $p,q\in N_1$,
\begin{itemize}
  \item[$(*)$] $d'_X(p,q)<k\,d(p,q)$ (strict inequality);
  \item[$(**)$] $\lip_p(X)\leq k-\varepsilon$;
  \item[$(*{*}*)$] if $p\in \partial N_1$ then $p$ has a neighborhood $V$ such that $X(V\cap N_1)$ admits a vector field extension to $V$ with lipschitz constant $<k$.
\end{itemize}
Then there exists a $(j,u)$-equivariant convex field $Y$ on~$\HH^2$, extending~$X$, such that
\begin{enumerate}
  \item if $k < 0$, then $\lip(Y)<0$;
  \item if $k \geq 0$, then $\lip(Y)=k$ and there is a $j(\Gamma)$-invariant geodesic lamination $\tilde{\mathscr{L}}$ in~$\overline{U}$ that is maximally stretched by~$Y$, in the sense that $d'_X(p,q) = k\,d(p,q)$ for any $p,q$ on a common leaf of~$\tilde{\mathscr{L}}$; in particular,
$$k \,=\, k_{\boldsymbol\alpha} \,:=\, \sup_{\gamma \text{ \rm{with} } \lambda (j(\gamma)) > 0}\ \frac{{\boldsymbol\alpha}_u(\gamma)}{\lambda(j(\gamma))}\,.$$
\end{enumerate}
\end{Proposition}

\begin{proof}
The idea of the proof is to construct an extension $Y$ of~$X$ that is in a certain sense ``optimal''.
Then the lamination $\tilde{\mathscr{L}}$ will arise as the union of the $k$-stretched segments of Proposition~\ref{prop:strict-extension}.(\ref{case1}) at points $p$ where $\lip_p(Y)=k$.
Considering segments with endpoints in~$N_1$ that spend most of their length near~$\tilde{\mathscr{L}}$, this will imply $\lip(Y)=k$. 

We first show that equivariant lipschitz extensions of~$X$ exist.
The following claim holds in general, independently of the regularity assumptions $(*),(**),(*{*}*)$, and even if $X$ is a convex field instead of a vector field.

\begin{Claim}
There exist $(j,u)$-equivariant lipschitz convex field extensions of $X$ to~$\HH^2$ (possibly with very bad lipschitz constant).
\end{Claim}

\begin{proof}
Let $B_1,\dots,B_m$ be open balls of~$\HH^2$ such that the sets $j(\Gamma)\cdot B_i$ for $1\leq i\leq m$ cover~$U_1$.
We take them small enough so that $j(\gamma)\cdot B_i$ is either equal to or disjoint from~$B_i$ for all $\gamma\in\Gamma$ and $1\leq i\leq m$.
Let $(\psi_i)_{1\leq i\leq m}$ be a $j(\Gamma)$-invariant partition of unity  on~$U_1$, with each $\psi_i$ supported in $j(\Gamma) \cdot B_i$.
We require the restriction of $\psi_i$ to~$B_i$ to be Lipschitz.
For any~$i$, Theorem~\ref{thm:compact-extension} gives a compact extension $Z_i$ of $X|_{B_i \cap N_1}$ to~$B_i$ with $\lip_{B_i}(Z_i) \leq \max\{k,0\}$; in case the stabilizer $\Gamma_i\subset\Gamma$ of~$B_i$ is nontrivial, we can assume that $Z_i$ is $(j|_{\Gamma_i},u|_{\Gamma_i})$-equivariant after replacing it with $\frac{1}{\#\Gamma_i} \sum_{\gamma\in\Gamma_i} \gamma\bullet Z_i$ (notation \eqref{eqn:u-action}), using Observation~\ref{obs:donkey}.
We then extend it to a $(j,u)$-equivariant convex field $Z_i$ on $j(\Gamma)\cdot B_i$.
The extension
$$Z := X \cup \sum_{i=1}^m \psi_i Z_i.$$
of~$X$ is $(j,u)$-equivariant (Observation~\ref{obs:donkey}).
Let us check that $Z$ is lipschitz.
By subdivision and equivariance (Observation~\ref{obs:donkey}), we only need to check that $Z$ is lipschitz on each of the balls $B_{i'}$ for $1\leq i'\leq m$.
Consider two distinct points $p,q \in B_{i'}$ and vectors $z_p \in Z(p)$ and $z_q \in Z(q)$.
We can write
$$z_p = \sum_{i=1}^m \psi_i(p) x_i \quad\quad\text{and}\quad\quad z_q = \sum_{i=1}^m \psi_i(q) y_i$$
where $x_i \in Z_i(p)$ and $y_i \in Z_i(q)$.
By Observation~\ref{obs:donkey},
\begin{align*}
d'(z_p, z_q) &= d'\left(\sum_{i=1}^m \psi_i(p) x_i, \sum_{i=1}^m \psi_i(q) y_i\right) \\
&= d'\left(\sum_{i=1}^m \psi_i(p) x_i, \sum_{i=1}^m \psi_i(p) y_i\right) + d'\left ( \underline{0}(p), \sum_{i=1}^m \big(\psi_i(q)- \psi_i(p)\big)\,y_i\right) \\
&\leq \sum_{i=1}^m \psi_i(p) \, d'(x_i, y_i)  + \left \| \sum_{i=1}^m \big(\psi_i(q)- \psi_i(p)\big)\,y_i\right \| \\
&\leq \sum_{i=1}^m \psi_i(p) \, \lip_{B_{i'}}(Z_i) \, d(p,q) + \sum_{i=1}^m \Lip(\psi_i) \, \| y_i\| \, d(p,q)\\
&\leq \left( \sup_{1\leq i\leq m} \lip_{B_{i'}}(Z_i) + \Big(\sup_{1\leq i\leq m} \Lip(\psi_i)\Big) \sum_{i=1}^m \sup_{z_i\in Z_i(B_{i'})} \| z_i\| \right) \, d(p,q),
\end{align*}
where the last supremum is $<+\infty$ because $B_{i'}$ meets only finitely many $j(\Gamma)$-translates of~$B_i$ and $Z_i(j(\gamma)\cdot B_i)$ is compact for all $\gamma\in\Gamma$.
Thus $d'_Z(p,q)/d(p,q)$ is uniformly bounded for $p,q \in B_{i'}$ and $Z$ is lipschitz.
\end{proof}

Returning to the proof of Proposition~\ref{prop:equivariant-kirszbraun}, let $k^* \in [k,+\infty)$ be the infimum of lipschitz constants over all $(j,u)$-equivariant convex field extensions of $X$ to $\HH^2$.
If $k^* < 0$, then $k < 0$; we may choose an extension $Y$ with $\lip(Y)$ arbitrarily close to~$k^*$, in particular with $\lip(Y)<0$.
This proves~(1).

From now on, we assume $k^* \geq 0$.
Let $(Z_n)_{n\in\NN}$ be a sequence of $(j,u)$-equivariant extensions of~$X$ with $\lip(Z_n)\rightarrow k^*$.
Note that $U_1$ is covered by the $j(\Gamma)$-translates of some polygon with vertices in~$N_1$.
Therefore, by Lemma~\ref{lem:convexbounds} (using Observation~\ref{obs:donkey}.(iv) and~(vii)), the convex fields $Z_n$ are uniformly bounded over any compact set.
By Proposition~\ref{prop:limitofconvex}, we may extract a subsequence which is Hausdorff convergent to a convex field $Z_{\infty}$ on~$\HH^2$ with $\lip(Z_{\infty}) = k^*$.
Thus the set $\mathcal Z$ of $(j,u)$-equivariant convex field extensions of $X$ to $\HH^2$ with minimal lipschitz constant~$k^*$ is nonempty.
For $Z \in \mathcal Z$, we set
$$E_Z := \{p \in \HH^2 \ |\ \lip_p(Z) = k^*\} .$$
It is a closed subset of~$\HH^2$, as the function $p\mapsto\lip_p(Z)$ is upper semicontinuous, and it is $j(\Gamma)$-invariant (Observation~\ref{obs:donkey}).
We define the \emph{stretch locus} relative to~$X$ to be
$$E := \bigcap_{Z \in \mathcal Z} E_Z.$$
We claim that there exists $Y\in\mathcal Z$ such that $E=E_Y$.
Indeed, for any $p \in \HH^2 \smallsetminus E$ we can find a convex field $Z_p \in \mathcal Z$ and a neighborhood $V_p$ of $p$ in~$\HH^2$ such that $\delta_p :=  k^* - \lip_{V_p}(Z_p) > 0$.
Let $\{V_{p_i}\}_{i\in\NN^{\ast}}$ be a countable set~of~such neighborhoods such that $\HH^2 \smallsetminus E = \bigcup_{i\in\NN^{\ast}} V_{p_i}$.
The equivariant~convex~field
$$Y = \sum_{i=1}^{+\infty}\, 2^{-i} Z_{p_i}$$
satisfies $\lip_{V_{p_i}}(Y)\leq k^* - 2^{-i} \delta_{p_i}$ for all~$i$, hence $E_Y = E$.
We next study the structure of~$E$.

By $(**)$, since $k\leq k^*$, the set $E$ is contained in~$\overline{U}_1$.
Moreover, $E$ is not empty: otherwise we could cover $\overline U_1$ (which is compact modulo $j(\Gamma)$) by the $j(\Gamma)$-translates of finitely many open sets $V_{p_i}$ for which $\lip_{V_{p_i}}(Y) < k^*$.
Since the local lipschitz constant in~$N_1$ is also bounded by $k -\varepsilon \leq k^* - \varepsilon$ by $(**)$, it would then follow by $j(\Gamma)$-invariance and subdivision (Observation~\ref{obs:donkey}) that $\lip(Y)< k^*$, a contradiction.
Thus $E$ contains a point of~$\overline{U}_1$.

In fact, $E$ must contain a point of~$U_1$.
Indeed, let us show that if $E$ contains a point $p$ of $\partial U_1 = \partial N_1$, then there exists $q\in U_1$ such that
\begin{equation} \label{eqn:stretchisfromX}
d'_Y(p,q) = k^*\,d(p,q) .
\end{equation}
Let $B$ be a small ball centered at~$p$, of radius $r>0$, and let $A$ be a thin annulus neighborhood of $\partial B$ in~$B$.
We have $\lip_{(B\cap N_1)\cup A}(Y)\leq\lip(Y)=k^*$.
Suppose by contradiction that
\begin{equation}\label{eqn:nostretchedradius}
\sup_{q\in A}\, \frac{d'_Y(p,q)}{d(p,q)} < k^* .
\end{equation}
By assumption $(*{*}*)$ on~$X$, if $r$ is sufficiently small then $X(B\cap\nolinebreak N_1)$ admits a vector field extension $X'$ to~$B$ with $\lip(X')<k\leq k^*$.
If $B'\subset B$ is another ball centered at~$p$, of radius $r'\ll r$ small enough, then \eqref{eqn:nostretchedradius} and the continuity of the vector field~$X'$ imply that $d'(X'(p'),y_q)<k^*d(p',q)$ for all $p'\in B'$, all $q\in A$, and all $y_q\in Y(q)$.
Then the convex field defined over $\specialC:=(B\cap N_1)\cup A \cup B'$ that agrees with $Y$ over $A$ and $B\cap N_1$ and with $X'$ over $B'$ is $k^*$-lipschitz (see Figure~\ref{fig:terminal}).
Applying Theorem~\ref{thm:compact-extension} to $\specialC$, we find a $k^*$-lipschitz convex field $Y'$ on the ball $B=\CH(\specialC)$ that contains $Y(A)$ and $X(B\cap N_1)$ and satisfies $\lip_p(Y')=\lip_p(X')<k\leq k^*$; we can extend it to $j(\Gamma)\cdot B$ in a $(j,u)$-equivariant way, and then to~$\HH^2$ by taking $Y'=Y$ on $\HH^2\smallsetminus j(\Gamma)\cdot B$.
Using subdivision at points of $j(\Gamma)\cdot A$, we see that $\lip(Y')\leq k^*$, which contradicts the fact that $p\in E$.
Thus \eqref{eqn:nostretchedradius} is false, and \eqref{eqn:stretchisfromX} holds for some $q\in U_1$, which implies $q\in E$.

\begin{figure}[h!]
{\centering
\def\svgwidth{8cm}
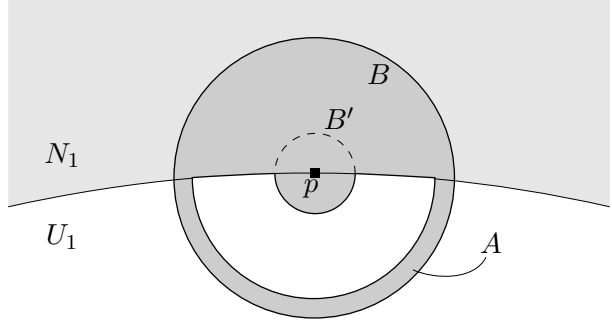}
\caption{Definition of the region $\specialC = (B \cap N_1) \cup A \cup B'$.}
\label{fig:terminal}
\end{figure}

Let us now prove that $E$ contains the lift of a $k^*$-stretched lamination contained in the convex core.

Assume first that $k^* > 0$.
Consider a point $p \in E \cap U_1$.
Let $Y'$ be the convex field obtained from $Y$ by simply removing all vectors above a small ball $B_p \subset U_1$ centered at~$p$, so that $Y'$ is defined over $\HH^2 \smallsetminus B_p$.
Proposition~\ref{prop:strict-extension}, applied to the restriction of $Y'$ to a small neighborhood of $\partial B_p$, implies that $p$ lies on a unique $k^*$-stretched segment $[q,q']$ with $q,q'\in\partial B_p$ (or else $Y'$ could have been extended with a smaller lipschitz constant at~$p$, and similarly at each $j(\gamma)\cdot p$ in an equivariant fashion, contradicting the fact that $p\in E$). 
This applies to all points $p \in E \cap U_1$, and so $E \cap U_1$ is a union of geodesic segments.
Moreover, the direction of the $k^*$-stretched segment $[q,q']$ at $p$ is unique for each $p$, and its length $2\,d(p,q)$ (the diameter of the ball $B_p$) can be taken to be bounded from below by a continuous, positive function of $p\in U_1$: it follows that any segment of $U_1$ that contains a $k^*$-stretched subsegment is $k^*$-stretched.
Note that one or more of these $k^*$-stretched (partial) geodesics in $E$ may have an endpoint in $\partial\overline{U}_1$; however, they cannot have two endpoints in $\partial\overline{U}_1$, by the assumption $(*)$.
Thus any (partial) geodesic $\ell\subset E$ descends to a simple (partial) geodesic in $j(\Gamma)\backslash\HH^2$ which, at least in one direction, remains in $j(\Gamma) \backslash U_1$ and accumulates on a geodesic lamination~$\mathscr{L}$.
This lamination~$\mathscr{L}$ must be contained in the convex core.
By closedness of~$E$, its preimage $\tilde{\mathscr{L}}\subset\HH^2$ is also part of~$E$, and each leaf of $\tilde{\mathscr{L}}$ is $k^*$-stretched.

Next, assume $k^*=0$.
Consider a point $p \in E \cap U_1$ and let $Y'$ be the convex field obtained from $Y$ by removing all vectors above a small ball $B_p \subset U_1$ around~$p$, so that $Y'$ is defined over $\HH^2 \smallsetminus B_p$.
Proposition~\ref{prop:strict-extension} applied to $Y'$ implies that either:
\begin{itemize}
  \item[$(i)$] $p$ lies on a segment $[q,q']$ with $q,q' \in \partial B_p$ and $d'_Y(q,q') = 0$; or
  \item[$(ii)$] $p$ lies in the convex hull of three distinct points $q_1,q_2,q_3 \in \partial B_p$ with $d'_Y(q_i, q_{i'}) = 0$ for all $1\leq i <i'\leq 3$.
\end{itemize} 
In case~$(i)$ the segment $[q,q']$ is $0$-stretched by $Y$, and in case~$(ii)$ the restriction of $Y$ to the interior of the triangle $q_1 q_2 q_3$ is a Killing field.
In both cases there is a $0$-stretched segment in~$E$ with midpoint~$p$, and we can bound from below the length of this segment by a continuous function of $p\in U_1$, \eg half the radius of $B_p$ (in case~$(ii)$, just pick the direction of the segment to be $pq_i$ where $q_i$ is at the smallest angle of the triangle $q_1 q_2 q_3$).
Now we must conclude that $E$ contains a geodesic lamination.
Whenever there are two intersecting, $0$-stretched, open line segments of $E$ containing~$p$, the $0$-lipschitz convex field $Y$ restricts to a Killing field on the union of these segments and in fact on their convex hull, so that points near $p$ satisfy case~$(ii)$.
So, if case~$(ii)$ never happens, then the germ of $0$-stretched segment through $p$ is unique, and we may proceed exactly as when $k^* > 0$.
If case~$(ii)$ does occur, then $E$ has an interior point near which $Y$ is a Killing field, which may be assumed to be $\underline{0}$ up to adjusting $u$ by a coboundary.
The set $E_0:=\{ p\in\HH^2 \,|\, \underline{0}(p)\in Y\}$ is convex by Proposition~\ref{prop:zeroconvex}, and contained in~$E$.
On the interior of~$E_0$, the convex field $Y$ coincides with~$\underline{0}$.
Consider $p\in \partial E_0 \smallsetminus \partial U_1$.
Points approaching $p$ from the interior of $E_0$ are midpoints of $0$-stretched segments, necessarily included in $E_0$, whose lengths are bounded from below.
This means $p$ cannot be an extremal point of~$E_0$.
Thus, if $E_0$ is strictly contained in $\overline{U}_1$, then any $p\in U_1\cap \partial E_0$ belongs to a (straight) side of $E_0$ which can only terminate, if at all, on $\partial\overline{U}_1$.
As in the case $k^*>0$, this side accumulates in the quotient $j(\Gamma) \backslash \HH^2$ to a geodesic lamination (it cannot terminate on $\partial\overline{U}_1$ at both ends by the assumption $(*)$).

Thus, for $k^*\geq 0$ we have a lamination $\mathscr{L}$ in the convex core of $j(\Gamma) \backslash \HH^2$ whose lift $\tilde{\mathscr{L}}$ to $\HH^2$ is $k^*$-stretched by~$Y$.
It remains to see that $k^* = k = k_{\boldsymbol\alpha}$.
We know that $k^*\geq k$ and $k^*\geq k_{\boldsymbol\alpha}$ (by \eqref{eqn:lip-marg2} applied to~$Y$).
In fact, $k^*=k_{\boldsymbol\alpha}$ by Proposition~\ref{prop:longcurves}, since any minimal component of~$\mathscr{L}$ can be approximated by simple closed curves.
Choose a smooth, $(j,u)$-equivariant vector field $W$ on~$\HH^2$ and recall the function $\newf'_W : T^1\HH^2\rightarrow\RR$ of \eqref{eqn:dintegral}.
For a long segment $[p,q]$ with endpoints in $N_1$, spending most of its length near~$\tilde{\mathscr{L}}$, we see as in Proposition~\ref{prop:longcurves} that $d'_W(p,q)=\int_{[p,q]}\newf'_W$ is roughly $k^*d(p,q)$ by uniform continuity of $\newf'_W$; therefore $k\geq k^*$ and finally $k=k^*$.
This completes the proof of Proposition~\ref{prop:equivariant-kirszbraun}.
\end{proof}

\section{Existence of a maximally stretched lamination}\label{sec:prooflamin}

We now prove Theorem~\ref{thm:laminations}, which states the existence of a lamination that is maximally stretched by all equivariant vector fields of minimal lipschitz constant.
First, in Section~\ref{subsec:weakermaxstretchedlamin}, we bring together all strands of Section~\ref{sec:extlipschitz} and prove the weaker Theorem~\ref{thm:laminations-convex} below, which differs from Theorem~\ref{thm:laminations} only in that the optimal lipschitz constant~$k$ is defined as an infimum over all convex fields (not just vector fields).
To prove Theorem~\ref{thm:laminations}, we then show that the infimum over convex fields is the same as the infimum over vector fields: this is done in Section~\ref{subsec:regularize}, after some technical preparation in Sections \ref{subsec:flowback} and~\ref{subsec:extraingredient} to approximate convex fields by vector fields with almost the same lipschitz constant.
Finally, in Section~\ref{subsec:smoothness} we describe a smoothing process that approximates a vector field by smooth vector fields, again with nearly the same lipschiz constant; this will be used in Section~\ref{sec:geotrans}.

\subsection{A weaker version of Theorem~\ref{thm:laminations}}\label{subsec:weakermaxstretchedlamin}

We first prove the following.

\begin{Theorem}\label{thm:laminations-convex}
Let $\Gamma$ be a discrete group, $j\in\Hom(\Gamma,G)$ a convex cocompact representation, and $u : \Gamma\rightarrow\g$ a $j$-cocycle.
Let $k\in\nolinebreak\RR$ be the infimum of lipschitz constants of $(j,u)$-equivariant \emph{convex fields} $X$ defined over~$\HH^2$.
Suppose $k\geq 0$.
Then there exists a geodesic lamination $\mathscr{L}$ in the convex core of $j(\Gamma)\backslash\HH^2$ that is $k$-stretched by all $(j,u)$-equivariant, $k$-lipschitz convex fields~$X$ defined over~$\HH^2$, meaning that
$$d'_X(p,q) = k\,d(p,q)$$
for any $p\neq q$ on a common leaf of the lift of $\mathscr{L}$ to~$\HH^2$; such convex fields $X$ exist, and can be taken to be standard in the funnels (Definition~\ref{def:standard-funnels}).
\end{Theorem}

\begin{proof}
Recall that by \eqref{eqn:lip-marg2},
$$k \geq k_{\boldsymbol\alpha} \,:=\, \sup_{\gamma \text{ \rm{with} } \lambda (j(\gamma)) > 0}\ \frac{{\boldsymbol\alpha}_u(\gamma)}{\lambda(j(\gamma))}\,,$$
and $k_{\boldsymbol\alpha}$ is still a lower bound for the lipschitz constant of any locally bounded, $(j,u)$-equivariant convex field defined over a nonempty subset of~$\HH^2$, by Proposition~\ref{prop:fixup-boundary}.(1).

Let $(X_n)_{n\in\NN}$ be a sequence of $(j,u)$-equivariant convex fields defined over~$\HH^2$ such that $\lip(X_n)\rightarrow \nolinebreak k$.
By Proposition~\ref{prop:bounds}, the $X_n$ are uniformly bounded over any compact subset of the interior $U \subset \HH^2$ of the convex core.
Therefore, some subsequence of $(X_n(U))_{n\in\NN}$ admits a Hausdorff limit $X_{\infty}$, which is a $k$-lipschitz, $(j,u)$-equivariant convex field defined over~$U$ by Proposition~\ref{prop:limitofconvex}. 
We now use the extension theory of Section~\ref{sec:extlipschitz} to produce a $k$-lipschitz, $(j,u)$-equivariant convex field, defined over all of~$\HH^2$, that agrees with $X_{\infty}$ inside (a subset of) $U$ and is standard in the funnels. 
Let $N$ (resp. $N'$) be the set of points in~$\HH^2$ at distance $>1$ (resp. $>1/2$) from~$U$.
Choose $\eta_0>0$, small enough so that all the hypercycles $H_i$ running inside $U$ at distance $\eta_0$ from the boundary components of~$U$ are disjoint. 
Let $U_0\subset U$ be the closed, connected region bounded by the~$H_i$; the restriction $X_{\infty}(U_0)$ is locally bounded and $k$-lipschitz. 
Choose a hypercycle $H_i$ parallel to a connected component $N'_i$ of~$N'$.
Applying Proposition~\ref{prop:fixup-boundary}.(2) with $\eta_1=-1/2$ and Proposition~\ref{prop:standard-lipschitz}, we obtain a $k$-lipschitz extension $X$ of $X_{\infty}(U_0)$ to $U_0\cup N'_i$ such that
\begin{itemize}
  \item $d'_X(p,q)<k_{\boldsymbol\alpha}\,d(p,q)$ for all distinct $p\in N'_i$ and $q\in H_i\cup N'_i$,
  \item $\lip_p(X)$ is uniformly smaller than $k_{\boldsymbol\alpha}$ for $p\in N'_i$,
  \item $X$ is standard on~$N'_i$.
\end{itemize}
We then extend $X$ equivariantly to $U_0\cup j(\Gamma)\cdot N'_i$, and repeat the procedure for each hypercycle $H_i$ modulo $j(\Gamma)$: this produces a vector field $X$ defined on $U_0 \cup N'$.
We have $k_N:=\lip_N(X)\leq k$ by construction, using a subdivision argument (Observation~\ref{obs:donkey}) to bound $d'_{X(N)}(p,q)/d(p,q)$ for $p$ and~$q$ in different components of~$N$. 
Moreover, the vector field $X(N)$, with lipschitz constant $k_N$, satisfies the hypotheses $(*),(**),(*{*}*)$ of Proposition~\ref{prop:equivariant-kirszbraun} since $k_N\geq k_{\boldsymbol\alpha}$.
(The hypothesis $(*{*}*)$ is satisfied because $X(N)$ is a restriction of the standard field $X(N')$.)
Let $Y$ be the $(j,u)$-equivariant extension of $X(N)$ to~$\HH^2$ given by Proposition~\ref{prop:equivariant-kirszbraun}.
If $k_N<0$, then $\lip(Y)<0$, which is impossible since $k\geq 0$ by assumption. 
Therefore, $k_N\geq 0$ and $k_N=\lip(Y)$.
Since $Y$ is defined over~$\HH^2$, this yields $k_N\geq k$, hence $k_N=k$.
Proposition~\ref{prop:equivariant-kirszbraun} also gives a $k$-stretched lamination in the convex core, and $k_N=k_{\boldsymbol\alpha}$.
\end{proof}

By Theorem~\ref{thm:laminations-convex}, the infimum $k$ of lipschitz constants of $(j,u)$-equivariant convex fields defined on~$\HH^2$ is achieved by a convex field $X$ that is standard in the funnels.
To prove Theorem~\ref{thm:laminations}, we now only need to establish the following proposition.

\begin{Proposition}\label{prop:exaequo}
Given a $(j,u)$-equivariant \emph{convex field}~$X$ defined on~$\HH^2$ and standard in the funnels, there exist $(j,u)$-equivariant \emph{vector fields} $X^*$ defined on $\HH^2$ with $\lip(X^*)$ arbitrarily close to $\lip(X)$. 
\end{Proposition}

The main idea (Proposition~\ref{prop:flowback} below) is that the backwards flow of a lipschitz convex field gives vector field approximates with nearly the same lipschitz constant.

\subsection{The flow-back construction}\label{subsec:flowback}

Let $Z$ be a convex field on~$\HH^2$.
For any $x,y\in\nolinebreak Z$ with $x\in T_p\HH^2$ and $y\in T_q\HH^2$, we set
$$\frac{d'}{d}(x,y) := \frac{d'(x,y)}{d(p,q)}\,,$$
where $d'$ has been defined in \eqref{eqn:d'}.
When $Z$ is a vector field, there is no ambiguity about the vectors $x\in Z(p)$ and $y\in Z(q)$ and we set
\begin{equation}\label{eqn:d'Z}
\frac{d'_Z}{d}(p,q) := \frac{d'}{d}(x,y) = \frac{d'(x,y)}{d(p,q)}\,;
\end{equation}
then $\lip(Z)=\sup_{p\neq q}\frac{d'_Z}{d}(p,q)$.

\begin{Definition}\label{def:flowback}
Let $(\varphi_t)_{t\in\RR}$ be the geodesic flow of~$\HH^2$, acting on $T\HH^2$.
For $t\in\RR$, we shall denote by $Z_t$ the convex field produced by flowing for time~$t$: 
$$Z_t := \varphi_t(Z).$$
When $t < 0$, we refer to $Z_t$ as the \emph{flow-back} of $Z$.
\end{Definition}

The point is that the flow-back of a lipschitz convex field is a (globally defined) vector field (see Figure~\ref{fig:1dflowback} for a one-dimensional illustration):

\begin{figure}[h]
{\centering
\def\svgwidth{4.4in}
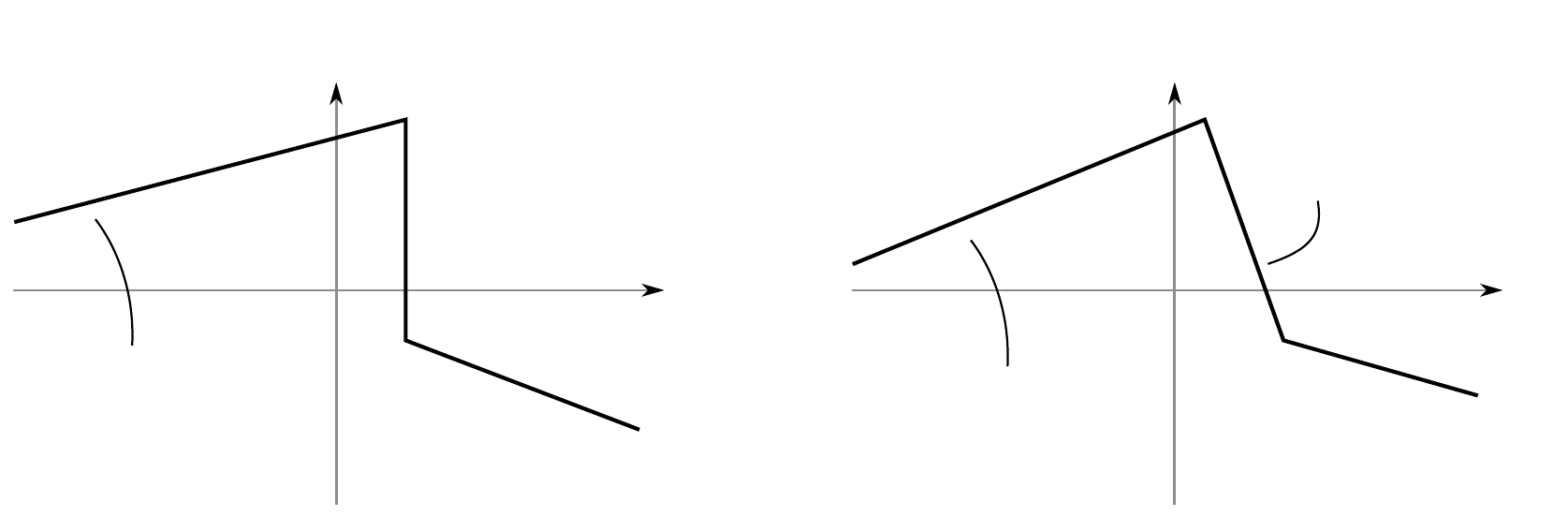}
\caption[]{A one-dimensional illustration: a lipschitz convex field $Z$ over~$\RR$ identifies with a subset of $\RR^2\simeq T\RR$, namely a curve whose slope is bounded from above (by $\lip(Z)$). Its flow-back $Z_{-\eps}$ is the image of $Z$ under the linear map $\left( \begin{smallmatrix} 1 & -\varepsilon\\ 0 & \ 1\end{smallmatrix}\right)$, and is a vector field (\ie the graph of a continuous function $\RR \rightarrow \RR$) with slightly larger lipschitz constant than~$Z$.}
\label{fig:1dflowback}
\end{figure}

\begin{Proposition}\label{prop:flowback}
Let $Z$ be an $R$-lipschitz \emph{convex field} defined on all of $\HH^2$, with $R>0$.
For any negative $t\in (\frac{-1}{R},0)$, the set $Z_t$ is a lipschitz \emph{vector field} defined on~$\HH^2$.
Moreover, for any $\varepsilon>0$ there exists $t_0<0$ such that for all $t\in (t_0,0)$ and all $x\neq y$ in~$Z$,
\begin{equation}\label{eqn:flowback}
\frac{d'}{d}(x_t,y_t) \leq \max \left\{ \frac{d'}{d}(x,y), -R\right\} + \varepsilon ,
\end{equation}
where we set $x_t:=\varphi_t(x)\in Z_t$ and $y_t:=\varphi_t(y)\in Z_t$, and we interpret the maximum in \eqref{eqn:flowback} to be $-R$ if $x$ and $y$ are based at the same point of~$\HH^2$.
If $Z$ is $j$-invariant, then~so~is~$Z_t$.
\end{Proposition}

\begin{proof}
The set $Z_t$ is closed in $T\HH^2$ (since $Z$ is), and $j$-invariant if $Z$ is.
We claim that if $-1/R<t<0$, then $Z_t$ has at most one vector above each point $p\in\HH^2$: indeed, if $Z_t(p)$ contained two vectors $x\neq y$, then we would have $\varphi_{-t}(x),\varphi_{-t}(y)\in Z$ but
$$\frac{d'}{d}\big(\varphi_{-t}(x), \varphi_{-t}(y)\big) \geq \frac{1}{-t} > R$$
by convexity of the distance function, which would contradict $\lip(Z) \leq R$.
Moreover, $Z_t=\varphi_t(Z)$ also has \emph{at least} one point above~$p$: indeed, if $W$ denotes the vector field $q\mapsto \log_q(p)$ that flows all points to $p$ in time~$1$, then $\lip(W)\leq -1$ (again by convexity of the distance function), and a vector $x\in T\HH^2$ satisfies $\varphi_t(x)\in T_p\HH^2$ if and only if $x\in\frac{1}{t}W$.
But $Z$ contains a vector of $\frac{1}{t}W$, because $Z+\frac{1}{-t}W$ has lipschitz constant at most $\lip(Z) +\frac{1}{-t}(-1)<0$ (Observation~\ref{obs:donkey}), hence admits a zero by Proposition~\ref{prop:fixpoint}.
Therefore $Z_t$ is a vector field defined on all of~$\HH^2$ (necessarily continuous, since $Z_t$ is closed in $T\HH^2$).

We now prove \eqref{eqn:flowback}.
Consider $x\neq y$ in~$Z$, with $x\in T_p\HH^2$ and $y\in T_q\HH^2$ for some $p,q\in\HH^2$.
Let $x_t=\varphi_t(x)$ and $y_t=\varphi_t(y)$ be the corresponding vectors of~$Z_t$, based at $p_t:=\exp_p(tx)$ and $q_t:=\exp_q(ty)$ respectively.
If $p=q$, then $\frac{d'}{d}(x_t,y_t)\leq 1/t$ for all $t<0$ by convexity of the distance function, hence \eqref{eqn:flowback} holds for all $-1/R<t<0$.
We now assume $p\neq q$.
Fix $\varepsilon>0$ and let $t_0:=\frac{1}{R+\varepsilon}-\frac{1}{R}<0$.
It is enough to prove that for all $t\in (t_0,0)$,
\begin{equation}\label{eqn:almost-lip}
\frac{d'}{d}(x_t, y_t)\leq f_t\left ( \frac{d'}{d} (x, y)\right ),~\text{ where }~f_t(\xi)=\frac{\xi}{1+t\xi}=(\xi^{-1}+t)^{-1} .
\end{equation}
Indeed, assume that \eqref{eqn:almost-lip} holds.
One checks that $f_{t_0}(\xi)\leq \xi+\varepsilon$ for all $|\xi|\leq R$, that $f_{t_0}(\xi)\leq -R+\varepsilon$ for all $\xi\leq -R$, and that $f_t\leq f_{t_0}$ for all $t\in (t_0,0)$.
Since $\frac{d'}{d}(x,y)\leq R$, this implies \eqref{eqn:flowback}.

\begin{figure}[h!]
{\centering
\def\svgwidth{8cm}
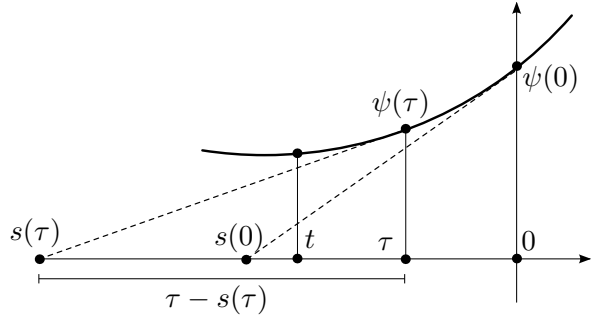}
\caption{The graph of $\psi$ with two tangents.}
\label{fig:convexcurve}
\end{figure}

We now prove \eqref{eqn:almost-lip}.
It is a pure consequence of the convexity of the distance function
$$\begin{array}{cccc}
\psi : & \RR & \longrightarrow & \RR^+ \\
& \tau & \longmapsto & d(p_{\tau},q_{\tau})\,.
\end{array}$$
We can rewrite \eqref{eqn:almost-lip} as 
\begin{equation} \label{eqn:almost-lip2}
\frac{\psi'(t)}{\psi(t)}\leq f_t \left ( \frac{\psi'(0)}{\psi(0)} \right ) .
\end{equation} 
By convexity, $\psi'(t)\leq\psi'(0)$.
If $\psi'(t)\leq 0\leq\psi'(0)$, then \eqref{eqn:almost-lip2} holds because $f_t(\xi)$ has the sign of $\xi$ when $t\in (t_0,0)$ and $\xi\leq R$.
We now assume that $\psi'(0)$ and $\psi'(t)$ have the same sign.
Then we can invert: \eqref{eqn:almost-lip2} amounts to $\frac{\psi(t)}{\psi'(t)}\geq t+\frac{\psi(0)}{\psi'(0)}$.
If $s(\tau)$ denotes the abscissa where the tangent to the graph of $\psi$ at $(\tau,\psi(\tau))$ meets the horizontal axis (see Figure~\ref{fig:convexcurve}), then $\frac{\psi(\tau)}{\psi'(\tau)}=\tau-s(\tau)$, hence \eqref{eqn:almost-lip2} becomes $s(t) \leq s(0)$, which is true by convexity of $\psi$ since $t<0$ and $\psi'(t)$ and $\psi'(0)$ have the same sign. 
This completes the proof of Proposition~\ref{prop:flowback}.
\end{proof}

\subsection{One more ingredient}\label{subsec:extraingredient}

The following lemma expresses the idea that if two points, moving uniformly on straight lines of~$\HH^2$, stay at nearly constant distance, then the line through them stays of nearly constant direction.

\begin{Lemma}\label{lem:almost-parallel}
For any $\theta>0$ there exists $0<\delta<\theta$ with the following property: if $p,p',q,q' \in \HH^2$ (with $p\neq q$ and $p'\neq q'$) all belong to a ball $B$ of diameter $\delta$ and the oriented lines $pq$ and $p'q'$ intersect (possibly outside $B$) at an angle $\geq \theta$, then the midpoints $p''$ of $pp'$ and $q''$ of $qq'$ satisfy
$$d(p'',q'')\leq (1-\delta) \max\{d(p,q),d(p',q')\} .$$ 
\end{Lemma}

\begin{proof}
Let $r$ be the midpoint of $p'q$.
Since $\HH^2$ is $\mathrm{CAT}(0)$, we have
$$d(p'',q'')\leq d(p'',r)+d(r,q'')\leq \frac{d(p,q)}{2}+\frac{d(p',q')}{2}\leq \max\{d(p,q),d(p',q')\} .$$ 
We just need to find a spare factor $(1-\delta)$ between the first and last terms.
Such a spare factor exists in the rightmost inequality unless $\frac{d(p,q)}{d(p',q')}$ is close to~1.
In the latter case, a spare factor exists in the leftmost inequality provided we can bound the angle $\widehat{p''rq''}$ from above by $\pi - \theta/2$.

\begin{figure}[h!]
{\centering
\def\svgwidth{5.5cm}
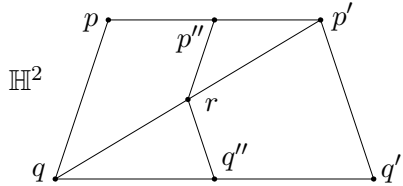}
\caption{If $pq$ and $p'q'$ form an angle, then so do $p''r$ and $rq''$.}
\label{fig:broken}
\end{figure}

The main observation is that for any noncolinear $a,b,c\in \HH^2$, the line $\ell'$ through the midpoints of $ab$ and $ac$ does not intersect the line $\ell$ through~$b$ and~$c$.
Indeed, suppose $(x_t)_{t\in \RR}$ is a parameterization of~$\ell$, and let $x'_t$ be the midpoint of $ax_t$.
Then $(x'_t)_{t\in \RR}$ is a convex curve with the same endpoints as~$\ell$: indeed, if we see $\HH^2$ as a hyperboloid in~$\RR^{2,1}$ as in Section~\ref{subsec:affineactions}, then $x'_t$ is just some positive multiple of $a+x_t$ which describes a branch of hyperbola (\emph{not} contained in a plane through the origin) as $t$ ranges over~$\RR$.
Since a branch of hyperbola in $\RR^3$ looks convex when seen from the origin, $x'_t$ describes a curve $C$ in $\HH^2$ that looks convex (in fact, an arc of conic) in the Klein model.
By convexity, $\ell'\cap C$ is then reduced to the two midpoints of $ab$ and $ac$; the line $\ell'$ cannot cross $C$ a third time, hence does not cross~the~line~$\ell$.

In our situation (see Figure~\ref{fig:broken}), this means the line $p''r$ (resp.\ $rq''$) is $\delta$-close to, but disjoint from the line $pq$ (resp.\ $p'q'$).
But $pq$ and $p'q'$ intersect at an angle $\geq \theta$, at distance at most on the order of $\frac{\delta}{\theta}$ from $B$ (hence close to $B$ if $\delta$ is small).
If $\delta$ is small enough in terms of $\theta$, this means the oriented lines $p''r$ and $rq''$ cross at an angle $>\theta/2$ at~$r$, which is what we wanted to prove.
\end{proof}

\subsection{Convex fields do no better than vector fields}\label{subsec:regularize}

Proposition~\ref{prop:flowback} does not immediately give Proposition~\ref{prop:exaequo}, since the flow-back of a $(j,u)$-equivariant convex field is not necessarily $(j,u)$-equivariant.
However, flowing backwards preserves $j$-invariance.
The trick will be to decompose an equivariant convex field into the sum of a smooth, equivariant vector field and a rough, but invariant convex field, and apply Proposition~\ref{prop:flowback} only to the latter term; the former term will be controlled by Lemma~\ref{lem:almost-parallel}.

\begin{Proposition}\label{prop:allfields}
Let $X$ be a $k$-lipschitz, $(j,u)$-equivariant convex field, defined on~$\HH^2$ and standard in the funnels (Definition~\ref{def:standard-funnels}).
Let $Y$ be a smooth, $(j,u)$-equivariant vector field on~$\HH^2$ that coincides with $X$ outside some cocompact $j(\Gamma)$-invariant set, and let $Z$ be the $j$-invariant convex field defined on~$\HH^2$ such that $X=Y+Z$.
Then $Y+Z_t$ is $(j,u)$-equivariant and
$$\limsup_{t\rightarrow 0^-}\;\lip(Y+Z_t)\leq \lip(X)$$
as $t$ goes to $0$ from below.
\end{Proposition}

\begin{proof}
By Proposition~\ref{prop:flowback} and Observation~\ref{obs:donkey}.(i), for any $t<0$ the convex field $X^t:=Y+Z_t$ is a $(j,u)$-equivariant lipschitz vector field.
Let us check that $\lip(X^t) \leq \lip(X)+o(1)$ as $t$ goes to $0$ from below.

Let $U'$ be a cocompact, $j(\Gamma)$-invariant set outside of which $Z$ is zero.
Note that $X$ is bounded on a compact fundamental domain for~$U'$ by Lemma~\ref{lem:convexbounds}, and so is~ $Y$ by smoothness.
Therefore the $j$-invariant convex field $Z$ is bounded: $\| Z\| <+\infty$.

Fix $\varepsilon>0$.
By smoothness and equivariance of~$Y$, there exist $\theta,R'>0$ such that for all $p\neq q$ in~$U'$ and $p'\neq q'$ in~$\HH^2$ with $d(p,q),d(p,p'),d(q,q')\leq\theta$, if the oriented lines $pq$ and $p'q'$ intersect at an angle $\leq \theta$ (or not at all), then
\begin{equation}\label{eqn:smoothfield}
\frac{d'_{Y}}{d}(p,q) \in [-R',R'] \quad\quad\mathrm{and}\quad\quad \left | \frac{d'_{Y}}{d}(p,q)-\frac{d'_{Y}}{d}(p',q')\right | \leq \varepsilon .
\end{equation} 
(The second condition means that $\theta$ is a modulus of $\varepsilon$-continuity for the function $\newf_{Y}' : T^1\HH^2\rightarrow \RR$ of \eqref{eqn:muintegral-currents}, because $\frac{d'_{Y}}{d}(p,q)$ is the average of $\newf_{Y}'$ over the unit tangent bundle of the segment $[p,q]$.)
We set
$$R := |\lip(X)| + R' >0 ,$$
so that $\lip(Z)\leq R$.
Let $\delta \in (0,\theta)$ be given by Lemma~\ref{lem:almost-parallel}, and $t_0=t_0(\varepsilon)<0$ by Proposition~\ref{prop:flowback}.
We shall prove that $X^t$ is $(k+2\varepsilon)$-lipschitz for any
\begin{equation}\label{eqn:condition-on-t}
\max \left \{ t_0~,~\frac{-\delta}{R}~,~\frac{-\delta}{3\|Z\|} \right \} < t < 0 .
\end{equation}
Let $x,y\in Z$ be vectors based at $\frac{\delta}{3}$-close points $p, q \in U'$.
For $t$ as in \eqref{eqn:condition-on-t}, let $x_t:=\varphi_t(x)$ and $y_t:=\varphi_t(y)$ be the corresponding vectors of~$Z_t$, based at $p_t:=\exp_p(tx)$ and $q_t:=\exp_q(ty)$ respectively.
If $\frac{d'_{Z_t}}{d}(p_t, q_t)\leq -R$, then
$$\frac{d'_{X^t}}{d}(p_t, q_t) \leq R' - R = - |\lip(X)| \leq \lip(X).$$
Note that this includes the case $p = q$ by Proposition~\ref{prop:flowback}.
We can therefore assume that $p\neq q$ and that $\frac{d'_{Z_t}}{d}(p_t, q_t)\geq -R$.
As in the proof of Proposition~\ref{prop:flowback}, the distance function
$$\begin{array}{cccc}
\psi : & \RR & \longrightarrow & \RR^+ \\
& \tau & \longmapsto & d(p_{\tau},q_{\tau})
\end{array}$$
is convex and satisfies $\frac{\psi'(\tau)}{\psi(\tau)}=\frac{d'_{Z_{\tau}}}{d}(p_{\tau},q_{\tau})$ for all~$\tau$.
Since $Z$ is $R$-lipschitz we have $\frac{\psi'(0)}{\psi(0)}\leq R$, hence
\begin{equation}\label{eqn:taught1}
\psi(\tau) \geq (1+tR)\,\psi(0) > (1-\delta)\,\psi(0)
\end{equation}
for all $\tau\in [t,0]$, by convexity of~$\psi$ and choice \eqref{eqn:condition-on-t} of~$t$.
Using convexity again and the assumption that $\frac{d'_{Z_t}}{d}(p_t, q_t)\geq -R$, we have
\begin{equation}\label{eqn:taught2}
\psi(\tau) \geq (1+tR)\,\psi(t) > (1-\delta)\,\psi(t).
\end{equation}
Since $|t|\leq\frac{\delta}{3\|Z\|}$, the points $p_t,q_t,p_0,q_0$ are all within $\delta \leq \theta$ of each other, and therefore the contrapositive of Lemma~\ref{lem:almost-parallel}, together with \eqref{eqn:taught1} and \eqref{eqn:taught2}, implies that the lines $p_t q_t$ and $p_0 q_0$ form an angle $\leq \theta$, or stay disjoint.
Then
$$\frac{d'_{Y}}{d}(p_t,q_t)\leq \frac{d'_{Y}}{d}(p_0,q_0) + \varepsilon$$
by \eqref{eqn:smoothfield}.
We also have
$$\frac{d'_{Z_t}}{d}(p_t,q_t)\leq \frac{d'_{Z}}{d}(p_0,q_0) + \varepsilon$$
by Proposition~\ref{prop:flowback}, since $t\in (t_0,0)$.
Adding these inequalities gives
$$\frac{d'_{X^t}}{d}(p_t,q_t)\leq \frac{d'_{X}}{d}(p_0,q_0) + 2\varepsilon.$$
By subdivision (Observation~\ref{obs:donkey}), we obtain $\lip(X^t)\leq \lip(X) +2\varepsilon$, as wished.
This completes the proof of Proposition~\ref{prop:allfields}.
\end{proof}

Proposition~\ref{prop:allfields} immediately implies Proposition~\ref{prop:exaequo}, and Theorem~\ref{thm:laminations} follows using Theorem~\ref{thm:laminations-convex}.

\subsection{Smooth vector fields}\label{subsec:smoothness}

As it will be useful in Section~\ref{sec:geotrans}, we prove that the $(j,u)$-equivariant vector field $X^*$ of Proposition~\ref{prop:exaequo} can be taken to be smooth (and standard in the funnels).

\begin{Proposition}\label{prop:smooth-approximates}
Given a $(j,u)$-equivariant \emph{convex field} $X$ defined on~$\HH^2$ and standard in the funnels, there exist \emph{smooth}, standard, $(j,u)$-equivariant \emph{vector fields} $X^*$ defined on $\HH^2$ with $\lip(X^*)$ arbitrarily close to $\lip(X)$. 
\end{Proposition}

It is sufficient to smooth out the $j$-invariant vector field $Z_t$ of Proposition~\ref{prop:allfields} (without destroying the lipschitz constant).
Our first step is to prove that $Z_t$ is actually a Lipschitz (uppercase!) section of the tangent bundle of~$\HH^2$.
First, let us recall the definition.

For $p,q\in\HH^2$, let $\phi_p^q$ be the isometry of~$\HH^2$ that takes $p$ to~$q$ by translating along the geodesic $(p,q)$.
To condense notation, the differential of this map will again be denoted by~$\phi_p^q$.
Note that the action of $\phi_p^q$ on $T_p\HH^2$ is parallel transport along the geodesic segment $[p,q]$.
We also note that $(\phi_p^q)^{-1} = \phi_q^p$.
A vector field $X$ on~$\HH^2$ is said to be \emph{Lipschitz} if it is a Lipschitz section of $T\HH^2$, in the sense that there exists $L\geq 0$ such that for all $p,q\in\HH^2$,
$$\| \phi_p^qX(p) - X(q) \| \leq L \, d(p,q)$$
This implies in particular $\lip(X)\leq L$.

\begin{Lemma}\label{lem:uppercase}
Let $Z'$ be a bounded \emph{lipschitz} vector field. 
For all $t_0<0$ close enough to~$0$, the flow-back $Z'_{t_0}$ is a \emph{Lipschitz} vector field.
\end{Lemma}

\begin{proof}
By Proposition~\ref{prop:flowback}, for $t_0$ close enough to $0$ (depending on $\lip(Z')$ only), the vector fields $Z'_t$ for $t_0\leq t \leq 0$ are all $R$-lipschitz for some $R\in\RR$ independent of~$t$.
Fix such $t_0$ and $R$, and define $Z:=Z'_{t_0}$.

Consider distinct points $p,q\in\HH^2$.
For $0\leq t \leq |t_0|$, set $p_t:=\exp_p(tZ(p))$ and $q_t:=\exp_q(tZ(q))$.
Since $Z_t=Z'_{t_0+t}$, we have $\lip(Z_t)\leq R$, which by integrating implies $d(p_t, q_t)\leq e^{Rt} d(p,q)$.
Define moreover $r_t:= \phi_q^p(q_t)\in\HH^2$.
If $\| Z\|\leq N$ (notation \eqref{eqn:normconvexfield}), then $d(q_t, r_t)\leq \cosh (Nt)\, d(p,q)$.

Define comparison points $p^*, q^* \in \RR^2$, distance $d(p,q)$ apart, as well as vectors $P,Q\in \RR^2$ such that $\| P\|=\| Z(p)\|$ and $\| Q \|=\| Z(q) \|$, and such that $[p^*, q^*]$ forms the same two angles with $P$ and~$Q$ as $[p,q]$ does with $Z(p)$ and $Z(q)$ respectively.
Define $p^*_t:=p^*+tP$ and $q^*_t:=q^*+tQ$ and $r^*_t:=p^*+tQ$.
Then
\begin{eqnarray*}
 d(p_t^*, q_t^*) & \leq &  d(p_t^*, r_t^*) +  d(r_t^*, q_t^*) \\
 &\leq & d(p_t, r_t) + d(p,q) \\
 & \leq & d(p_t, q_t) + d(q_t, r_t) + d(p,q) \\
 & \leq & \left ( e^{Rt} + \cosh (Nt) + 1 \right ) d(p,q) .
\end{eqnarray*}
On the other hand, by the triangle inequality, 
$$ d(p_t^*, q_t^*) \geq d(p_t^*, r_t^*) - d(r_t^*, q_t^*) = t \| P-Q\|  - d(p,q) .$$
But $\| P-Q\| =\|  \phi_p^q Z(p) - Z(q) \| $.
Combining these estimates, we find 
$$\|  \phi_p^q Z(p) - Z(q) \|  \leq \frac{e^{Rt} + \cosh (Nt) + 2}{t} \, d(p,q) .$$
Taking $t=|t_0|$ gives an upper bound for the Lipschitz constant of~$Z$.
\end{proof}

It follows from Proposition~\ref{prop:flowback} and Lemma~\ref{lem:uppercase} that the $j$-invariant component of the vector field constructed in Proposition~\ref{prop:allfields} is Lipschitz.
We next apply to it a smoothing procedure.

For any $\Delta > 0$, we choose a smooth kernel $\psi_{\Delta} : \HH^2 \times \HH^2 \to \RR^{+}$, invariant under all isometries of $\HH^2$, that vanishes on all pairs of points distance $>\Delta$ apart, and such that
\begin{equation}\label{eqn:assumptionpsi}
\int_{\HH^2} \psi_{\Delta}(p,p') \, \D p' = 1
\end{equation}
for all $p\in\HH^2$.
For any vector field $Z$ on~$\HH^2$, we define a smoothed vector field $\widetilde{Z}_{\Delta}$ on~$\HH^2$ by convolution by~$\psi_{\Delta}$:
$$\widetilde{Z}_\Delta :\ p \longmapsto \int_{\HH^2} \psi_\Delta(p,p') \, \phi_{p'}^p Z(p')  \, \D p' .$$
Then $\widetilde{Z}_{\Delta}$ inherits the smoothness of~$\psi$.
If $Z$ is $j$-invariant, then so is~$\widetilde{Z}_\Delta$. 

\begin{Lemma}\label{lem:smooth}
Let $X$ be a lipschitz, $(j,u)$-equivariant vector field on~$\HH^2$.
Let $Y$ be a smooth, $(j,u)$-equivariant vector field on~$\HH^2$ that coincides with $X$ outside some cocompact $j(\Gamma)$-invariant set, and let $Z$ be the $j$-invariant vector field on~$\HH^2$ such that $X=Y+Z$.
If $Z$ is Lipschitz, then for any $\varepsilon>0$ and any small enough $\Delta > 0$,
$$\lip\big(Y+\widetilde{Z}_{\Delta}\big) \leq \lip(X) + \varepsilon.$$
\end{Lemma}

\begin{proof}
Fix $\varepsilon>0$ and consider $\Delta >0$ (to be adjusted later).
Throughout the proof, we write $\widetilde Z$ and $\psi$ in place of $\widetilde Z_\Delta$ and $\psi_\Delta$.
By subdivision (Observation~\ref{obs:donkey}), it is enough to prove that if $\Delta$ is small enough, then $d'_{Y+\widetilde{Z}}(p,q)\leq (\lip(X) + \varepsilon)\,d(p,q)$ for all distinct $p,q\in\HH^2$ with $d(p,q)\leq\Delta$.
Consider such a pair $(p,q)$, let $u_p\in T_p\HH^2$ be the unit vector at $p$ pointing to~$q$, and set $u_q:=\phi_p^q(u_p)\in T_q\HH^2$. 
By Remark~\ref{rem:d'proj},
\begin{align*}
d'_{\widetilde{Z}}(p,q) & =~\langle \widetilde{Z}(q) \,|\, u_q \rangle - \langle \widetilde{Z}(p) \,|\, u_p \rangle \\
&= \int_{\HH^2} \psi(q,q') \langle \phi_{q'}^q Z(q') \,|\, u_q \rangle \, \D q' - \int_{\HH^2} \psi(p,p') \langle \phi_{p'}^p Z(p') \,|\, u_p \rangle \, \D p' \\
&= \int_{\HH^2} \psi(q,q') \langle Z(q') \,|\, \phi_q^{q'} u_q \rangle \, \D q' - \int_{\HH^2} \psi(p,p') \langle Z(p') \,|\, \phi_p^{p'} u_p \rangle \, \D p' .
\end{align*}
In the first integral, we make the substitution $q'=\phi_p^q (p')$:
$$\int_{\HH^2} \psi(q,q') \langle Z(q') \,|\, \phi_q^{q'} u_q \rangle \, \D q' = \int_{\HH^2} \psi(p,p') \langle \phi_p^q Z(q') \,|\, \phi_p^{p'} u_p \rangle \, \D p' ,$$
where we use the invariance of $\psi$ under~$\phi_p^q$ and the fact that $\phi_q^p$ takes $\phi_q^{q'} u_q$ to $\phi_p^{p'} u_p$, by the conjugacy relation $\phi_p^{p'} = \phi_q^p\,\phi_q^{q'}\!(\phi_q^p)^{-1}$.
Therefore,
\begin{equation}\label{eqn:intlemmasmooth}
d'_{\widetilde{Z}}(p,q) = \int_{\HH^2} \psi(p,p') \, \langle \phi_q^p Z(q') -Z(p') \,|\, \phi_p^{p'} u_p \rangle \, \D p' .
\end{equation}
We now focus on the integrand.
We may restrict to $p'$ at distance $\leq\Delta$ from~$p$, otherwise $\psi(p,p')=0$.
Let $u'$ to be the unit vector at $p'$ pointing to~$q'$ (see Figure~\ref{fig:smoothness-proof}).
Then
\begin{align*}
\langle \phi_q^p Z(q') -Z(p') \,|\, \phi_p^{p'} u_p \rangle  \ \ \ &=&  \ \ \ \ &
\langle \phi_{q'}^{p'} Z(q') - Z(p') \,|\, u' \rangle \\ 
&& + \ \ \ &\langle \phi_{q'}^{p'} Z(q')-Z(p') \,|\, \phi_p^{p'} u_p - u' \rangle \\ 
&& + \ \ \ &\langle (\phi_q^p - \phi_{q'}^{p'}) \, Z(q') \,|\, \phi_p^{p'} u_p \rangle\,.
\end{align*}
The first term on the right-hand side is $d'_Z(p',q')$ (Remark~\ref{rem:d'proj}), and we shall see that the second and third terms are but small corrections, bounded respectively by $L\,\Delta \,d(p',q')$ and $\|Z\|\,\Delta\,d(p',q')$ if $\Delta$ is small enough.
Here $L$ is any positive number such that $Z$ is $L$-Lipschitz; we have $\| Z \|<+\infty$ since $Z$ is continuous, $j$-invariant, and is zero outside some cocompact set.
Let us explain these bounds.

In the second term we have $\| \phi_{q'}^{p'} Z(q')-Z(p')\|\leq L\,d(p',q')$ because $Z$ is $L$-Lipschitz.
For $\Delta$ small enough, we also have $\| \phi_p^{p'} u_p - u'\| \leq \Delta$ by the Gauss--Bonnet formula.
Indeed, let $w = \phi_p^{p'} u_p$, let $p''$ and $q''$ be the closest points to $p'$ and $q'$ respectively on the line $pq$, and let $r'$ and $r''$ be the midpoints of $[p',q']$ and $[p'',q'']$ respectively as in Figure~\ref{fig:smoothness-proof}.
Then the angle between $u'$ and $w$ is equal to the area of the quadrilateral $pr''r'p'$, because $u'$ may be obtained by transporting $u_p$ along the broken line $pr''r'p'$ while $w$ is the transport of $u_p$ along $pp'$.
Since $d(p,q)$ and $d(p,p')=d(q,q')$ are bounded by~$\Delta$, we obtain that $\| w - u'\|$ is bounded by the area $2\pi(\cosh(2\Delta)-1)$ of a ball of radius~$2\Delta$ in~$\HH^2$.
In particular, $\| w - u'\|\leq\Delta$ if $\Delta$ is small enough.
We note that as $q \to p$, the angle between $w$  and $u'$ approaches the area of triangle $pp''p'$, which in general is nonzero.
This negative curvature phenomenon makes the Lipschitz assumption on~$Z$ necessary (at least for this proof).

In the third term, we observe that $\phi_q^p Z(q')$ is the parallel translation of $Z(q')$ along the broken line $q'qpp'$ and thus differs from $\phi_{q'}^{p'} Z(q')$ by a rotation of angle equal to the area of the quadrilateral $q'qpp'$.
This area is at most $\Delta \, d(p',q')$, hence $\| (\phi_q^p - \phi_{q'}^{p'})\,Z(q')\| \leq \Delta \| Z\| \, d(p',q')$.

\begin{figure}[h!]
{\centering
\def\svgwidth{7cm}
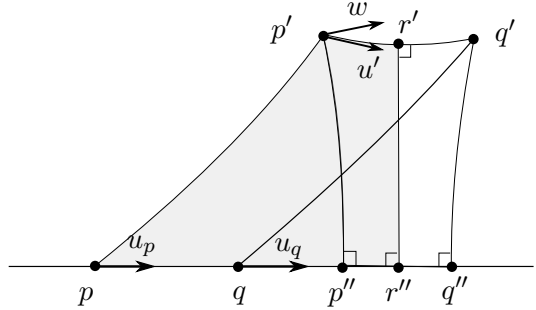}
\caption{In the proof of Lemma~\ref{lem:smooth}, the angle between $u'$ and $w = \phi_p^{p'}u_p$ is the area of the shaded quadrilateral $pr''r'p'$.}
\label{fig:smoothness-proof}
\end{figure}

We now turn back to the integrand in \eqref{eqn:intlemmasmooth}.
If $\Delta$ is small enough, with $(L+\| Z\|) \Delta \leq \eps/4$ and $\cosh \Delta \leq \frac{L + \eps/2}{L+ \eps/4}$, then
\begin{eqnarray}\label{eqn:integrandlemmasmooth}
\langle \phi_q^p Z(q') - Z(p') \,|\, \phi_p^{p'} (u_p) \rangle & \leq & \left ( \frac{d'_Z}{d}(p',q') + L\Delta + \| Z\| \Delta \right ) \,  d(p',q') \nonumber\\ 
& = & \left ( \frac{d'_Z}{d}(p',q') + \frac{\varepsilon}{4} \right) \, d(p,q) \, \cosh \Delta \nonumber\\ 
& \leq & \left (\frac{d'_Z}{d}(p',q') +\frac{\eps}{2} \right ) \, d(p,q),
\end{eqnarray}
where the last inequality follows from the bounds $\cosh \Delta \leq \frac{L + \eps/2}{L+ \eps/4}$ and $|\frac{d'_Z}{d}(p',q')| \leq L$, and from the monotonicity of the function $t\mapsto\frac{t+\eps/2}{t+\eps/4}$.

We now bound $\lip(Y+\widetilde{Z})$.
The vector field $Z$ is zero outside some $j(\Gamma)$-invariant, cocompact set $U''\subset\HH^2$.
Let $U'$ be a $j(\Gamma)$-invariant, cocompact neighborhood of~$U''$.
If $\Delta$ is small enough, then the interior of~$U'$ contains the closed $\Delta$-neighborhood $U_{\Delta}$ of~$U''$.
By smoothness of~$Y$, up to taking $\Delta$ even smaller, we may assume that for any $p\neq q$ in~$U'$ and $q' = \varphi_p^q (p')$ with $d(p,q),d(p,p')\leq\Delta$ we have $\frac{d'_Y}{d}(p,q)\leq \frac{d'_Y}{d}(p',q')+\varepsilon/2$.
Then \eqref{eqn:assumptionpsi}, \eqref{eqn:intlemmasmooth}, and \eqref{eqn:integrandlemmasmooth} imply that for all $p,q\in U'$ with $0<d(p,q)\leq\Delta$,
\begin{align*}
\frac{d'_{Y+\widetilde{Z}}}{d}(p,q)
&\leq \int_{\HH^2} \psi(p,p') \left ( \frac{d'_Y}{d}\big(p',\phi_p^q(p')\big) + \frac{\eps}{2} + \frac{d'_Z}{d}\big(p',\phi_p^q(p')\big) +  \frac{\varepsilon}{2} \right ) \, \D p' \\
&\leq \lip(Y+Z) + \varepsilon \ =\ \lip(X) + \varepsilon .
\end{align*}
By subdivision, $\lip_{U'}(Y+\widetilde{Z})\leq\lip(X)+\varepsilon$.
On the other hand, on $\HH^2\smallsetminus U_{\Delta}$ we have $Y+\widetilde{Z}=Y=X$, hence $\lip_{\HH^2\smallsetminus U_{\Delta}}(Y+\widetilde{Z})=\lip(X)$.
We conclude using Observation~\ref{obs:donkey}.(iv).
\end{proof}

Proposition~\ref{prop:smooth-approximates} follows from Proposition~\ref{prop:allfields} and Lemmas \ref{lem:uppercase} and~\ref{lem:smooth}.

\section{Applications of Theorem~\ref{thm:laminations}: proper actions and fibrations}\label{sec:applications}

What we have proved so far shows that the conditions $(1)$ and $(2)$ of Theorem~\ref{thm:newGLM} are equivalent: $(1)$ implies $(2)$ by \eqref{eqn:lip-marg2}, and the negation of $(1)$ implies by Theorem~\ref{thm:laminations} the existence of a $k$-stretched lamination $\mathscr{L}$ (with $k\geq 0$), hence the negation of $(2)$ by Proposition~\ref{prop:longcurves} (any minimal component of~$\mathscr{L}$ can be approached by simple closed curves).

In Section~\ref{subsec:nec-proper} (resp.~\ref{subsec:suff-proper}), we prove that $(1)$ is necessary (resp.\ sufficient) for the action on~$\RR^{2,1}$ to be properly discontinuous, thus proving Theorem ~\ref{thm:newGLM}.
The first direction (Proposition~\ref{prop:proper-lip}) uses Theorem~\ref{thm:laminations}; the second (Proposition~\ref{prop:fibrations}, where fibrations by timelike lines appear) can be read independently. Theorem~\ref{thm:fibrations} is then a byproduct of Proposition~\ref{prop:fibrations}.

Note that it is enough to prove Theorems~\ref{thm:newGLM} and~\ref{thm:fibrations} for a finite-index, torsion-free subgroup $\Gamma'$ of~$\Gamma$ (such a subgroup exists by the Selberg lemma \cite[Lem.\,8]{sel60}).
Indeed, if $X$ is a $(j|_{\Gamma'},u|_{\Gamma'})$-equivariant convex or vector field, we can always average out the translates $\gamma_i\bullet X$ (notation \eqref{eqn:affineaction}), where the cosets $\gamma_i\Gamma'$ form a partition of $\Gamma$, to produce a $(j,\rho)$-equivariant field whose lipschitz constant is at most that of~$X$ (using Observation~\ref{obs:donkey}).
Thus, we assume $\Gamma$ to be torsion-free in this section.

\subsection{A necessary condition for properness}\label{subsec:nec-proper}

Let $\Gamma$ be a torsion-free discrete group and $j\in\Hom(\Gamma,G)$ a convex cocompact representation.
Recall that $\lambda(g)$ refers to the translation length of $g\in G$, see \eqref{eqn:length_contraction}.
In the ``macroscopic'' case, the following was established in \cite{kasPhD}; see also \cite{gk12} for generalizations.

\begin{Propositionwithref}[{\cite[Th.\,5.1.1]{kasPhD}}]
Let $\rho\in\Hom(\Gamma,G)$ be an arbitrary representation such that $\lambda(\rho(\beta))\leq\lambda(j(\beta))$ for at least one $\beta \in \Gamma \smallsetminus \{1\}$.
If the group
$$\Gamma^{j,\rho} := \big\{ (j(\gamma),\rho(\gamma)) \,|\, \gamma\in\Gamma\big\} \subset G\times G$$
acts properly discontinuously on $\AdS$, then there exists a $(j,\rho)$-equivariant Lipschitz map $f : \HH^2\rightarrow\HH^2$ with Lipschitz constant $<1$.
\end{Propositionwithref}

Here we prove the following ``microscopic'' version.
Recall that if $u$ is a $j$-cocycle, then the Margulis invariant ${\boldsymbol\alpha}_u(\gamma)=\frac{\D}{\D t}\big|_{t=0}\,\lambda(e^{tu(j(\gamma))}j(\gamma))$ is the infinitesimal rate of change of the length of $j(\gamma)$ under deformation in the $u$ direction (see Section~\ref{subsec:length_deriv}).

\begin{Proposition}\label{prop:proper-lip}
Let $u : \Gamma\rightarrow\g$ be a $j$-cocycle such that ${\boldsymbol\alpha}_u(\beta)\leq 0$ for at least one $\beta \in \Gamma \smallsetminus \{1\}$.
If the group
$$\Gamma^{j,u} := \big\{ (j(\gamma),u(\gamma)) \,|\, \gamma\in\Gamma\big\} \subset G\ltimes\g$$
acts properly discontinuously on~$\RR^{2,1}$, then there exists a $(j,u)$-equivariant lipschitz vector field on~$\HH^2$ with lipschitz constant $<0$.
\end{Proposition}

\begin{proof}
We prove the contrapositive.
Let $k$ be the infimum of lipschitz constants of $(j,u)$-equivariant vector fields on~$\HH^2$ and assume $k \geq 0$.
Let $\ell$ be a geodesic line in $\HH^2$ that projects to a leaf of the maximally stretched lamination $\mathscr{L}$ given by Theorem~\ref{thm:laminations}.
Let $X$ be the $k$-lipschitz, globally defined convex field also given by Theorem~\ref{thm:laminations}. 

$\bullet$ \textbf{The case $k=0$.}
We first suppose that $k=0$.
Since $X$ is defined on all of~$\HH^2$, Lemma~\ref{lem:convexbounds} implies that $\| X(\specialC) \|<+\infty$ for any compact set $\specialC\subset\HH^2$.

We claim that there is a Killing field $Y_0$ on~$\HH^2$ such that $X(p)$ contains $Y_0(p)$ for any $p$ in the leaf~$\ell$.
Indeed, choose for every $p \in \ell$ a vector $x_p\in X(p)$.
If $p,q\in \ell$ are distinct, then there is a unique Killing field $Y$ such that $x_p=Y(p)$ and $x_q=Y(q)$.
For any point $r$ on the segment $[p,q]$, the vector $Y(r)$ belongs to $X(r)$ by Proposition~\ref{prop:zeroconvex}.
The Killing field $Y$ containing $x_p$ and $x_q$ may depend on $p$ and $q$, but as $p$ and~$q$ escape along the line~$\ell$ in opposite directions, the possible Killing fields $Y$ that arise all belong to a compact subset of $\g$, because any $X(r)$ is bounded.
Therefore we can extract a limit $Y_0$ such that $Y_0(p)\in X(p)$ for all $p\in\ell$.

Up to modifying $u$ by a coboundary, we may assume that $Y_0$ is the zero vector field, \ie $\underline{0}(p)\in X(p)$ for all $p\in\ell$.
Since $\ell$ is contained in the convex core, which is compact modulo $j(\Gamma)$, we can find a ball $B\subset\HH^2$, a sequence $(\gamma_n)_{n\in\NN}$ of pairwise distinct elements of~$\Gamma$, and, for any $n\in\NN$, two points $p_n, p'_n \in \ell$, distance one apart, such that $q_n:= j(\gamma_n) \cdot p_n$  and $q'_n:=j(\gamma_n) \cdot p'_n$ both belong to~$B$.
Since $\underline{0}(p_n)\in X(p_n)$, the set $X(q_n)={j(\gamma_n)}_*(X(p_n))+u(\gamma_n)(q_n)$ contains $u(\gamma_n)(q_n)$; similarly $u(\gamma_n)(q'_n)\in X(q'_n)$.
Since $R:=\| X(B)\|<+\infty$ and $d(q_n, q'_n)=1$, for any $n\in\NN$ the Killing field $u(\gamma_n)$ lies in the \emph{compact} set
$$\mathcal{C} = \big\{ v \in \g \ |\ \exists\, q,q'\in B \text{ with } d(q,q')=1 \text{ and } \| v(q)\|, \|v(q') \| \leq R \big\} .$$
Thus the action of $\Gamma^{j,u}$ on $\g\cong\RR^{2,1}$ fails to take the origin $0\in\g$ off~$\mathcal{C}$; in particular, it cannot be properly discontinuous.

$\bullet$ \textbf{The case $k>0$ and the opposite sign lemma.}
We now suppose that $k>0$.
By Proposition~\ref{prop:longcurves}, there is an element $\alpha\in\Gamma\smallsetminus\{ 1\}$, corresponding to a closed curve of $j(\Gamma)\backslash\HH^2$ nearly carried by~$\mathscr{L}$, whose Margulis invariant ${\boldsymbol\alpha}_u(\alpha)$ is positive.
Recall that by assumption in Proposition~\ref{prop:proper-lip}, there is also an element $\beta\in\Gamma\smallsetminus\{ 1\}$ with ${\boldsymbol\alpha}_u(\beta) \leq 0$.
The existence of such $\alpha,\beta$ is a well-known obstruction to properness, known as Margulis's \emph{opposite sign lemma} \cite{mar83,mar84}. For convenience, we give the idea of a proof here. 

If ${\boldsymbol\alpha}_u(\beta)=0$, then $(j(\beta),u(\beta))\in G\ltimes\g$ has a fixed point in $\g\cong\RR^{2,1}$, hence $\Gamma^{j,u}$ does not act properly discontinuously on~$\RR^{2,1}$.
We now suppose ${\boldsymbol\alpha}_u(\alpha)>0>{\boldsymbol\alpha}_u(\beta)$.
Let $(a^+,a^0,a^-)$ and $(b^+,b^0,b^-)$ be bases of eigenvectors for the adjoint action of $j(\alpha)$ and $j(\beta)$ on~$\g$, respectively.
As in Section~\ref{subsec:margulis-invariant}, we assume that $a^-,a^+,b^-,b^+$ belong to the \emph{positive} light cone $\mathcal{L}$ of~$\g$, that $a^-,b^-$ (resp.\ $a^+,b^+$) correspond to eigenvalues $<1$ (resp.\ $>1$), and that $a^0$ (resp.~$b^0$) is a positive multiple of $a^-\wedge a^+$ (resp.\ of $b^-\wedge b^+$), of norm~$1$.
From $a^0$, when looking at~$\mathcal{L}$, one sees $a^+$ on the left and $a^-$ on the right, and similarly for $b^0,b^+,b^-$.
Let $A, B \subset\g$ be the affine lines, of directions $\RR a^0$ and $\RR b^0$, that are preserved by $(j(\alpha),u(\alpha))$ and $(j(\beta),u(\beta))$ respectively (see Section~\ref{subsec:margulis-invariant}).
For simplicity, we first assume that $A$ and~$B$ both contain the origin~$O$.

Let $\mathcal{A}=A+\RR a^+$ be the unstable lightlike plane containing~$A$, and $\mathcal{B}=B+\RR b^-$ the stable lightlike plane containing~$B$.
Up to applying a linear transformation in $\OO(2,1)_0$ and rescaling $a^+$ and~$b^-$, we may assume that $\mathcal{A}\cap\mathcal{B}$ is the $x$-axis $\RR (1,0,0)$ and that $a^+=(0,-1,1)$ and $b^-=(0,1,1)$.
Then $a^0\in\mathcal{A}$ and $b^0\in\mathcal{B}$ have positive $x$-coordinates, which we denote by $a^0_x$ and $b^0_x$ respectively.
Figure~\ref{fig:opposite} looks down the $z$-axis.

\begin{figure}[h!]
{\centering
\def\svgwidth{11cm}
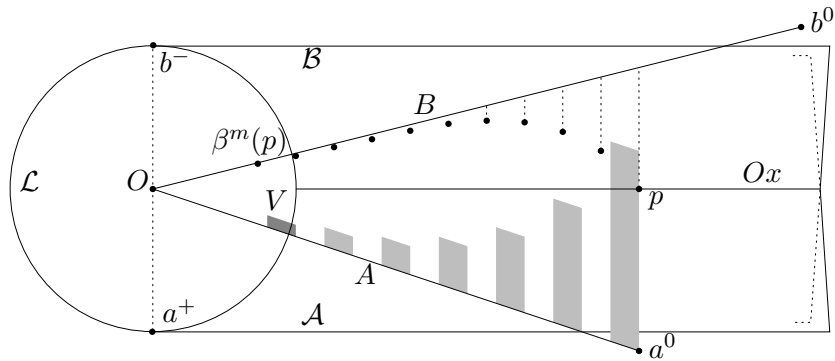}
\caption{Looking down the $z$-axis when the axes $A$ and~$B$ both contain the origin~$O$. The positive light cone $\mathcal{L}$, truncated, appears as a circle.}
\label{fig:opposite}
\end{figure}

Let $V\subset\mathcal{A}$ be a compact neighborhood of a point of~$A$, for example near the origin.
Since ${\boldsymbol\alpha}_u(\alpha)>0$, the action of $(j(\alpha),u(\alpha))\in G\ltimes\g$ on~$\mathcal{A}$ is to translate in the $x$-direction to the right (preserving the axis~$A$), while expanding exponentially in the $a^+$-direction outwards from~$A$.
In particular, for all large enough~$n$, the iterate $(j(\alpha),u(\alpha))^n\cdot V$ meets the $x$-axis at a point $p_n$ of large positive abscissa $n\,a^0_x\,{\boldsymbol\alpha}_u(\alpha)+O(1)$.

Since ${\boldsymbol\alpha}_u(\beta)<0$, the action of $(j(\beta),u(\beta))\in G\ltimes\g$ on~$\mathcal{B}$ is to translate in the $x$-direction to the left (preserving the axis~$B$), while contracting exponentially in the $b^-$-direction towards~$B$.
In particular, $(j(\beta),u(\beta))^m\cdot p_n$ is close to the origin $O$ if
$$m = \left \lfloor n \, \frac{a^0_x\,{\boldsymbol\alpha}_u(\alpha)}{b^0_x\,|{\boldsymbol\alpha}_u(\beta)|}\right \rfloor,$$
where $\lfloor\cdot\rfloor$ denotes the integral part.
For such integers $n,m$, the element $(j(\beta),u(\beta))^m\,(j(\alpha),u(\alpha))^n\in\Gamma^{j,u}$ does not carry $V$ off some compact set.
Thus the action of $\Gamma^{j,u}$ on~$\RR^{2,1}$ is not properly continuous.

If the axes $A$ and~$B$ do not both contain the origin~$O$, then the argument is essentially the same, replacing the $x$-axis with the intersection line $\Delta=\mathcal{A}\cap\mathcal{B}$ (which we may assume to be parallel to the $x$-axis): we see that $\beta^m\cdot p_n$ is still close to the point $B\cap\Delta$ for $n,m$ as above.
This completes the proof of Proposition~\ref{prop:proper-lip}.
\end{proof}

\subsection{Fibrations by timelike geodesics and a sufficient condition for properness}\label{subsec:suff-proper}

The following proposition, together with Proposition~\ref{prop:proper-lip}, completes the proof of Theorems \ref{thm:newGLM} and~\ref{thm:fibrations}.

\begin{Proposition}\label{prop:fibrations}
Let $\Gamma$ be a torsion-free discrete group and $j\in\Hom(\Gamma,G)$ a discrete and injective representation, with quotient surface $S:=j(\Gamma)\backslash\HH^2$.
Let $\rho\in\Hom(\Gamma,G)$ be an arbitrary representation and $u : \Gamma\rightarrow\g$ a $j$-cocycle.
\begin{enumerate}
  \item Suppose there exists a $(j,\rho)$-equivariant Lipschitz map $f : \HH^2\rightarrow\HH^2$ with Lipschitz constant $K<1$. Then the group
  $$\Gamma^{j,\rho} := \big\{ (j(\gamma),\rho(\gamma)) \,|\, \gamma\in\Gamma\big\} \subset G\times G$$
  acts properly discontinuously on $\AdS$ and $f$ induces a fibration $\mathcal{F}_f$ of the quotient $\Gamma^{j,\rho}\backslash\AdS$ over $S$ by timelike geodesic circles.
  \item Suppose there exists a $(j,u)$-equivariant lipschitz vector field\linebreak $X : \HH^2\rightarrow T\HH^2$ with lipschitz constant $k<0$. Then the group
  $$\Gamma^{j,u} := \big\{ (j(\gamma),u(\gamma)) \,|\, \gamma\in\Gamma\big\} \subset G\ltimes\g$$
  acts properly discontinuously on~$\RR^{2,1}$ and $X$ induces a fibration $\mathcal{F}_X$ of the quotient $\Gamma^{j,u}\backslash\RR^{2,1}$ over $S$ by timelike geodesic lines.
\end{enumerate}
\end{Proposition}

The fact that the existence of a $(j,\rho)$-equivariant contracting Lipschitz map implies the properness of the action of~$\Gamma^{j,\rho}$ is an easy consequence of the general \emph{properness criterion} of Benoist \cite{ben96} and Kobayashi \cite{kob96}: see~\cite{kasPhD}.
Our method here gives a different, short proof.

Note that we do not require $j$ to be convex cocompact in Proposition~\ref{prop:fibrations}.
In particular, using \cite[Th.\,1.8]{gk12}, we obtain that \emph{any Lorentzian $3$-manifold which is the quotient of $\AdS$ by a finitely generated group fibers in circles over a ($2$-dimensional) hyperbolic orbifold}.

\begin{proof}[Proof of Proposition~\ref{prop:fibrations}]
(1) The timelike geodesics in $\AdS=G$ are parameterized by $\HH^2 \times \HH^2$: for $p,q \in \HH^2$, the corresponding timelike geodesic~is
$$L_{p,q} := \{g \in G \ |\ g \cdot p = q\}, $$
which is a coset of the stabilizer of~$p$, hence topologically a circle.
The map $f : \HH^2\rightarrow\HH^2$ determines a natural collection $\{L_{p,f(p)} \,|\, p \in \HH^2\}$ of timelike geodesics in $G$.
This collection is a fibration of $G$ over~$\HH^2$.
Indeed, for $g \in G$, the map $g^{-1}\circ f$ is $K$-Lipschitz because $f$ is; since $K<1$, it has a unique fixed point, which we denote by $\Pi(g)\in\HH^2$.
Thus $g$ belongs to $L_{p,f(p)}$ for a unique $p\in\HH^2$, namely $\Pi(g)$.
The surjective map $\Pi : G\rightarrow\HH^2$ is continuous.
Indeed, if $g' \in G$ is close enough to $g$ in the sense that $d(p,g'^{-1}\circ f(p)) < (1-K)\,\delta$ for $p=\Pi(g)$, then $g'^{-1} \circ f$ maps the ball of radius $\delta$ around $p$ into itself, hence the unique fixed point $\Pi(g')$ of $g'^{-1} \circ f$ is within $\delta$ of $p=\Pi(g)$.
Finally, the map~$\Pi$ satisfies the following equivariance property:
$$\Pi\big(\rho(\gamma) g j(\gamma)^{-1}\big) = j(\gamma) \cdot \Pi(g)$$
for all $\gamma\in\Gamma$ and $g\in G$.
Therefore, the properness of the action of $\Gamma^{j,\rho}$ on $G=\AdS$ follows from the properness of the action of $j(\Gamma)$ on~$\HH^2$, and $\Pi$ descends to a bundle projection $\Pi : \Gamma^{j,\rho}\backslash\AdS \rightarrow S$. 

(2) The timelike geodesics in $\RR^{2,1}=\g$ are parameterized by the tangent bundle $T\HH^2$: for $p\in \HH^2$ and $x \in T_p\HH^2$, the corresponding timelike geodesic is the affine line
$$\ell_x := \{Y\in\g \ |\ Y(p)=x\} ,$$
which is a translate of the infinitesimal stabilizer of~$p$. 
The vector field $X : \HH^2\rightarrow T\HH^2$ determines a natural collection $\{\ell_{X(p)} \,|\, p \in \HH^2\}$ of timelike geodesics.
This collection is a fibration of $\g$ over~$\HH^2$.
Indeed, for $Y \in \g$, the vector field $X - Y$ is $k$-lipschitz because $X$ is $k$-lipschitz and $Y$ is a Killing field; by Proposition~\ref{prop:fixpoint}, it has a unique zero, which we denote by $\varpi(Y) \in \HH^2$.
Thus $Y$ belongs to $\ell_{X(p)}$ for a unique $p\in\HH^2$, namely $\varpi(Y)$.
The surjective map $\varpi : \g\rightarrow \HH^2$ is continuous.
Indeed, if $Y' \in \g$ is close enough to $Y$ in the sense that $\| (Y - Y')(p) \| < |k|\,\delta$ for $p=\varpi(Y)$, then the $k$-lipschitz field $X-Y'=(X-Y)+(Y-Y')$ points inward along the sphere of radius $\delta$ centered at~$p$, hence the unique zero $\varpi(Y')$ of $X-Y'$ is within~$\delta$ of $p=\varpi(Y)$.
Finally, the map $\varpi$ satisfies the following equivariance property:
$$\varpi\big(j(\gamma) \cdot Y + u(\gamma)\big) = j(\gamma) \cdot \varpi(Y)$$
for all $\gamma\in\Gamma$ and $Y\in\g$.
Therefore, the properness of the action of $\Gamma^{j,u}$ on $\g=\RR^{2,1}$ follows from the properness of the action of $j(\Gamma)$ on~$\HH^2$, and $\varpi$ descends to a bundle projection $\varpi : \Gamma^{j,u}\backslash\RR^{2,1} \rightarrow S$.
\end{proof}

Note that in Proposition~\ref{prop:fibrations}.(2), replacing $X$ with a \emph{convex field} of negative lipschitz constant would still lead to a fibration of $\Gamma^{j,u}\backslash\RR^{2,1}$ (distinct vectors in $X(p)$ for the same $p\in\HH^2$ lead to parallel, but distinct lines of~$\RR^{2,1}$).
However, using a vector field has the advantage of making the leaf space canonically homeomorphic to~$S$.

In fact, in the $\RR^{2,1}$ case the quotient manifold is diffeomorphic to $S \times \RR$ as the fibers are oriented by the time direction.
Note however that the quotient manifold is \emph{not} globally hyperbolic (see \eg \cite{mes90} for a definition), for Charette--Drumm--Brill \cite{cdb03} have shown that Margulis spacetimes contain closed timelike curves (smooth, but not geodesic).
Nonetheless, the fibration $\mathcal{F}_X$ of Proposition~\ref{prop:fibrations}.(2) admits a spacelike section.
Indeed, Barbot \cite{bar05} has shown the existence of a convex, future-complete domain $\Omega_+$ in~$\RR^{2,1}$, invariant under~$\Gamma^{j,u}$, such that the quotient $\Gamma^{j,u}\backslash\Omega_+$ is globally hyperbolic and Cauchy-complete (this works in the general context where $j$ is convex cocompact and $u$ is any cocycle).
In particular, the quotient manifold contains a Cauchy surface whose lift to $\Omega_+$ is a spacelike, convex, complete, embedded disk.
This disk intersects all timelike geodesics in $\RR^{2,1}$ exactly once and so gives a spacelike section of~$\mathcal{F}_X$.

\section{Margulis spacetimes are limits of $\AdSS$ manifolds}\label{sec:geotrans}

In this section, we make precise the idea that all Margulis spacetimes should arise by zooming in on collapsing $\AdSS$ spacetimes.
As in the hyperbo\-lic-to-$\AdSS$ transition described in \cite{dan13}, one natural framework to describe this geometric transition is that of \emph{real projective geometry} (Section~\ref{subsec:projective-basics}); this approach eliminates Lorentzian metrics from the analysis.
The proof of Theorem~\ref{thm:geomtrans} is given in Section~\ref{subsec:proofoftransthm}; the main technical step (Section~\ref{subsec:collapse}) is to control the geometry of collapsing $\AdSS$ manifolds by producing geodesic fibrations as in Proposition~\ref{prop:fibrations} and controlling how they degenerate.
Finally, in Section~\ref{subsec:metrics} we prove Corollary~\ref{cor:smooth-metrics}, which describes the geometric transition in the language of Lorentzian metrics.

Background material on locally homogeneous geometric structures, particularly relevant for Section~\ref{subsec:proofoftransthm}, may be found in \cite{thu80, gol10}.

\subsection{$\AdS$ and $\RR^{2,1}$ as projective geometries}\label{subsec:projective-basics}

Both $\AdS$ and~$\RR^{2,1}$ can be realized as domains in projective space.
Indeed, the map
$$I: \begin{pmatrix} y_1+y_4 & y_2-y_3 \\ y_2+y_3 & -y_1+y_4 \end{pmatrix} \longmapsto [y_1:y_2:y_3:y_4]$$
defines an embedding of $\AdS = G = \PSL_2(\RR)$ into $\RP^3$ whose image is the open set $\{ [y]\in\RP^3 \,|\, y_1^2 + y_2^2 - y_3^2 - y_4^2 < 0\}$ (the interior of a projective quadric); it induces an injective group homomorphism $I_{\ast} :\ \mathrm{Isom}(\AdS)_0 = G\times G\ \hookrightarrow\ \PGL_4(\RR)$, and $I$ is $I_{\ast}$-equivariant:
\begin{equation}\label{eqn:I*}
I(A\cdot x) = I_{\ast}(A)\cdot I(x)
\end{equation}
for all $A\in\mathrm{Isom}(\AdS)_0$ and $x\in\AdS$.
The map 
$$i: \begin{pmatrix} z_1 & z_2-z_3 \\ z_2+z_3 & -z_1 \end{pmatrix} \longmapsto [z_1:z_2:z_3:1]$$
defines an embedding of $\RR^{2,1} = \g = \ssl_2(\RR)$ into $\RP^3$ whose image is the affine chart $\{ [z]\in\RP^3 \,|\, z_4 \neq 0\}$\,; it induces an injective group homomorphism $i_{\ast} :\ \mathrm{Isom}(\RR^{2,1})_0 = G\ltimes\g\ \hookrightarrow\ \PGL_4(\RR)$, and $i$ is $i_{\ast}$-equivariant:
\begin{equation}\label{eqn:i*}
i\big(B\cdot w) = i_{\ast}(B) \cdot i(w)
\end{equation}
for all $B\in\mathrm{Isom}(\RR^{2,1})_0$ and $w\in\RR^{2,1}$.
We see flat Lorentzian geometry as a limit of $\AdSS$ geometry by inflating an infinitesimal neighborhood of the identity matrix.
More precisely, we rescale by applying the family of projective transformations
\begin{equation}\label{eqn:rescale}
\small
r_t := \begin{pmatrix} t^{-1} & & & \\ & t^{-1} & & \\ & & t^{-1} & \\ & & & 1\end{pmatrix} \in \PGL_4(\RR).
\end{equation}
Note that $r_t\cdot I(\AdS)\subset r_{t'}\cdot I(\AdS)$ for $0<t'<t$ and that
$$\bigcup_{t>0}\ r_t \cdot I(\AdS) \,=\, i(\RR^{2,1}) \cup \HH^2_{\infty}\,,$$
where $\HH^2_{\infty}:=\{ [y]\in\RP^3\,|\, y_1^2 + y_2^2 - y_3^2 < 0 = y_4\}$ is a copy of the hyperbolic plane.
We observe that the limit as $t \to 0$ of the action of $r_t$ is differentiation:

\begin{Proposition}\label{prop:inflate} 
\begin{enumerate}
  \item For any smooth path $t\mapsto g_t\in G=\AdS$ with~$g_0=\nolinebreak 1$,
  $$r_t \cdot I(g_t) \,\xrightarrow[t\to 0]{}\, i\bigg(\frac{\D}{\D t}\Big|_{t=0}\, g_t\bigg) \in \RP^3.$$
  \item For any smooth path $t\mapsto (h_t,k_t)\in G\times G=\mathrm{Isom}(\AdS)_0$ with~$h_0=\nolinebreak k_0$,
  $$r_t\,I_{\ast}(h_t,k_t)\,r_t^{-1} \,\xrightarrow[t\to 0]{}\, i_{\ast}\bigg(h_0,\frac{\D}{\D t}\Big|_{t=0}\, h_t k_t^{-1}\bigg) \in \PGL_4(\RR).$$
\end{enumerate}
\end{Proposition}

\begin{proof}
Statement~(1) is an immediate consequence of the definitions.
Statement~(2) follows from~(1) by using the equivariance relations \eqref{eqn:I*} and \eqref{eqn:i*}: given $X\in\RR^{2,1}$, for small enough $t>0$ we can write $i(X)=r_t\cdot I(g_t)$ for some $g_t\in G$, and $g_t$ converges to~$1$; by~(1) we have $\frac{\D}{\D t}\big|_{t=0}\, g_t=X$ and
$$\big(r_t\,I_{\ast}(h_t,k_t)\,r_t^{-1}\big) \cdot i(X) = r_t \cdot I(k_t g_t h_t^{-1})$$
converges to
\begin{eqnarray*}
i\bigg(\frac{\D}{\D t}\Big|_{t=0}\, k_t g_t h_t^{-1}\bigg) & = & i\bigg(\frac{\D}{\D t}\Big|_{t=0}\, k_t g_t k_t^{-1} \, k_t h_t^{-1}\bigg)\\
& = & i\bigg(\Ad(k_0) \bigg(\frac{\D}{\D t}\Big|_{t=0}\, g_t\bigg) + \frac{\D}{\D t}\Big|_{t=0}\, k_t h_t^{-1}\bigg)\\
& = & i_{\ast}\bigg(k_0,\frac{\D}{\D t}\Big|_{t=0}\, k_t h_t^{-1}\bigg) \cdot i(X).
\end{eqnarray*}
We conclude using the fact that an element of $\PGL_4(\RR)$ is determined by its action on $i(\RR^{2,1})$.
\end{proof}


Now, let $\Gamma$ be a discrete group, $j\in\Hom(\Gamma,G)$ a convex cocompact representation, and $u : \Gamma\rightarrow\g$ a $j$-cocycle.
Proposition~\ref{prop:inflate} has the following immediate consequence.

\begin{Corollary}\label{cor:limitrep}
Let $t\mapsto j_t$ and $t\mapsto\rho_t$ be two smooth paths in $\Hom(\Gamma,G)$ with $j_0=\rho_0=j$ and $\frac{\D}{\D t}\big|_{t=0}\, \rho_t\, j_t^{-1} = u$.
For all $\gamma\in\Gamma$,
$$r_t\,I_{\ast}\big(j_t(\gamma),\rho_t(\gamma)\big)\,r_t^{-1} \ \xrightarrow[t\to 0]{}\ i_{\ast}\big(j(\gamma), u(\gamma)\big) \in \PGL_4(\RR).$$
\end{Corollary} 
 
Corollary~\ref{cor:limitrep} states that $r_t (I_*\Gamma^{j_t,\rho_t})r_t^{-1}$ converges to $i_* \Gamma^{j,u}$ as groups acting on $\RP^3$.
When $\Gamma^{j,u}$ and $\Gamma^{j_t,\rho_t}$ act properly discontinuously on $\RR^{2,1}$ and $\AdS$ respectively (which will be the case below), we thus obtain a path of proper actions on the subspaces $r_t\cdot I(\AdS)$ that converges (algebraically) to a proper action on $i(\RR^{2,1})$.
However, this is not enough to construct a geometric transition, for algebraic convergence \emph{may not} in general give a well-defined continuous path of quotient manifolds.
It is even possible that the topology of the quotient manifolds could change: see \cite{ac96} or \cite{can10} for examples of ``bumping" in the context of hyperbolic $3$-manifolds.
We thus need a more careful geometric investigation of the situation.

\subsection{Collapsing fibered $\AdSS$ manifolds}\label{subsec:collapse}

In this section, we prove two technical statements needed for Theorem~\ref{thm:geomtrans}.
The first gives the proper discontinuity of Theorem~\ref{thm:geomtrans}.(1) and, using Proposition~\ref{prop:fibrations}, produces geodesic fibrations that are well controlled as the quotient $\AdSS$ manifolds collapse.
The second statement gives sections of these fibrations with suitable behavior under the collapse.

Let $\Gamma$ be a discrete group, $j\in\Hom(\Gamma,G)$ a convex cocompact representation, and $u : \Gamma\rightarrow\g$ a proper deformation of~$j$ in the sense of Definition~\ref{def:proper-def}.
Fix two smooth paths $t\mapsto j_t\in\Hom(\Gamma,G)$ and $t\mapsto\rho_t\in\Hom(\Gamma,G)$ with $j_0=\rho_0=j$ and $\frac{\D}{\D t}\big|_{t=0}\,\rho_t\,j_t^{-1} = u$, as well as a smooth, $(j,u)$-equivariant vector field $X$ defined on~$\HH^2$, standard in the funnels (Definition~\ref{def:standard-funnels}), with $k := \lip (X) < 0$ (such an $X$ exists by Proposition~\ref{prop:smooth-approximates}).
The following two propositions are the main tools we will need for Theorem~\ref{thm:geomtrans}.

\begin{Proposition}\label{prop:Lipschitz-family}
There is a smooth family of $(j_t,\rho_t)$-equivariant diffeomorphisms $f_t : \HH^2\rightarrow\HH^2$, defined for small enough $t\geq 0$, such that
\begin{enumerate}
  \item $f_t$ is Lipschitz with $\Lip(f_t) = 1 + k t + O(t^2)$;
  \item $f_0 = \mathrm{Id}_{\HH^2}$;
  \item $\frac{\D}{\D t}\big|_{t=0}\, f_t=X$.
\end{enumerate}
\end{Proposition}

In particular, by Proposition~\ref{prop:fibrations}, the group $\Gamma^{j_t,\rho_t}$ acts properly discontinuously on $\AdS$ for all small enough $t>0$.

\begin{Proposition}\label{prop:sections}
Given a smooth family $(f_t)$ of diffeomorphisms as in Proposition~\ref{prop:Lipschitz-family}, there is a smooth family of smooth maps $\sigma_t : \HH^2 \rightarrow G$, defined for small enough $t\geq 0$, such that
\begin{enumerate}
  \item $\sigma_t$ is equivariant with respect to $j_t$ and $(j_t,\rho_t)$, meaning that for all $\gamma\in\Gamma$ and $p\in\HH^2$,
  $$\sigma_t(j_t(\gamma)\cdot p) = \rho_t(\gamma) \, \sigma_t(p) \, j_t(\gamma)^{-1} \, ;$$
  \item $\sigma_t$ is an embedding for all $t>0$, while $\sigma_0(\HH^2) = \{ 1\} \subset G$;
  \item $\sigma_t(p)\cdot p=f_t(p)$ for all $p\in\HH^2$; in other words, $\sigma_t$ defines a section of the fibration $\mathcal F_{f_t}$ of $\Gamma^{j_t,\rho_t}\backslash G$ from Proposition~\ref{prop:fibrations};
  \item the derivative $\sigma' := (\frac{\D}{\D t}\big|_{t=0}\, \sigma_t) : \HH^2 \rightarrow \g=\RR^{2,1}$ is an embedding, equivariant with respect to $j$ and $(j,u)$, meaning that for all $\gamma\in\Gamma$ and $p\in\HH^2$,
  $$\sigma'(j(\gamma)\cdot p) = j(\gamma)\cdot \sigma'(p) + u(\gamma);$$
  \item $\sigma'(p)(p) = X(p)$ for all $p\in\HH^2$; in other words, $\sigma'$ defines a section of the fibration $\mathcal F_X$ of $\Gamma^{j,u}\backslash\g$ from Proposition~\ref{prop:fibrations}.
\end{enumerate}
\end{Proposition}

We first prove Proposition~\ref{prop:sections}.
For a fixed~$t$, there is a lot of flexibility in defining a section $\sigma_t$ of~$\mathcal{F}_{f_t}$, since the bundle is trivial.
One natural construction with the desired first-order behavior is to take the \emph{osculating isometry map} to~$f_t$.

\begin{proof}[Proof of Proposition~\ref{prop:sections}]
Let $\sigma_t: \HH^2 \to G$ be the osculating isometry map to the diffeomorphism $f_t : \HH^2 \to \HH^2$: by definition, for any $p \in \HH^2$, the element $\sigma_t(p)\in G$ is the orientation-preserving isometry of~$\HH^2$ that coincides with $f_t$ at~$p$ and maps the (mutually orthogonal) principal directions of $f_t$ in $T_p \HH^2$ to their (mutually orthogonal) images in $T_{f_t(p)} \HH^2$.
(The differential map $\D_{p}f_t$ has two principal values close to~$1$; if they happen to be equal, then all pairs of mutually perpendicular directions in $T_{p}\HH^2$ yield the same definition.)
Then $\sigma_t$ is smooth for any~$t$, and varies smoothly with~$t$.
For $t = 0$, we obtain the constant map with image $1\in G$ (the isometry osculating the identity map is the identity).
The $(j_t,(j_t,\rho_t))$-equivariance of~$\sigma_t$ follows from the $(j_t,\rho_t)$-equivariance of~$f_t$ (using the fact that the osculating map is unique).

The derivative $\sigma' = (\frac{\D}{\D t}\big|_{t=0}\, \sigma_t) : \HH^2\rightarrow\g$ maps any $p\in\HH^2$ to the infinitesimal isometry osculating $X$ at~$p$ in the sense that the Killing field $\sigma'(p)$ agrees with $X$ at~$p$ and has the same curl.
(This means that the vector field $\sigma'(p)-\nolinebreak X$ vanishes at~$p$ and that its linearization $x\mapsto \nabla_x(\sigma'(p)-X)$ is a symmetric endomorphism of $T_p\HH^2$; note by contrast that the linearization of a Killing field near a zero is \emph{anti}symmetric.)
Finally, $\sigma'$ inherits~both~smooth\-ness and equivariance from~$X$ (using the uniqueness of the construction).
\end{proof}

We now turn to the proof of Proposition~\ref{prop:Lipschitz-family}.
Suppose $(f_t)$ is \emph{any} smooth family of smooth maps $\HH^2\rightarrow\HH^2$ with $f_0=\mathrm{Id}_{\HH^2}$ and $\frac{\D}{\D t}\big|_{t=0}\, f_t = X$.
For any $p\neq q$ in~$\HH^2$,
$$\frac{\D}{\D t}\Big|_{t=0}\, d\big(f_t(p),f_t(q)\big) = d'\big(X(p),X(q)\big),$$
hence 
$$\frac{d\big(f_t(p), f_t(q)\big)}{d(p,q)} = 1 + t \frac{d'_X}{d}(p,q) + O(t^2).$$
Since $(f_t)$ is a smooth family of smooth maps, the constant in the $O(t^2)$ can be taken uniform for $p,q$ in a compact set, including for $q \to p$, hence
$$\Lip(f_t) \leq 1 + k t + O(t^2)$$
over any compact set.
Thus we just need to find a smooth family $(f_t)$ for which the Lipschitz constant can be controlled uniformly above the \emph{funnels} of the surface $j(\Gamma)\backslash\HH^2$.
We now explain how this can be done.

\begin{proof}[Proof of Proposition~\ref{prop:Lipschitz-family}]
Let $(\smooth_t^0)$ and $(\smooth_t^1)$ be smooth families of diffeo\-morphisms of~$\HH^2$ with $\smooth_0^i=\mathrm{Id}_{\HH^2}$, such that $\smooth_t^0$ is $(j,j_t)$-equivariant and $\smooth_t^1$ is $(j,\rho_t)$-equivariant.
For $i\in\{ 0,1\}$, we assume that the family $(\smooth^i_t)$ is \emph{standard in the funnels}, by which we mean the following.
Let $\mathcal{E}\subset\HH^2$ be a connected component of the exterior of the convex core for~$j$, whose boundary is the translation axis $\A_{j(\gamma)}$ for some peripheral $\gamma\in\Gamma\smallsetminus\{ 1\}$.
In Fermi coordinates $F(\xi,\eta)$ (see Section~\ref{subsec:standard-funnels}) relative to $\A_{j(\gamma)}$, such that $\mathcal{E}=F(\RR\times\RR_-^{\ast})$, we ask that there exist a smooth family $(a^i_t)$ of isometries of~$\HH^2$ with $a^i_0=\mathrm{Id}_{\HH^2}$, and smooth families $(\kappa^i_t),(r^i_t)$ of reals with $\kappa^i_0=r^i_0=1$, such that
$$\smooth^i_t \big(F(\xi,\eta)\big) = a^i_t \circ F(\kappa^i_t\xi, r^i_t\eta)$$
for all $\eta<0$ smaller than some constant and all $\xi\in\RR$.
Note that, by equivariance, the isometry $a^0_t$ takes the axis $\A_{j(\gamma)}$ to $\A_{j_t(\gamma)}$, and $a^1_t$ takes $\A_{j(\gamma)}$ to $\A_{\rho_t(\gamma)}$; these axes vary smoothly.
We ask that this hold for any component $\mathcal{E}$ of the complement of the convex core.
Then the smooth, $(j,u)$-equivariant vector field $Y := \frac{\D}{\D t}\big|_{t=0}\, \smooth_t^1\circ (\smooth_t^0)^{-1}$ on~$\HH^2$ is standard in the funnels in the sense of Definition~\ref{def:standard-funnels}.
By choosing the first-order behavior of $a_t^i, \kappa_t^i, r_t^i$ appropriately in each component~$\mathcal{E}$, we can arrange for $Y$ to agree with $X$ outside some cocompact $j(\Gamma)$-invariant neighborhood $U'$ of the convex core (because $X$ itself is standard in the funnels).
 
We claim that the $(j_t,\rho_t)$-equivariant diffeomorphisms $\smooth_t := \smooth_t^1\circ (\smooth_t^0)^{-1}$ have the desired Lipschitz properties in the funnels.
Indeed, in each component $\mathcal{E}$ of the complement of the convex core, direct computation gives
$$\sup_{p\in\smooth^0_t(\mathcal{E}\smallsetminus U')}\, \| \D \smooth_t(p) \| \leq \max\left\{\frac{r^1_t}{r^0_t}, \frac{\kappa^1_t}{\kappa^0_t}\right\} \leq 1 + k t + O(t^2),$$
where the right-hand inequality follows from the fact that $\frac{\D}{\D t}\big|_{t=0}\, (r_t^1/r_t^0) \leq\nolinebreak k$ and $\frac{\D}{\D t}\big|_{t=0}\, (\kappa_t^1/\kappa_t^0) \leq k$ since $X$ is $k$-lipschitz.
Therefore, for each component $\mathcal{E}$ there is a constant $R>0$ such that
\begin{align}\label{eqn:Lip-bound-outside}
\frac{d(\smooth_t(p), \smooth_t(q))}{d(p,q)} \leq 1 + k t + R t^2
\end{align}
for all small enough $t>0$ and all $p \neq q$ in $\mathcal{E}\smallsetminus\smooth^0_t(U')$.
We can actually take the same $R$ for all components $\mathcal{E}$  by equivariance (there are only finitely many funnels in the quotient surface).

We now modify $(\smooth_t)$ into a smooth family of $(j_t,\rho_t)$-equivariant diffeomorphisms whose derivative at $t=0$ is $X$ instead of~$Y$, keeping the good Lipschitz properties in the funnels.
The vector field $Z := X - Y$ is smooth, $j$-invariant, and zero outside~$U'$.
Let $\phi^Z_t : \HH^2 \to \HH^2$ be the time-$t$ flow of~$Z$: it is a $j(\Gamma)$-invariant map, equal to the identity outside the $1$-neighborhood $U''$ of~$U'$ for small enough $t\geq 0$.
The map
$$f_t := \smooth^1_t \circ \phi^Z_t \circ (\smooth^0_t)^{-1}$$
is $(j_t,\rho_t)$-equivariant, coincides with $\smooth_t$ outside $\smooth^0_t(U'')$, and satisfies
$$\frac{\D}{\D t}\Big|_{t=0}\, f_t = Y + Z = X.$$
Let $\specialC$ be a compact neighborhood of a fundamental domain of~$U''$ for the action of $j(\Gamma)$.
Then $\specialC$ also contains fundamental domains of $\smooth^0_t(U'')$ for the action of $j_t(\Gamma)$ for small enough $t\geq 0$.
By smoothness and cocompactness, as explained before the proof, we have $\Lip(f_t)\leq 1+kt+O(t^2)$ on~$\specialC$, hence on $j_t(\Gamma)\cdot\specialC$ by subdivision.
Since $j_t(\Gamma)\cdot\specialC$ is a neighborhood of $\smooth^0_t(U'')$ for small~$t$, by subdivision and \eqref{eqn:Lip-bound-outside} we have $\Lip(f_t)\leq 1+kt+O(t^2)$~on~all~of~$\HH^2$.
\end{proof}

\subsection{Convergence of projective structures}\label{subsec:proofoftransthm}

In this section we prove Theorem~\ref{thm:geomtrans}, using Propositions \ref{prop:Lipschitz-family} and~\ref{prop:sections}.
We have already seen that (\ref{short-time-admissible}) for all $t>0$ small enough, $\Gamma^{j_t,\rho_t}$ acts properly on $\AdS$ (Propositions \ref{prop:fibrations} and~\ref{prop:Lipschitz-family}).
We now aim to prove that:
(\ref{parameterization}) there is a smooth family of $(j\times\mathrm{Id}_{\mathbb{S}^1},(j_t,\rho_t))$-equivariant diffeomorphisms (developing maps) $\Dev_t :\nolinebreak\HH^2\times\nolinebreak\mathbb S^1\to\AdS$ determining complete $\AdSS$ structures $\mathscr A_t$ on $S \times \mathbb S^1$;
(\ref{limit}) the $\RP^3$ structures $\mathscr P_t$ underlying~$\mathscr A_t$ have a well-defined limit~$\mathscr P_0$ as $t \to\nolinebreak 0$, and the Margulis spacetime $M$ is obtained by removing $S \times \{\pi\}$ from~$\mathscr P_0$.

Note that as $\AdSS$ structures, the $\mathscr A_t$ do \emph{not} converge as $t \to 0$.
In fact, we shall see in the proof that the central surface $S \times \{0\}$ collapses to a point (the developing maps $\Dev_t$ satisfy $\Dev_t(p,0)\rightarrow 1\in G$ for all $p\in\HH^2$~as~$t\to 0$).

Throughout the section, we set $\mathbb{S}^1=\RR/2\pi\ZZ$.
The normalization of the metric of $\AdS$ chosen in Section~\ref{subsec:defspaces} makes the length of any timelike geodesic circle $2\pi$.

\begin{proof}[Proof of Theorem~\ref{thm:geomtrans}]
Let $X$ be a smooth, $(j,u)$-equivariant vector field defined on~$\HH^2$, standard in the funnels, with $\lip(X)<0$, and let $(f_t)$ and $(\sigma_t)$ be as in Propositions \ref{prop:Lipschitz-family} and~\ref{prop:sections}.
By Proposition~\ref{prop:fibrations}, for any $t>0$ the map $f_t$ induces a principal fiber bundle structure
$$\mathbb{S}^1 \longhookrightarrow G \overset{\Pi_t}{\longrightarrow} \HH^2$$
with timelike geodesic fibers.
The maps $\Pi_t$ vary smoothly because the~$f_t$~do.
Using $(\sigma_t)$, which is a smooth family of smooth sections, and the $\mathbb{S}^1$-action on the fibers, we obtain a smooth family of global trivializations
$$\Phi_t : \HH^2 \times \mathbb{S}^1 \longrightarrow G.$$
Explicitly, $\Phi_t(p, \theta)$ is the point at distance $\theta\in\mathbb{S}^1$ from $\sigma_t(p)$ along the fiber $\Pi_t^{-1}(p)$ in the future direction:
$$\Phi_t(p,\theta) = \sigma_t(p) \operatorname{Rot}(p,\theta),$$
where $\operatorname{Rot}(p,\theta) \in G$ is the rotation of angle~$\theta$ centered at $p\in\HH^2$, and the product is for the group structure of~$G$.
By construction, $\Phi_t$ is equivariant with respect to $j_t\times\mathrm{Id}_{\mathbb{S}^1}$ and $(j_t,\rho_t)$; it is therefore a \emph{developing map} defining a complete $\AdSS$ structure on the manifold $j_t(\Gamma)\backslash \HH^2 \times \mathbb{S}^1$, which is diffeomorphic to $S \times \mathbb S^1$.
To prove~(\ref{parameterization}), we precompose the maps $\Phi_t$ with a smooth family of (lifts of) diffeomorphisms identifying $j_t(\Gamma) \backslash \HH^2 \times \mathbb S^1$ with $S \times \mathbb S^1$.
Recall the $(j,j_t)$-equivariant diffeomorphisms $\smooth^0_t: \HH^2 \to \HH^2$ from the proof of Proposition~\ref{prop:Lipschitz-family}.
For $t>0$, the map $\Dev_t : \HH^2 \times \mathbb S^1 \to G$ defined by
$$\Dev_t(p,\theta) := \Phi_t\big(\smooth^0_t(p), \theta\big)$$
is equivariant with respect to $j\times\mathrm{Id}_{\mathbb{S}^1}$ and $(j_t,\rho_t)$, hence is a developing map for a complete $\AdSS$ structure $\mathscr A_t$ on the fixed manifold $S \times \mathbb S^1$, as desired.
Note that only the smooth structure of~$S$ (and not the hyperbolic structure determined by~$j$) is important for this definition: indeed, if $\Sigma$ is a surface diffeomorphic to~$S$, then the maps $\smooth^0_t$ could be replaced by any smooth family of diffeomorphisms taking the action of $\pi_1(\Sigma)$ on $\widetilde{\Sigma}$ to the $j_t$-action of $\Gamma$ on~$\HH^2$.

We now prove~(\ref{limit}).
By Proposition~\ref{prop:fibrations}, the vector field~$X$ induces a principal fiber bundle structure
$$\RR \longhookrightarrow \g \overset{\varpi}{\longrightarrow} \HH^2$$
with timelike geodesic fibers.
The derivative $\sigma' := (\frac{\D}{\D t}\big|_{t=0}\, \sigma_t) : \HH^2\rightarrow\g$ is a smooth section; as above, we obtain a global trivialization
$$\dev : \HH^2 \times \RR \longrightarrow \g.$$
Explicitly, $\dev(p, \theta')$ is the point at signed distance $\theta'\in\RR$ in the future of $\sigma'(p)$ along the fiber $\varpi^{-1}(p)$:
$$\dev(p, \theta') = \sigma'(p) + \operatorname{rot}(p,\theta'),$$
where $\operatorname{rot}(p,\theta') \in \g$ is the infinitesimal rotation by amount~$\theta'$ around~$p$. 
By construction, $\dev$ is equivariant with respect to $j\times\mathrm{Id}_{\RR}$ and $(j,u)$; it is therefore a developing map for the Margulis spacetime $M=\Gamma^{j,u}\backslash\g$.

In order to obtain the convergence of projective structures as in~(3), we precompose the developing maps $\Dev_t$ and $\dev$ with diffeomorphims, changing coordinates in the fiber direction.
Let $\psi : \mathbb{S}^1\rightarrow\RP^1$ be any diffeomorphism with $\psi(\theta)\sim\theta$ for $\theta$ near~$0$ and $\psi(\pi)=\infty$, for instance $\psi(\theta)=2\tan(\theta/2)$.
For $t > 0$, let $\xi_t : \mathbb{S}^1\rightarrow\mathbb{S}^1$ be the diffeomorphism 
$$\xi_t(\theta) := \psi^{-1}\big(t\,\psi(\theta)\big)\,.$$ 
We precompose $\Dev_t$ by this change of coordinates in the $\mathbb{S}^1$ factor, yielding a new developing map $\hatDev_t : \HH^2\times\mathbb{S}^1\rightarrow G$ for the same $\AdSS$ structure:
\begin{equation}\label{eqn:def-hatDev}
\hatDev_t(p,\theta) := \Dev_t\big(p, \xi_t(\theta)\big).
\end{equation}
The map $\hatdev : \HH^2\times (-\pi,\pi)\rightarrow\g$ given by
$$\hatdev(p,\theta) := \dev\!\big(p, \psi(\theta)\big)$$
is a developing map for a complete flat Lorentzian structure on $S \times (-\pi,\pi)$ isometric to~$M$.
 
\begin{Claim}\label{claim:dev-maps}
For any $(p,\theta)\in\HH^2\times (-\pi,\pi)$,
$$\hatDev_t(p,\theta) \,\xrightarrow[t\to 0]{}\, 1\in G \quad\mathrm{and}\quad \frac{\D}{\D t}\Big|_{t=0}\, \hatDev_t(p,\theta) = \hatdev(p,\theta).$$
\end{Claim}

\begin{proof}
Fix $(p,\theta)\in\HH^2\times (-\pi,\pi)$.
Recall that $(\smooth^0_t)$ and $(\sigma_t)$ are smooth families, that $\smooth^0_0=\mathrm{Id}_{\HH^2}$, and that $\sigma_0 : \HH^2\rightarrow G$ is the constant map with image $1\in G$, whose differential is zero everywhere.
Moreover, $\xi_t(\theta)\to 0$ as $t\rightarrow 0$ since $\psi(0)=0$, hence
$$\hatDev_t(p, \theta) = \sigma_t\big(\smooth^0_t(p)\big) \operatorname{Rot}\!\big(\smooth^0_t(p),\xi_t(\theta)\big) \,\xrightarrow[t\to 0]{}\, 1\in G$$
and
$$\frac{\D }{\D t}\Big|_{t=0} \hatDev_t(p, \theta) = \sigma'(p) + \frac{\D }{\D t}\Big|_{t=0}\, \operatorname{Rot}\!\big(\smooth^0_t(p),\xi_t(\theta)\big).$$
Since $\psi(x)\sim x$ near~$0$, we have $\frac{\D}{\D t}\big|_{t=0}\, \xi_t (\theta)= \psi(\theta)$, hence
$$\frac{\D }{\D t}\Big|_{t=0}\, \operatorname{Rot}\!\big(\smooth^0_t(p),\xi_t(\theta)\big) = \operatorname{rot}\!\bigg(s^0_0(p),\frac{\D }{\D t}\Big|_{t=0}\, \xi_t (\theta)\bigg) = \operatorname{rot}\!\big(p,\psi(\theta)\big).$$
We conclude that $\frac{\D}{\D t}\big|_{t=0}\, \hatDev_t(p,\theta) = \dev(p, \psi(\theta)) = \hatdev(p,\theta)$.
\end{proof}

Note that the right-hand equality in Claim~\ref{claim:dev-maps} can also be written as
$$\frac{1}{t} \log\circ\,\hatDev_t \,\xrightarrow[t\to 0]{}\, \hatdev$$
on $\HH^2\times (-\pi,\pi)$, where $\log$ is the inverse of the exponential map, defined from a neighborhood of $1$ in~$G$ to a neighborhood of $0$ in~$\g$, and it follows from the proof that this convergence is uniform on compact sets.
Using Proposition~\ref{prop:inflate}, we conclude that on $\HH^2\times (-\pi,\pi)$, 
\begin{equation}\label{eqn:limit-formula}
r_t I \circ \hatDev_t \,\xrightarrow[t\to 0]{}\, i \circ \hatdev
\end{equation}
uniformly on compact sets.
This shows that, when restricted to $S \times (-\pi,\pi)$, the real projective structure $\mathscr P_t$ underlying the $\AdSS$ structures $\mathscr A_t$ converges to the real projective structure underlying the Margulis spacetime.
What remains to be shown is that the projective structures $\mathscr P_t$ on the full manifold $S \times \mathbb S^1$ converge.

The embedding $i : \g\hookrightarrow\RP^3$ of Section~\ref{subsec:projective-basics} extends to a diffeomorphism $\overline{i} : \overline{\g} \overset{\scriptscriptstyle\sim}{\rightarrow} \RP^3$, where
$$\overline \g := \mathbb P(\g \oplus \RR) = \g \cup \mathbb P(\g)$$
is the visual compactification of~$\g$.
The map $\hatdev : \HH^2\times (-\pi,\pi)\rightarrow\g$ extends to a map $\overline{\dev}: \HH^2 \times \mathbb S^1 \rightarrow \overline \g$, with
$$\overline{\dev}(p,\pi) = \left[\operatorname{rot}(p,1)\right] \in \mathbb P(\g) \subset \overline{\g}$$
for all $p\in\HH^2$. 
The restriction of $\overline{i}\circ\overline{\dev}$ to $\HH^2 \times \{\pi\}$ is a diffeomorphism onto the set
$$\HH^2_{\infty} := \big\{ [y]\in\RP^3\,|\, y_1^2 + y_2^2 - y_3^2 < 0 = y_4\big\} $$
of timelike directions of $i(\RR^{2,1})$, which is a copy of the hyperbolic plane.
The extended map $\overline{i}\circ\overline{\dev}$ is a diffeomorphism onto $i(\RR^{2,1})\cup\HH^2_{\infty}\subset\RP^3$.

Note that the action of $\Gamma^{j,u}$ on $i(\RR^{2,1})$ via~$i_{\ast}$ (see Section~\ref{subsec:projective-basics}) extends to an action of $\Gamma^{j,u}$ on $i(\RR^{2,1})\cup\HH^2_{\infty}$.
This action is properly discontinuous because $\overline{i}\circ\overline{\dev}$ is an equivariant diffeomorphism.
Thus,  $\overline{i}\circ\overline{\dev} : \HH^2\times\mathbb{S}^1\overset{\scriptscriptstyle\sim}{\rightarrow} i(\RR^{2,1})\cup\HH^2_{\infty}$ is a developing map identifying $S\times \mathbb S^1$ with the real projective manifold 
$$\overline{M} := \Gamma^{j,u} \backslash \big(i(\RR^{2,1}) \cup  \HH^2_{\infty}\big).$$
We conclude the proof of~(3) by extending Formula~\eqref{eqn:limit-formula} in this context.

\begin{Claim}
On $\HH^2 \times \mathbb{S}^1$, we have $r_t I \circ\hatDev_t \underset{\scriptscriptstyle t\to 0}{\longrightarrow} \overline i \circ \overline{\dev}$.
\end{Claim}

\begin{proof}
Given \eqref{eqn:limit-formula}, we can restrict to $\HH^2\times\{ \pi\}$.
Note (to be compared with Proposition~\ref{prop:inflate}.(1)) that for any smooth path $t\mapsto g_t\in G$ with $g_0\neq 1$,
\begin{equation}\label{eqn:limit-r_t-not1}
r_t \cdot I(g_t) \,\xrightarrow[t\to 0]{}\, \overline{i}\big( [\log(g_0)] \big) \in \HH^2_{\infty},
\end{equation}
where $\log(g_0)$ is any preimage of $g_0$ under the exponential map $\exp : \g\rightarrow G$ (the projective class $[\log(g_0)]$ does not depend on the choice of the preimage).
Indeed, \eqref{eqn:limit-r_t-not1} can be checked using the explicit coordinates of Section~\ref{subsec:projective-basics}: just note that
$$\overline{i} \left[ \log \begin{pmatrix} y_1 + y_4 & y_2 - y_3 \\ y_2 + y_3 & - y_1 + y_4 \end{pmatrix} \right] = \overline{i} \left[ \begin{pmatrix} y_1 & y_2 - y_3 \\ y_2 + y_3 & - y_1 \end{pmatrix} \right] = [y_1 : y_2 : y_3 : 0].$$
We then apply \eqref{eqn:limit-r_t-not1} to $g_t=\hatDev_t(p,\pi)$, which satisfies $g_0=\operatorname{Rot}(p,\pi)$ and $[\log(g_0)]=[\operatorname{rot}(p,\pi)]=\overline{\dev}(p,\pi)$.
\end{proof}

Thus the limit $\mathscr P_0$ of the projective structures $\mathscr P_t$ exists and is naturally identified with $\overline M$ (which might reasonably be called the \emph{timelike completion} of~$M$).
This completes the proof of Theorem~\ref{thm:geomtrans}.
\end{proof}

Here is a consequence of the proof.

\begin{Corollary}
The $\mathbb{S}^1$ fibers in Theorem~\ref{thm:geomtrans} can always be assumed to be timelike geodesics in the manifolds~$M_t$.
These geodesic fibers converge to timelike geodesic fibers in the limiting Margulis spacetime.
\end{Corollary}

\subsection{Convergence of metrics}\label{subsec:metrics}

Finally, we prove Corollary~\ref{cor:smooth-metrics} by showing that the $\AdSS$ metrics on $S \times \mathbb S^1$ determined by the developing maps $\hatDev_t$ that we constructed in Section~\ref{subsec:proofoftransthm} converge under appropriate rescaling to the complete flat metric on the limiting Margulis spacetime.

\begin{proof}[Proof of Corollary~\ref{cor:smooth-metrics}]
Recall from Section~\ref{subsec:defspaces} that we equip both $\RR^{2,1} =\nolinebreak\g$ and $\AdS = G$ with Lorentzian metrics induced by half the Killing form of~$\g$.
Let $\gmetric^{\Mink}$ and $\gmetric^\AdSS$ denote these metrics, of respective curvature $0$ and~$-1/4$.

Let $\hmetric^{\Mink}$ be the parallel flat metric on $V:=i(\RR^{2,1})\subset\RP^3$ obtained by pushing forward $\gmetric^{\Mink}$ by~$i$, and let $\hmetric^\AdSS$ be the constant-curvature metric on $I(\AdS)\subset\RP^3$ obtained by pushing forward $\gmetric^\AdSS$ by~$I$.
If we identify $V$ with the tangent space to $\RP^3$ at $x_0:=i(0)=I(1)$, then $i : \g\rightarrow V$ coincides with the differential of $I$ at $1\in G$; therefore, $\hmetric^{\Mink}_{x_0}=\hmetric^{\AdSS}_{x_0}$.

For $t>0$, consider the $\AdSS$ metric $\hmetric^t$ defined on $r_t\cdot I(\AdS)\subset\RP^3$ by
$$\hmetric^t := (r_t)_* \, \hmetric^{\AdSS},$$
where $(r_t)_*$ is the pushforward by~$r_t$.
Let us check that $t^{-2}\,\hmetric^t$ converges to $\hmetric^{\Mink}$ uniformly on compact subsets of~$V$ as $t \to 0$, where for a given compact set~$\mathcal{C}$ we only consider $t$ small enough so that $\hmetric^t$ is defined on~$\mathcal{C}$ (recall from Section~\ref{subsec:projective-basics} that the union of the sets $r_t\cdot I(\AdS)$ for $t>0$ is increasing and contains~$V$).
In what follows, we use the trivialization of $TV$ (which is preserved under affine transformations), denoting the parallel transport of a vector $v \in T_x V$ to $T_y V$ again by~$v$.
First, note that for any tangent vector $v \in T_x V$ we have $(r_t^{-1})_*v = t v \in T_{r_t^{-1}(x)}$.
Thus, for $v,w \in T_x V$,
\begin{align*}
t^{-2}\,\hmetric^t_x(v,w) &= t^{-2}\,\big((r_t)_* \hmetric^\AdSS\big)_x (v,w)\\
&= t^{-2}\ \hmetric^\AdSS_{r_t^{-1}(x)}\big((r_t^{-1})_*v, (r_t^{-1})_* w\big)\\
&= \hmetric^\AdSS_{r_t^{-1}(x)}(v,w).
\end{align*}
Given a compact set $\mathcal{C}\subset V$, the projective transformation $r_t^{-1}$ maps $\mathcal{C}$ into arbitrarily small neighborhoods of the basepoint~$x_0$ as $t \to 0$.
Therefore, by continuity of~$\hmetric^\AdSS$,
$$t^{-2}\,\hmetric^t_x = \hmetric^\AdSS_{r_t^{-1}(x)} \,\xrightarrow[t\to 0]{}\, \hmetric^\AdSS_{x_0} = \hmetric^{\Mink}_{x_0} = \hmetric^{\Mink}_x$$
uniformly for $x\in\mathcal{C}$ (where we use again the trivialization of $TV$).

Now, recall the developing maps $\hatDev_t : \HH^2 \times \mathbb{S}^1 \to \AdS$ defined for $t > 0$ by \eqref{eqn:def-hatDev}.
For $t > 0$, we consider the $\AdSS$ metric
$$\gmetric^t := (\hatDev_t)^* \gmetric^\AdSS$$
on $\HH^2\times\mathbb{S}^1$, where $(\hatDev_t)^*$ is the pullback by $\hatDev_t$.
The rescaled metrics $t^{-2}\gmetric^t$ determine complete metrics of curvature $-t^2/4$ on $S \times \mathbb S^1$.
We also consider the flat metric
$$\gmetric^0 := (\hatdev)^* \gmetric^{\Mink}$$
on $\HH^2 \times (-\pi,\pi)$.
By \eqref{eqn:limit-formula} and the convergence of $t^{-2}\hmetric^t$ proved above,
\begin{eqnarray*}
t^{-2} \gmetric^t & =\, & \big(r_t I \circ \hatDev_t\big)^* \,  t^{-2} \, \hmetric^t\\
& \xrightarrow[t\to 0]{} & (i \circ \hatdev)^* \hmetric^{\Mink} = (\hatdev)^* \gmetric^{\Mink} = \gmetric^0,
\end{eqnarray*}
on $\HH^2\times (-\pi,\pi)$, and the convergence is uniform on compact sets.
The induced rescaled metrics $t^{-2}\gmetric^t$ on $S \times (-\pi,\pi)$ converge to the induced flat metric $\gmetric^0$ on $S \times (-\pi,\pi)$.
The metric~$\gmetric^0$ makes $S \times (-\pi,\pi)$ isometric to~$M$ since $\hatdev$ is a developing map for~$M$.
\end{proof}

\vspace{0.5cm}

\vspace{0.5cm}

\end{document}